\RequirePackage{fix-cm}
\documentclass[11pt, a4paper, oneside, final, pagebackref]{amsart}
\usepackage[notref,notcite]{showkeys}
\usepackage[latin1]{inputenc}
\usepackage[T1]{fontenc}
\usepackage{amssymb}
\usepackage[stretch=10]{microtype}
\usepackage{mathtools}
\usepackage[DIV=11, headinclude=true]{typearea}
\usepackage[hidelinks, linktoc=all]{hyperref}
\usepackage{tikz}
\usetikzlibrary{arrows}

\providecommand{\href}[2]{#2}

\providecommand*{\backref}{}
\providecommand*{\backrefalt}{}
\renewcommand*{\backref}[1]{}
\renewcommand*{\backrefalt}[4]{%
	\ifcase #1 %
	\or
	  Cited page~#2.
	\else
	  Cited pages~#2.
	\fi
}

\newcommand\MTkillspecial[1]{
  \bgroup
  \catcode`\&=9
  \let\\\relax%
  \scantokens{#1}%
  \egroup
}
\newcommand\DeclarePairedDelimiterMultiline[3]{
  \DeclarePairedDelimiter{#1}{#2}{#3}
  \reDeclarePairedDelimiterInnerWrapper{#1}{star}{
    \mathopen{##1\vphantom{\MTkillspecial{##2}}\kern-\nulldelimiterspace\right.}
    ##2
    \mathclose{\left.\kern-\nulldelimiterspace\vphantom{\MTkillspecial{##2}}##3}}
}

\newcommand{\boH}{\mathcal{H}}
\newcommand{\boC}{\mathcal{C}}
\newcommand{\boI}{\mathcal{I}}
\newcommand{\boB}{\mathcal{B}}
\newcommand{\boD}{\mathcal{D}}
\newcommand{\boT}{\mathcal{T}}
\newcommand{\Z}{\mathbb{Z}}
\newcommand{\N}{\mathbb{N}}
\newcommand{\R}{\mathbb{R}}
\newcommand{\C}{\mathbb{C}}
\newcommand{\Sbb}{\mathbb{S}}
\newcommand{\dd}{\mathop{}\!\mathrm{d}}
\newcommand{\ic}{\mathbf{i}}
\newcommand{\bpi}{\boldsymbol{\pi}}
\DeclarePairedDelimiterMultiline{\abs}{\lvert}{\rvert}
\DeclarePairedDelimiterMultiline{\norm}{\lVert}{\rVert}
\DeclarePairedDelimiterMultiline{\pare}{(}{)}
\newcommand{\st}{\::\:}
\newcommand{\restr}{|}

\DeclareMathOperator{\ind}{ind}
\DeclareMathOperator{\Fix}{Fix}
\DeclareMathOperator{\sgn}{sgn}
\DeclareMathOperator{\Id}{Id}
\DeclareMathOperator{\Card}{Card}
\DeclareMathOperator{\Leb}{Leb}
\DeclareMathOperator{\dLeb}{dLeb}

\newcommand{\coloneqq}{\mathrel{\mathop:}=}
\DeclareMathOperator{\flattr}{\mathrm{tr}^{\flat}}
\DeclareMathOperator{\tr}{tr}

\renewcommand{\epsilon}{\varepsilon}

\renewcommand{\phi}{\varphi}
\renewcommand{\leq}{\leqslant}
\renewcommand{\geq}{\geqslant}

\newtheorem{thm}{Theorem}[section]
\newtheorem*{thm*}{Theorem}
\newtheorem{prop}[thm]{Proposition}
\newtheorem{definition}[thm]{Definition}
\newtheorem{lem}[thm]{Lemma}
\newtheorem{cor}[thm]{Corollary}
\newtheorem*{prop*}{Proposition}

\theoremstyle{definition}

\newtheorem{rmk}[thm]{Remark}

\numberwithin{equation}{section}

\title{Ruelle spectrum of linear pseudo-Anosov maps}
\author{Fr\'ed\'eric Faure, S\'ebastien Gou\"ezel and Erwan Lanneau}

\address{Univ.\ Grenoble Alpes, CNRS UMR 5582, Institut Fourier, F-38000 Grenoble, France}
\email{frederic.faure@univ-grenoble-alpes.fr}

\address{Laboratoire Jean Leray, CNRS UMR 6629,
Universit\'e de Nantes, 2 rue de la
Houssini\`ere,
44322 Nantes, France}
\email{sebastien.gouezel@univ-nantes.fr}

\address{Univ.\ Grenoble Alpes, CNRS UMR 5582, Institut Fourier, F-38000 Grenoble, France}
\email{erwan.lanneau@univ-grenoble-alpes.fr}

\date{\today}
\thanks{We thank Corinna Ulcigrai, Mauro Artigiani and Giovanni Forni for their enlightening comments.}

\begin{document}

\begin{abstract}
The Ruelle resonances of a dynamical system are spectral data describing
the precise asymptotics of correlations. We classify them completely for a
class of chaotic two-dimensional maps, the linear pseudo-Anosov maps, in
terms of the action of the map on cohomology. As applications, we obtain a
full description of the distributions which are invariant under the linear
flow in the stable direction of such a linear pseudo-Anosov map, and we
solve the cohomological equation for this flow.
\end{abstract}

\maketitle

\tableofcontents

\section{Introduction, statements of results}

\subsection*{Ruelle resonances} Consider a map $T$ on a smooth manifold
$X$, preserving a probability measure $\mu$. One feature that encapsulates a
lot of information on its probabilistic behavior is the speed of decay of
correlations. Consider two smooth functions $f$ and $g$. Then one expects
that $\int f \cdot g \circ T^n \dd \mu$ converges to $(\int f \dd\mu) \cdot
(\int g \dd\mu)$ if iterating the dynamics creates more and more independence
-- if this is the case, $T$ is said to be mixing for the measure $\mu$.
Often, one can say more than just the mere convergence to $0$ of the
correlations $\int f \cdot g \circ T^n \dd \mu - (\int f \dd\mu) \cdot (\int
g \dd\mu)$, and this is important for applications. For instance, the central
limit theorem for the Birkhoff sums $S_n f = \sum_{k=0}^{n-1} f \circ T^k$ of
a function $f$ with $0$ average often follows from the summability of the
correlations between $f$ and $f\circ T^n$.

When $T$ is very chaotic, the correlations tend exponentially fast to $0$. It
is sometimes possible to obtain the next few terms in their asymptotic
expansion, in terms of the~\emph{Ruelle spectrum} (or~\emph{Ruelle
resonances}) of the map.

\begin{definition}
\label{defn:Ruelle_spectrum} Let $T$ be a map on a space $X$, preserving a
probability measure $\mu$. Consider a space of bounded functions $\boC$ on
$X$. Let $I$ be a finite or countable set, let $\Lambda = (\lambda_i)_{i \in
I}$ be a set of complex numbers with $\abs{\lambda_i} \in (0,1]$ such that
for any $\epsilon>0$ there are only finitely many $i$ with $\abs{\lambda_i}
\geq \epsilon$, and let $(N_i)_{i\in I}$ be nonnegative integers. We say that
$T$ has the Ruelle spectrum $(\lambda_i)_{i \in I}$ with Jordan blocks
dimension $(N_i)_{i\in I}$ on the space of functions $\boC$ if, for any $f, g
\in \boC$ and for any $\epsilon>0$, there is an asymptotic expansion
\begin{equation*}
  \int f \cdot g \circ T^n \dd\mu = \sum_{\abs{\lambda_i} \geq \epsilon} \sum_{j\leq N_i} \lambda_i^n n^j c_{i,j}(f,g)
  + o(\epsilon^n),
\end{equation*}
where $c_{i,j}(f,g)$ are bilinear functions of  $f$ and $g$, that we suppose finite rank but non zero.
\end{definition}

In other words, there is an asymptotic expansion for the correlations of
functions in $\boC$, up to an arbitrarily small exponential error. With this
definition, it is clear that the Ruelle spectrum is an intrinsic object, only
depending on $T$, $\mu$ and the space of functions $\boC$. In general, one
takes for $\boC$ the space of $C^\infty$ functions on a manifold.

As an example, assume that $T$ is a $C^\infty$ uniformly expanding map on a
manifold and $\mu$ is its unique invariant probability measure in the
Lebesgue measure class. Then the correlations of $C^r$ functions admit an
asymptotic expansion up to an exponential term $\epsilon_r^n$, where
$\epsilon_r$ tends to $0$ when $r$ tends to infinity. Hence,
Definition~\ref{defn:Ruelle_spectrum} is not satisfied for $\boC = C^r$, but
it is satisfied for $\boC = C^\infty(M)$. The same holds for Anosov maps,
when $\mu$ is a Gibbs measure.

\medskip

The first question one may ask is if it makes sense to talk about the Ruelle
spectrum, i.e., if Definition~\ref{defn:Ruelle_spectrum} holds for some
$\Lambda=(\lambda_i)_{i\in I}$. Virtually all proofs of such an abstract
existence result follow from spectral considerations, exhibiting the
$\lambda_i$ as the spectrum of an operator associated to $T$, acting on a
Banach space or a scale of Banach spaces. General spectral theorems taking
advantage of compactness or quasi-compactness properties of this operator
then imply that there is some set $\Lambda$ for which
Definition~\ref{defn:Ruelle_spectrum} holds (and moreover all elements of
$\Lambda$ have finite multiplicity), but without giving any information
whatsoever on $\Lambda$ in addition to the fact that it is discrete and at
most countable -- in particular, it is not guaranteed that $\Lambda$ is not
reduced to the eigenvalue $1$, which is always a Ruelle resonance as one can
see by taking $f=g=1$. Indeed, if $T$ is the doubling map $x \mapsto 2 x \mod
1$ on the circle and $\boC=C^\infty(\Sbb^1)$, then there is no other
resonance. In the same way, there is no other resonance for linear Anosov map
of the torus (these facts are easy to check by computing the correlations
using Fourier series). That Definition~\ref{defn:Ruelle_spectrum} holds is
notably known for uniformly expanding and uniformly hyperbolic smooth maps,
see~\cite{ruelle_fredholm, bt_aniso,GL_Anosov2}.

Once the answer to this first question is positive, there is a whole range of
questions one may ask about $\Lambda$: is it reduced to $\{1\}$? is it
infinite? are there asymptotics for $\Card (\Lambda \cap \{\abs{z} \geq
\epsilon\})$ (possibly counted with multiplicities) when $\epsilon$ tends to
$0$? is it possible to describe explicitly $\Lambda$? The answers to these
questions depend on the map under consideration. Let us only mention the
results of Naud~\cite{naud_lower_bound} (for generic analytic expanding maps,
there is nontrivial Ruelle spectrum, with density at $0$ bounded below
explicitly), Adam~\cite{adam_hyperbolic} (the spectrum is generically
non-empty for hyperbolic maps),
Bandtlow-Jenkinson~\cite{bandtlow_jenkinson_upper} (upper bound for the
density of Ruelle resonances at $0$ in analytic expanding maps, extending
previous results of Fried),
Bandtlow-Just-Slipantschuk~\cite{bandtlow_just_slipantschuk_expanding,
bandtlow_just_slipantschuk_hyperbolic} (construction of expanding or
hyperbolic maps for which the Ruelle spectrum is completely explicit),
Dyatlov-Faure-Guillarmou~\cite{dyatlov_faure_guillarmou} (classification of
the Ruelle resonances for the geodesic flow on compact hyperbolic manifolds
in any dimension).

Our goal in this article is to investigate these questions for a class of
maps of geometric origin, namely linear pseudo-Anosov maps. They are analogues of
linear Anosov maps of the two-dimensional torus, but on higher genus
surfaces. The difference with the torus case is that the expanding and
contracting foliations have singularities. Apart from these singularities,
the local picture is exactly the same as for linear Anosov maps of the torus
(in particular, it is the same everywhere in the manifold). We will obtain a
complete description of the Ruelle spectrum of linear pseudo-Anosov map.
Then, using the philosophy of Giulietti-Liverani~\cite{giulietti_liverani}
that Ruelle resonances contain information on the translation flow along the
stable manifold on the map, we will discuss consequences of these results on
the vertical translation flow in translation surfaces supporting a
pseudo-Anosov map. We will in particular obtain complete results on the set
of distributions which are invariant under the vertical flow, and on smooth
solutions to the cohomological equation, recovering in this case results due
to Forni on generic translation surfaces~\cite{forni_cohomological,
forni_deviation, forni_regularity}.

\subsection*{Linear pseudo-Anosov maps}
There are several equivalent definitions of pseudo-Anosov maps (especially in
terms of foliations carrying a transverse measure). We will use the following
one in which the foliations have already been straightened (i.e., we use
coordinates where the foliations are horizontal and vertical), in terms of
half-translation surfaces (see e.g.~\cite{zorich} for a nice survey on half-translation surfaces).

\begin{definition}
\label{def:half_translation} Let $M$ be a compact connected surface and let
$\Sigma$ be a finite subset of $M$. A half-translation structure on
$(M,\Sigma)$ is an atlas on $M-\Sigma$ for which the coordinate changes have
the form $x \mapsto x+v$ or $x \mapsto -x+v$. Moreover, we require that
around each point of $\Sigma$ the half translation surface is isomorphic to a
finite ramified cover of $\R^2/{\pm \Id}$ around $0$.
\end{definition}
A half-translation surface carries a canonical complex structure: it is just
the canonical complex structure in the charts away from $\Sigma$, which
extends to the singularities. In particular, it also has a $C^\infty$
structure, and it is orientable.

In a half-translation structure, the horizontal and vertical lines in the
charts define two foliations of $M-\Sigma$, called the horizontal and
vertical foliations. Of particular importance to us will be the case where
the coordinate changes are of the form $x \mapsto x + v$. In this case, we
say that $M$ is a translation surface. Singularities are then finite ramified
cover of $\R^2$ around $0$. Moreover, the horizontal and vertical foliations
carry a canonical orientation.

\begin{definition}
Consider a half-translation structure on $(M,\Sigma)$. A homeomorphism
$T:M\to M$ is a linear pseudo-Anosov map for this structure if
$T(\Sigma)=\Sigma$ and there exists $\lambda>1$ such that, for any $x \in
M-\Sigma$, one has in half-translation charts around $x$ and $Tx$ the
equality $T y = \left(\begin{smallmatrix} \pm \lambda & 0 \\ 0 & \pm
\lambda^{-1}
\end{smallmatrix}\right)y$, where the choice of signs depends on the
choice of coordinate charts. We say that $\lambda$ is the~\emph{expansion
factor} of $T$.
\end{definition}
In other words, $T$ sends horizontal segments to horizontal segments and
vertical segments to vertical segments, expanding by $\lambda$ in the
horizontal direction and contracting by $\lambda$ in the vertical direction.
In particular, Lebesgue measure is invariant under $T$.

When $M$ is a translation surface, there are two global signs $\epsilon_h$
and $\epsilon_v$ saying if $T$ preserves or reverses the orientation of the
horizontal and vertical foliations. The simplest case is when $\epsilon_h =
\epsilon_v = 1$. In this case, $T$ preserves the orientation of both
foliations, and can be written in local charts as $\left(\begin{smallmatrix}
\lambda & 0 \\ 0 & \lambda^{-1}
\end{smallmatrix}\right)$.

While we obtain a complete description of the Ruelle spectrum in all
situations (orientable foliations or not, $\epsilon_v$ and $\epsilon_h$ equal
to $1$ or $-1$), it is easier to explain in the simplest case of translation
surfaces with $\epsilon_v = \epsilon_h = 1$. We will refer to this case as
linear pseudo-Anosov maps preserving orientations. We will focus on this case
in this introduction and most of the paper, and refer to
Section~\ref{sec:orientations} for the general situation (that we will deduce from the
case of linear pseudo-Anosov maps preserving orientations).

In the definition of Ruelle resonances, there is a subtlety related to the
choice of the space of functions $\boC$ for which we want asymptotic
expansions of the correlations. While it is clear that we want $C^\infty$
functions away from the singularities, the requirements at the singularities
are less obvious. Denote by $C^\infty_c(M-\Sigma)$ the space of $C^\infty$
functions that vanish on a neighborhood of the singularities. This is the
space we will use for definiteness.

Let $T$ be a linear pseudo-Anosov map, preserving orientations, on a genus
$g$ translation surface $M$. Let $\lambda$ be its expansion factor. As the
local picture for $T$ is the same everywhere, it should not be surprising
that the only data influencing the Ruelle spectrum are of global nature,
related to the action of $T$ on the first cohomology group $H^1(M)$ (a vector
space of dimension $2g$). By Thurston~\cite{thurston_pseudo_anosov},
$\lambda$ and $\lambda^{-1}$ are two simple eigenvalues of $T^* : H^1(M)\to
H^1(M)$ (the corresponding eigenvectors are the cohomology classes of the
horizontal and the vertical foliations). The orthogonal subspace to these two
cohomology classes has dimension $2g-2$, it is invariant under $T^*$, and the
spectrum $\Xi = \{\mu_1,\dotsc,\mu_{2g-2}\}$ of $T^*$ on this subspace is
made of eigenvalues satisfying $\lambda^{-1} < \abs{\mu_i} < \lambda$ for all
$i$.

\medskip

Here is our main theorem when $T$ preserves orientations.

\begin{thm}
\label{thm:main_preserves_orientations} Let $T$ be a linear pseudo-Anosov map
preserving orientations on a genus $g$ compact surface $M$, with expansion
factor $\lambda$ and singularity set $\Sigma$. Then $T$ has a Ruelle spectrum
on $\boC = C^\infty_c(M-\Sigma)$ given as follows. First, there is a simple
eigenvalue at $1$. Denote by $\Xi = \{\mu_1,\dotsc, \mu_{2g-2}\}$ the
spectrum of $T^*$ on the orthogonal subspace to the classes of the horizontal
and vertical foliations in $H^1(M)$. Then, for any $i$ and for any integer
$n\geq 1$, there is a Ruelle resonance at $\lambda^{-n} \mu_i$ of
multiplicity $n$.
\end{thm}
Note that a complex number $z$ may sometimes be written in different ways as
$\lambda^{-n}\mu_i$ (for instance if the spectrum of $T^*$ is not simple,
i.e., if there is $i\neq j$ with $\mu_i = \mu_j$ -- but it can also happen
that there is $i\neq j$ with $\mu_i = \lambda^{-1}\mu_j$, which will lead to
more superpositions). In this case, to get the multiplicity of $z$, one
should add all the multiplicities from the theorem corresponding to the
different possible decompositions.

Let us note that some nonzero functions can be orthogonal to all Ruelle resonances. For instance,
if $T$ lifts a linear Anosov map of the torus to a higher genus surface covering
the torus, then the correlations of any two smooth functions lifted from the torus
tend to $0$ faster than any exponential, as this is the case in the torus.

\subsection*{A quick sketch of the proof}
Before we discuss further results, we should explain briefly the strategy to
prove Theorem~\ref{thm:main_preserves_orientations}. First, we want to show
that Ruelle resonances make sense as in
Definition~\ref{defn:Ruelle_spectrum}. This part is classical. We introduce a
scale of Banach spaces of distributions, denoted by $\boB^{-k_h, k_v}$, which
behaves well under the composition operator $\boT : f \mapsto f \circ T$. The
elements of $\boB^{-k_h, k_v}$ are objects that can be integrated along
horizontal segments against $C^{k_h}$-functions, and moreover have $k_v$
vertical derivatives: this is an anisotropic Banach space, taking advantage
of the contraction of $T$ in the vertical direction and of its expansion in
the horizontal direction, as is customary in the study of hyperbolic
dynamics. On the technical level, the definition of $\boB^{-k_h, k_v}$ is
less involved than in many articles on hyperbolic dynamics (see for
instance~\cite{GL_Anosov2, bt_aniso}), as we may take advantage of the fact
that the stable and unstable directions are smooth -- in this respect, it is
closer to~\cite{baladi_Cinfty, avila_gouezel}. The only additional difficulty
compared to the literature is the singularities, but it turns out that they
do not play any role in this part. Hence, we can prove that the essential
spectral radius of $\boT$ on $\boB^{-k_h, k_v}$ is at most
$\lambda^{-\min(k_h, k_v)}$. The existence of Ruelle resonances in the sense
of Definition~\ref{defn:Ruelle_spectrum} readily follows. One important point
we want to stress here is that, since we are interested in Ruelle resonances
for functions in $C^\infty_c(M-\Sigma)$, we take for $\boB^{-k_h, k_v}$ the
closure of $C^\infty_c(M-\Sigma)$ for an anisotropic norm as described above.
In particular, smooth functions are dense in $\boB^{-k_h, k_v}$.

The second step in the proof is to show that the elements described in
Theorem~\ref{thm:main_preserves_orientations} belong to the set of Ruelle
resonances or, equivalently, to the spectrum of $\boT$ on $\boB^{-k_h, k_v}$
when $k_h$ and $k_v$ are large enough. It is rather easy to show that $1$ and
$\lambda^{-1} \mu_i$ belong to the spectrum, by considering a smooth $1$-form
$\omega = \omega_x \dd x + \omega_y \dd y$ whose cohomology class is an
eigenfunction for the iteration of $T^*$, and looking at the asymptotics of
$\boT^n \omega_x$ to obtain an element $f\in \boB^{-k_h, k_v}$ with $\boT f =
\lambda^{-1} \mu_i f$. Then, one deduces that $\lambda^{-n} \mu_i$ also
belongs to the spectrum, as $L_h^{n-1} f$ is an eigenfunction for this
eigenvalue, where $L_h$ denotes the derivative in the horizontal direction.

The most interesting part of the proof is to show that there is no other
eigenvalue, and that the multiplicities are as stated in the theorem. For
this, start from an eigenfunction $f \in \boB^{-k_h, k_v}$ for an eigenvalue
$\rho$. Denote by $L_v$ the derivative in the vertical direction. Then $L_v^n
f$ is an eigenfunction for the eigenvalue $\lambda^n \rho$. Since all
eigenvalues have modulus at most $1$, we deduce that $L_v^n f = 0$ for large
enough $n$. Consider the last index $n$ where $L_v^n f \neq 0$, and write $g
= L_v^n f$. It is an eigenfunction, and $L_v g = 0$. If we can prove that the
corresponding eigenvalue has the form $\lambda^{-k} \mu_i$ for some $k$ and
$i$, then we get $\rho = \lambda^{-(n+k)} \mu_i$, as desired. To summarize,
it is enough to understand eigenfunctions that, additionally, satisfy $L_v g
= 0$. For this, we introduce a cohomological interpretation of elements of
$\boB^{-k_h, k_v} \cap \ker L_v$. Heuristically, elements of $\boB^{-k_h,
k_v}$ can be integrated along horizontal segments by definition, so what
really matters is not the distribution $g$, but the $1$-current $g\dd x$. (In
the language of Forni~\cite{forni_deviation}, elements $g$ of $\boB^{-k_h,
k_v} \cap \ker L_v$ are the vertically invariant distributions, see his
Definition 6.4, while $g\dd x$ is the corresponding basic current on $M$.)
Formally, its differential is
\begin{equation*}
  \dd(g\dd x) =
  (\partial_x g \dd x + \partial_y g \dd y) \wedge \dd x
  = - L_v g \dd x \wedge \dd y.
\end{equation*}
Hence, elements of $\boB^{-k_h, k_v} \cap \ker L_v$ give rise to closed
currents, and have an associated cohomology class in $H^1(M)$ by de Rham
Theorem (in fact, we do not use de Rham theorem directly, but a custom
version suited for our needs that deals more carefully with the
singularities). From the equality $\boT g = \rho_g g$ one deduces that this
class is an eigenfunction for $T^*$ acting on $H^1(M)$, for the eigenvalue
$\lambda \rho_g$. If the class is nonzero, we get that $\lambda \rho_g$ is
one of the $\mu_i$, and $\rho_g = \lambda^{-1} \mu_i$ as desired. If the
class is zero, this means that $g \dd x$ is itself the differential of a
$0$-current $\tilde g$. It turns out that $\tilde g$ belongs to our scale of
Banach spaces, and is an eigenfunction for the eigenvalue $\lambda \rho_g$.
One can then argue in this way by induction to show that all eigenvalues are
of the form claimed in Theorem~\ref{thm:main_preserves_orientations}. There
are additional difficulties related to the eigenvalue $\lambda^{-1}$ of $T^*
: H^1(M)\to H^1(M)$: it does not show up in the statement of
Theorem~\ref{thm:main_preserves_orientations}, but this does not follow from
the sketch we have just given. Moreover, getting the precise multiplicities
requires further arguments, based on duality arguments and beyond this
introduction.

Here is the precise description we get in the end, illustrated on Figure~\ref{fig:dessin}, assuming to simplify that
$\mu_i$ is simple for $T^* : H^1(M) \to H^1(M)$ and that $\lambda^{-1}\mu_i$
is not an eigenvalue of $T^*$. Then the eigenvalue $\lambda^{-1} \mu_i$ for
$\boT$ is simple, and realized by a distribution $f_i$ which is annihilated
by $L_v$ (i.e., it is invariant under vertical translation) and such that the
cohomology class of $f_i \dd x$ is the eigenfunction in $H^1(M)$ under $T^*$,
for the eigenvalue $\mu_i$. Denoting by $E_\alpha$ the generalized eigenspace
associated to the eigenvalue $\alpha$, then $L_v$ is onto from
$E_{\lambda^{-n-1} \mu_i}$ to $E_{\lambda^{-n} \mu_i}$, and its kernel is
one-dimensional, equal to $L_h^n E_{\lambda^{-1}\mu_i}$. Therefore, there is
a flag decomposition
\begin{equation}
\label{eq:flag_decomposition}
  \{0\} \subset L_h^n E_{\lambda^{-1}\mu_i} \subset L_h^{n-1} E_{\lambda^{-2}\mu_i}
  \subset \dotsb \subset L_h^2 E_{\lambda^{-n+1} \mu_i} \subset L_h E_{\lambda^{-n} \mu_i}
  \subset E_{\lambda^{-n-1}\mu_i},
\end{equation}
in which the $k$-th term $L_h^{n+1-k} E_{\lambda^{-k} \mu_i}$ has dimension
$k$, and is equal to $E_{\lambda^{-n-1}\mu_i} \cap \ker L_v^k$. This
decomposition shows that the elements of $E_{\lambda^{-n-1}\mu_i}$ behave
like polynomials of degree $n$ when one moves along the vertical direction.
Moreover, the decomposition~\eqref{eq:flag_decomposition} is invariant under
the transfer operator $\boT$, which is thus in upper triangular form with
$\lambda^{-n-1}\mu_i$ on the diagonal. We do not know if there are genuine
Jordan blocks, or a choice of basis for which $\boT$ is diagonal. In
particular, we do not identify in
Theorem~\ref{thm:main_preserves_orientations} the Jordan blocks dimension of
the Ruelle resonances, in the sense of Definition~\ref{defn:Ruelle_spectrum}.
The decomposition~\eqref{eq:flag_decomposition} can also be interpreted in
terms of the operator $N = L_h L_v$, which is nilpotent of order $n+1$ on the
$n+1$-dimension space $E_{\lambda^{-n-1} \mu_i}$: the $k$-th term is the
kernel of $N^k$, and also the image of $N^{n+1-k}$.
\begin{figure}[htb]
\centering
\begin{tikzpicture}[scale=0.35]
	\def\xmin{-1} \def\xmax{15} \def\ymin{-1} \def\ymax{12}
	\tikzstyle{axe}=[->,>=stealth']
	\tikzstyle{grille}= [step=3,dashed]
	\draw [grille] (\xmin,\ymin) grid (\xmax+2,\ymax+2);
	\draw [axe] (\xmin,0)--(\xmax+3,0) node[below right]  {$k_x$};
	\draw [axe] (0,\ymin)--(0,\ymax+3) node[above left] {$k_y$};
	\node[below] at (0,-0.5) {$\scriptstyle 0$}; \node[below] at (3,-0.5) {$\scriptstyle 1$};
	\node[left] at (-0.5,0) {$\scriptstyle 0$}; \node[left] at (-0.5,3) {$\scriptstyle 1$};
	\draw[blue,thick] (3,0) -- (16,0);
	\foreach \x in {1,...,4}
	{\draw[red,thick] (3,3*\x) -- (3*\x+3,0);}
	\foreach \x in {1,...,4}
	\foreach \y in {0,...,3}
	{\filldraw (3*\x,3*\y) circle (10pt);}
	\filldraw (3*5,0) circle (10pt); \filldraw (3,3*4) circle (10pt);
	\draw [thick,->] (-2,9) -- (-2,6) node[above left] {$L_v = \partial_y$};
	\draw [thick,->] (6,-2) -- (9,-2) node[below] {$L_h = \partial_x$};
	\draw [->] (9,18) node[above,right] {Eigenspace $\scriptstyle E_{\lambda^{-n}\mu_i}, \ n=k_x+k_y$} .. controls +(left:55mm) and +(up:25mm) .. (4,11.5) ;
	\draw [->] (18,4) node[above,right] {ker $L_v$} .. controls +(left:55mm) and +(up:25mm) .. (13,0.2) ;
	\end{tikzpicture}
\caption{For a given eigenvalue $\mu_i$ of $T^*$ ($\mu_i \in (\lambda^{-1},\lambda)$), each black point of the lattice  $(k_x,k_y)_{k_x \geq 1, k_y\geq 0}$  represents an independent Ruelle distribution $u_{(k_x,k_y)}$. In particular $f_i=u_{(1,0)}$. The eigenvalues of the transfer operator $\mathcal{T}$ are $\lambda^{-n} \mu_i$ with $n\geq 1$  and the associated eigenspace is $E_{\lambda^{-n} \mu_i} = \mathrm{Span}\left\{u_{k_x,k_y}, k_x+k_y=n \right\}$ with dimension $n$ and represented by a diagonal red line. The operator $L_h \equiv \partial_x$ maps $u_{(k_x,k_y)}$ to $u_{(k_x+1,k_y)}$ and $L_v \equiv \partial_y$ maps $u_{(k_x,k_y)}$ to $u_{(k_x,k_y-1)}$. In particular the space $\mathrm{ker} L_v$ is represented by the first horizontal blue line $k_y=0$.}
 \label{fig:dessin}
\end{figure}
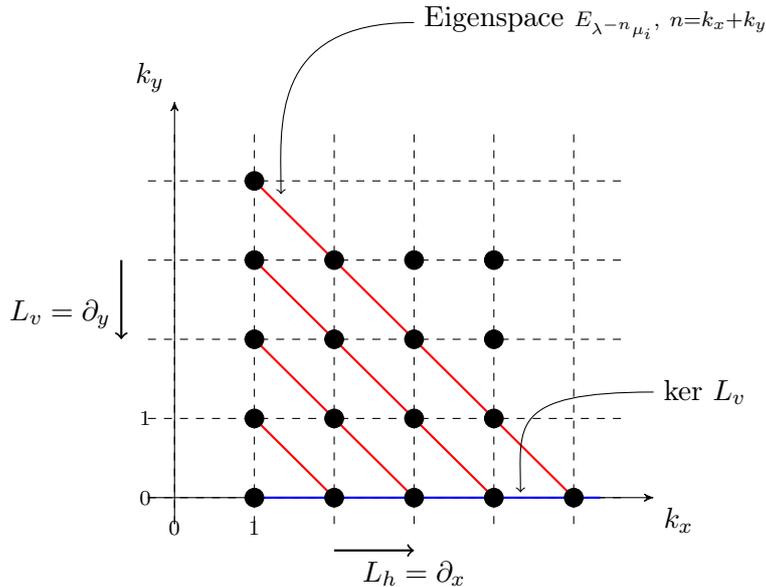

\subsection*{Invariant distributions for the vertical flow}

The above description is a first step into the direction of classifying all
distributions on $M-\Sigma$ which are invariant under the vertical flow. We
will call such distributions vertically invariant, or $L_v$-annihilated, or
sometimes $L_v$-invariant. It turns out that there is another family of such
$L_v$-annihilated distributions, which do not show up in the Ruelle
resonances and correspond to relative homology. They belong to an extended
space $\boB^{-k_h, k_v}_{ext}$ defined like $\boB^{-k_h, k_v}$ above, except
that we do not restrict to the closure of the set of smooth functions. (In
the language of Forni~\cite{forni_deviation}, elements $g$ of
$\boB_{ext}^{-k_h, k_v} \cap \ker L_v$ are the vertically quasi-invariant
distributions, see his Definition 6.4, while $g\dd x$ is the corresponding
basic current on $M-\Sigma$). \label{page_B_ext_not_B} An example of an element of $\boB_{ext}^{-k_h,
k_v}\setminus \boB^{-k_h, k_v}$ is as follows: consider a vertical segment
$\Gamma_\sigma$ ending at a singularity $\sigma$, a function $\rho$ on this
segment which is equal to $1$ on a neighborhood of the singularity and to $0$
on a neighborhood of the other endpoint of the segment, and define a
distribution $\xi_\sigma^{(0)}$ by $\langle \xi_\sigma^{(0)}, f\rangle =
\int_{\Gamma_\sigma} \rho(y) f(y)\dd y$. In other words, the corresponding
distribution on a horizontal segment $I$ is equal to $\rho(x_I) \delta_{x_I}$
if $I$ intersects $\Gamma_\sigma$ at a point $x_I$, and $0$ otherwise. It
turns out that these are essentially the only elements of $\boB_{ext}^{-k_h,
k_v}\setminus \boB^{-k_h, k_v}$: the latter has (almost) finite codimension
in the former (see Proposition~\ref{prop:finite_codim} for a precise
statement). Note that if one chooses another vertical segment
$\Gamma'_\sigma$ ending on the same singularity, then the difference of the
two distributions associated to $\Gamma_\sigma$ and $\Gamma'_\sigma$ belongs
to $\boB^{-k_h, k_v}$ when $k_h \geq 1$. The same happens if one replaces
$\rho$ by another function $\rho'$. Hence, modulo $\boB^{-k_h, k_v}$, the
distribution $\xi_\sigma^{(0)}$ is canonically defined and depends only on
$\sigma$.

\begin{prop}
\label{prop:spurious_distributions} Let $k_h, k_v\geq 3$. For $\sigma \in
\Sigma$, there exists a distribution $\xi_\sigma \in \boB_{ext}^{-k_h, k_v}$
such that $\xi_\sigma - \xi_\sigma^{(0)} \in \boB^{-k_h, k_v}$ and $L_v
\xi_\sigma$ is the constant distribution equal to $1/\Leb(M)$. Therefore, the
distributions $\xi_\sigma- \xi_{\sigma'}$ span a subspace of dimension $\Card
\Sigma- 1$ of $L_v$-annihilated distributions.
\end{prop}

The full description of $L_v$-annihilated distributions is given in the next
theorem. It states that all such distributions come from the distributions
associated to Ruelle resonances described in
Theorem~\ref{thm:main_preserves_orientations}, and additional spurious
distributions coming from the singularities as in
Proposition~\ref{prop:spurious_distributions}.

To give a precise statement, we have to deal carefully with the exceptional
situation when there is an eigenvalue $\mu'$ of $T^*$ such that
$\mu=\lambda^{-1}\mu'$ is also an eigenvalue of $T^*$: then $L_h
E_{\lambda^{-1}\mu'}$ is contained in $E_{\lambda^{-1}\mu}$, and there are
some formal difficulties.

For each eigenvalue $\mu \in \Xi = \{\mu_1,\dotsc, \mu_{2g-2}\}$, there is a
map $f \mapsto [f]$ from $E_{\lambda^{-1}\mu} \cap \ker L_v$ to $H^1(M)$,
whose image is the generalized eigenspace associated to the eigenvalue $\mu$
of $T^*$. It is an isomorphism except in the exceptional situation above
where it is onto, with a kernel equal to $L_h E_{\lambda^{-1}\mu'}$. Denote
by $E_{\lambda^{-1}\mu}^H$ a subspace of $E_{\lambda^{-1}\mu}\cap \ker L_v$
which is sent isomorphically to the generalized eigenspace of $T^*$ for the
eigenvalue $\mu$, i.e., $E_{\lambda^{-1}\mu}^H = E_{\lambda^{-1}\mu}$, except
in the exceptional case above where $E_{\lambda^{-1}\mu}^H$ is a vector
complement to $L_h E_{\lambda^{-1}\mu'}$ in $E_{\lambda^{-1}\mu}\cap \ker
L_v$.

\begin{thm}
\label{thm:Lv_invariant_preserves_orientations} Let $T$ be a linear
pseudo-Anosov map preserving orientations on a genus $g$ compact surface $M$,
with expansion factor $\lambda$ and singularity set $\Sigma$. Let $L_v$
denote the differentiation in the vertical direction. Then the space of
distributions in the kernel of $L_v$ is exactly given by the direct sum of
the constant functions, of the spaces $L_h^n E^H_{\lambda^{-1}\mu_i}$ for
$n\geq 0$ and $i=1,\dotsc, 2g-2$, of the multiples of the distributions
$\xi_\sigma - \xi_{\sigma'}$ for $\sigma,\sigma'\in \Sigma$, and of the
multiples of $L_h^n \xi_\sigma$ for $n \geq 1$ and $\sigma \in \Sigma$, where
$\xi_\sigma$ is defined in Proposition~\ref{prop:spurious_distributions}.
\end{thm}

In particular, the space of $L_v$-annihilated distributions of order $\geq
-N$ is finite-dimensional for any $N$, and its dimension grows like $(2g-2 +
\Card \Sigma)N$ when $N\to \infty$. This is an analogue
of~\cite[Theorem~7.7(i)]{forni_deviation} in our context (see
Remark~\ref{rmk:cohomological} for a further cohomological description). If
one restricts to $L_v$-annihilated distributions coming from $\boB^{-k_h,
k_v}$, one should remove the distributions $\xi_\sigma - \xi_{\sigma'}$ and
$L_h^n \xi_\sigma$. Their dimension grows like $(2g-2)N$, corresponding
to~\cite[Theorem~7.7(ii)]{forni_deviation}.

Bufetov has also studied vertically invariant distributions of the vertical
foliation of a linear pseudo-Anosov map in~\cite{bufetov_pseudo_anosov}. In
this article, the author is only interested in distributions of small order,
which can be integrated against characteristic functions of intervals. He
obtains a full description of such distributions, by more combinatorial
means, and gets further properties such as their local H\"older behavior. These
distributions correspond exactly to the elements of
$\bigcup_{\abs{\alpha}>\lambda^{-1}} E_\alpha$.

\subsection*{Solving the cohomological equation for the vertical flow}

One of the main motivations to study $L_v$-annihilated distributions is that
they are related to the cohomological equation for the vertical flow. Indeed,
if one wants to write a function $f$ as $L_v F$ for some function $F$ with
some smoothness, then one should have for any distribution $\omega$ in the
kernel of $L_v$ the equality
\begin{equation}
\label{eq:woixcuyvjklmlsw<cmviw}
  \langle \omega, f \rangle = \langle \omega, L_v F \rangle
  = -\langle L_v \omega, F\rangle
  = 0,
\end{equation}
at least if $F$ is more smooth than the order of $\omega$ and if $L_v$ is
antiselfadjoint on the relevant distributions (note that, in general, $F$
will not be supported away from the singularities, so the fact the $\langle
\omega, F\rangle$ or $\langle L_v \omega, F\rangle$ are well defined is not
obvious, and neither is the formal equality $\langle \omega, L_v F \rangle
  = -\langle L_v \omega, F\rangle$). Such necessary conditions to
have a coboundary are also often sufficient. In this direction, we obtain the
following statement. The philosophy that results on the coboundary equation
should follow from results on Ruelle resonances comes from
Giulietti-Liverani~\cite{giulietti_liverani}. Note that the converse is also
true: in a resent work, Forni~\cite{forni2018} studied Ruelle resonances and
obstructions to the existence of solution to the cohomological equation. In
particular his work independently reproves some of the results of our paper
(with very different methods). The cohomological equation was first solved
for a large class of interval exchange maps (including the ones corresponding
to pseudo-Anosov maps) in~\cite{marmi_moussa_yoccoz}. The proof we give of
the next theorem also owes a lot to the techniques
of~\cite{giulietti_liverani} (although the local affine structure makes many
arguments simpler compared to their article, but the presence of
singularities creates new difficulties, as usual).

\begin{thm}
\label{thm:coboundary_Ck_main} In the setting of
Theorem~\ref{thm:Lv_invariant_preserves_orientations}, consider a $C^\infty$
function $f$ with compact support in $M-\Sigma$. Assume that $\langle \omega,
f\rangle = 0$ for all $\omega \in \bigcup_{\abs{\alpha} \geq \lambda^{-k-1}}
E_\alpha \cap \ker L_v$. Then there exists a function $F$ on $M$ which is
$C^k$ whose $k$ derivatives are bounded and continuous on $M$, such that $f =
L_v F$ on $M-\Sigma$.
\end{thm}
The fact that $f$ is $C^\infty$ and compactly supported in $M-\Sigma$ is for
the simplicity of the statement. Indeed, the theorem also holds if $f$ is
continuous on $M-\Sigma$ and $C^{k+2}$ along horizontal lines, with $L_h^j f$
uniformly bounded for any $j \leq k+2$, see the more precise
Theorem~\ref{thm:coboundary_Ck} below (in this case, the primitive $F$ is
$C^k$ along horizontal lines). Even more, $C^{k+1+\epsilon}$ along horizontal
lines would suffice, for any $\epsilon>0$. So, the loss of derivatives in the
above theorem is really $1+\epsilon$ (which is optimal). Moreover, the $k$-th
derivative of the solution of the coboundary equation is automatically H\"older
continuous. This corresponds in our context respectively to the results
of~\cite{forni_regularity} and~\cite{marmi_yoccoz_Holder}.

It is not surprising that distributions in $E_\alpha \cap \ker L_v$ show up
as conditions to solve the cohomological equation, as explained before the
theorem. The main outcome of Theorem~\ref{thm:coboundary_Ck_main} is that
there are finitely many obstructions to be a $C^k$ coboundary. The number of
such obstructions grows like $(2g-2)k$ when $k\to \infty$, by the
classification of the Ruelle spectrum given in
Theorem~\ref{thm:main_preserves_orientations} and the following discussion.
This answers the problem raised by Forni at the end
of~\cite{forni_cohomological}, where a similar theorem is proved for the
vertical flow on generic translation surfaces, using different methods based
on the Laplacian.

Note that the distributions that appear in
Theorem~\ref{thm:coboundary_Ck_main} only come from the Ruelle spectrum. The
other $L_v$-annihilated distributions from
Theorem~\ref{thm:Lv_invariant_preserves_orientations} do not play a role. The
reason is that the formal computation in~\eqref{eq:woixcuyvjklmlsw<cmviw}
does not work for these distributions, as $F$ is not compactly supported away
from $\Sigma$. These distributions would appear if one were trying to find a
vertical primitive of $f$ which, additionally, vanishes at all singularities.

\subsection*{Trace formula}

In finite dimension, the trace of an operator is the sum of its eigenvalues.
This does not hold in general in infinite dimension (sometimes for lack of a
good notion of trace, or for lack of summability of the eigenvalues), but it
sometimes does for well behaved operators. In the dynamical world, this often
holds for analytic maps (for which the transfer operator can be interpreted
as a nuclear operator on a suitable space), but it fails most of the time
outside of this class, see~\cite{jezequel} and references therein.

In our case, it is easy to investigate this question, as we have a full
description of the Ruelle spectrum. One should also define a suitable trace
of the composition operator $\boT$. On smooth manifolds, one can define the
flat trace of a composition operator as the limit of the integral along the
diagonal of the Schwartz kernel of a smoothed version of $\boT$, when the
smoothing parameter tends to $0$. When $T$ is a diffeomorphism with isolated
fixed points, this reduces to a sum over the fixed points of
$1/{\abs{\det(\Id - DT(x))}}$, as follows from an easy computation involving
the change of variables $y = x - T x$.

In our case, the determinant is $(1-\lambda) (1-\lambda^{-1})$ everywhere,
but one should also deal with the singularities, where the smoothing
procedure is not clear (one can not convolve with a kernel because of the
singularity). We recall the notion of~\emph{Lefschetz index} of an isolated
fixed point $x$ of a homeomorphism $T$ in two dimensions (see for
instance~\cite[Section 8.4]{katok}): it is the number
\begin{equation*}
  \ind_T (x) = \deg (p \mapsto (p-T p)/\norm{p-T p}),
\end{equation*}
where the degree is computed on a small curve around $x$, identified with
$\Sbb^1$. If one could make sense of a smoothing at the singularity $\sigma$,
then its contribution to the flat trace would be $\ind_T \sigma/((1-\lambda)
(1-\lambda^{-1}))$, as follows from the same formal computation with the
change of variables $y = x- Tx$ (the index comes from the number of branches
of this map, giving a multiplicity when one computes the integral). Thus, to
have a sound definition independent of an unclear smoothing procedure, we
\emph{define} the flat trace of $\boT^n$ as
\begin{equation*}
  \flattr(\boT^n) = \sum_{T^n x = x} \frac{ \ind_{T^n} x}{(1-\lambda^n) (1-\lambda^{-n})}.
\end{equation*}
If $T^n$ is smooth at a fixed point $x$, then its index is $-1$ and we
recover the usual contribution of $x$ to the flat trace. More generally, if
$T$ is such that $\det(I-DT)$ has a limit at all fixed points of $T$ (regular
or singular) then one defines its flat trace as the sum over all fixed points
$x$ of $\ind_T x/(\lim_x \det(I-DT))$.

\begin{thm}
\label{thm:trace_formula} Let $T$ be a linear pseudo-Anosov map preserving
orientations on a compact surface $M$. Then, for all $n$,
\begin{equation}
\label{eq:sdfuiolkjxcvuiop}
  \flattr(\boT^n) = \sum_\alpha d_\alpha \alpha^n,
\end{equation}
where the sum is over all Ruelle resonances $\alpha$ of $T$, and $d_\alpha$
denotes the multiplicity of $\alpha$.
\end{thm}
\begin{proof}
The Lefschetz fixed-point formula (see~\cite[Theorem~8.6.2]{katok}) gives
\begin{align*}
  \sum_{T^n x = x} \ind_{T^n} x
  &= \tr((T^n)_{\restr H^0(M)}^*) - \tr((T^n)_{\restr H^1(M)}^*) + \tr((T^n)_{\restr H^2(M)}^*)
  \\& = 1 - \pare*{\lambda^n + \lambda^{-n} + \sum_{i=1}^{2g-2} \mu_i^n} + 1,
\end{align*}
where $\{\mu_1,\dotsc, \mu_{2g-2}\}$ denote the eigenvalues of $T^*$ on the
subspace of $H^1(M)$ orthogonal to $[\dd x]$ and $[\dd y]$, as in the
statement of Theorem~\ref{thm:main_preserves_orientations}. We can also
compute the right hand side of~\eqref{eq:sdfuiolkjxcvuiop}, using the
description of Ruelle resonances: $1$ has multiplicity one, and
$\lambda^{-k}\mu_i$ has multiplicity $k$ for $k\geq 1$. As $\sum k x^k =
x/(1-x)^2 = -1/((1-x)(1-x^{-1}))$, we get
\begin{align*}
  \sum_\alpha d_\alpha \alpha^n
  &= 1 + \sum_{i=1}^{2g-2} \sum_{k=1}^\infty k \lambda^{-nk} \mu_i^n
  = 1 - \sum_{i=1}^{2g-2} \frac{\mu_i^n}{(1-\lambda^{-n})(1-\lambda^n)}
  \\&
  = \frac{(1-\lambda^{-n})(1-\lambda^n) - \sum_{i=1}^{2g-2}\mu_i^n}{(1-\lambda^{-n})(1-\lambda^n)}
  = \frac{2 - \pare*{\lambda^n + \lambda^{-n} +\sum_{i=1}^{2g-2}\mu_i^n}}{(1-\lambda^{-n})(1-\lambda^n)}.
\end{align*}
Combining the two formulas with the definition of the flat trace, we get the
conclusion of the theorem.
\end{proof}

\subsection*{Organization of the paper}
In Section~\ref{sec:def_boB}, we define the anisotropic Banach spaces
$\boB^{-k_h, k_v}$ we will use to understand the spectrum of the composition
operator $\boT$. The construction works in any translation surface. We prove
the basic properties of these Banach spaces, including notably compact
inclusion statements, a duality result, and a cohomological interpretation of
elements of the space which are vertically invariant. All these tools are put
to good use in Section~\ref{sec:Ruelle_spectrum}, where we describe the
Ruelle spectrum of a linear pseudo-Anosov map preserving orientations,
proving Theorem~\ref{thm:main_preserves_orientations}. Then, we use (and
extend) this theorem in Section~\ref{sec:vertically_invariant} to classify
all vertically invariant distributions (proving
Theorem~\ref{thm:Lv_invariant_preserves_orientations}), and in
Section~\ref{sec:cohomological} to find smooth solutions to the cohomological
equation (proving Theorem~\ref{thm:coboundary_Ck_main}). Finally,
Section~\ref{sec:orientations} is devoted to the discussion of the Ruelle
spectrum for linear pseudo-Anosov maps which do not preserve orientations.

\section{Functional spaces on translation surfaces}
\label{sec:def_boB}

\subsection{Anisotropic Banach spaces on translation surfaces}

\label{subsec:Bkhkv}

In this section, we consider a translation surface $(M, \Sigma)$. We wish to
define anisotropic Banach spaces of distributions on such a surface, i.e.,
spaces of distributions which are smooth along the vertical direction, and
dual of smooth along the horizontal direction. Indeed, this is the kind of
space on which the transfer operator associated to a pseudo-Anosov map will
be well behaved, leading ultimately to the existence of Ruelle spectrum for
such a map, and to its explicit description. The definition we use below is
of geometric nature: we will require that the objects in our space can be
integrated along horizontal segments when multiplied by smooth functions, and
that they have vertical derivatives with the same property. This
simple-minded definition in the spirit of~\cite{GL_Anosov2, avila_gouezel} is
very well suited for the constructions we have in mind below (especially for
the cohomological interpretation in Paragraph~\ref{subsec:cohomological}
below) and makes it possible to deal transparently with the singularities.
However, it is probably possible to use other approaches as explained
in~\cite{baladi_ultimate} and references therein.

Let $V^h$ be the unit norm positively oriented horizontal vector field, i.e.,
the vector field equal to $1 \in \C$ in the translation charts. It is
$C^\infty$ on $M-\Sigma$, but singular at $\Sigma$. In particular, the
derivation $L_h$ given by this vector field acts on $C^\infty(M-\Sigma)$. In
the same way, the vertical vector field $V^v$ (equal to $\ic$ in the complex
translation charts) is $C^\infty$ on $M-\Sigma$, and the corresponding
derivation $L_v$ acts on $C^\infty(M-\Sigma)$. On this space, the two
derivations $L_v$ and $L_h$ commute, as this is the case in $\C$.

Choose two real numbers $k\geq 0$ and $\beta >0$. Denote by $\boI^h_\beta$
the set of horizontal segments of length $\beta$ in $M-\Sigma$. For $I\in
\boI^h_\beta$, denote by $C^k_c(I)$ the set of $C^k$ functions on $I$ which
vanish on a neighborhood of the boundary of $I$, endowed with the $C^k$ norm
(when $k$ is not an integer, this is the set of functions of class
$C^{\lfloor k \rfloor}$ whose $\lfloor k \rfloor$-th derivative is H\"older
continuous with exponent $k-\lfloor k \rfloor$).

When $k_h\geq 0$ is a nonnegative real number, and $k_v\geq 0$ is an integer,
we define a seminorm on $C^\infty_c(M-\Sigma)$ by
\begin{equation*}
  \norm{f}'_{-k_h, k_v, \beta} =
    \sup_{I \in \boI^h_\beta} \sup_{\phi \in C^{k_h}_c(I),
    \norm{\phi}_{C^{k_h}} \leq 1} \abs*{\int_I \phi \cdot (L_v)^{k_v} f \dd x}.
\end{equation*}
Essentially, this seminorm measures $k_v$ derivatives in the vertical
direction, and $-k_h$ derivatives in the horizontal direction (as one is
integrating against a function with $k_h$ derivatives). Hence, it is indeed a
norm of anisotropic type. One could define many such norms, but this one is
arguably the simplest one: it takes advantage of the fact that the horizontal
and vertical foliations are smooth, and even affine.

\begin{prop}
\label{prop:beta_irrelevant} If $\beta$ is smaller than the length of the
shortest horizontal saddle connection, then this seminorm does not really
depend on $\beta$: if $\beta_1$ is another such number, then there exists a
constant $C=C(\beta, \beta_1, k_h, k_v)$ such that , for any $f \in
C^\infty_c(M-\Sigma)$,
\begin{equation*}
C^{-1}\norm{f}'_{-k_h, k_v, \beta_1} \leq \norm{f}'_{-k_h, k_v, \beta} \leq C\norm{f}'_{-k_h, k_v, \beta_1}.
\end{equation*}
\end{prop}
We recall that a horizontal saddle connection is a horizontal segment
connecting two singularities. There is no horizontal saddle connection in a
surface carrying a pseudo-Anosov map: otherwise, iterating the inverse of the
map (which contracts uniformly the horizontal segments), we would deduce the
existence of arbitrarily short horizontal saddle connections, a
contradiction.
\begin{proof}
Assume for instance $\beta_1>\beta$. The inequality $\norm{f}'_{-k_h, k_v,
\beta} \leq \norm{f}'_{-k_h, k_v, \beta_1}$ is clear: an interval $I \in
\boI^h_\beta$ is contained in an interval $I_1$ in $\boI^h_{\beta_1}$ as
$\beta_1$ is smaller than the length of any horizontal saddle connection.
Moreover, a compactly supported test function $\phi$ on $I$ can be extended
by $0$ to outside of $I$ to get a test function on $I_1$. The result follows
readily.

Conversely, consider a smooth partition of unity $(\rho_j)_{j\in J}$ on
$[0,\beta_1]$ by $C^\infty$ functions whose support has length at most
$\beta$ (we do not require that the functions vanish at $0$ or $\beta_1$.
Using this partition of unity, for $I_1 \in \boI^h_{\beta_1}$, one may
decompose a test function $\phi\in C^{k_h}_c(I_1)$ as the sum of the
functions $\phi\cdot \rho_j$, which are all compactly supported on intervals
belonging to $\boI^h_{\beta}$. Moreover, their $C^{k_h}$ norms are controlled
by the $C^{k_h}$ norm of $\phi$. It follows that the integrals defining
$\norm{f}'_{-k_h, k_v, \beta_1}$ are controlled by finitely many integrals
that appear in the definition of $\norm{f}'_{-k_h, k_v, \beta}$, giving the
inequality $\norm{f}'_{-k_h, k_v, \beta_1} \leq C \norm{f}'_{-k_h, k_v,
\beta}$.
\end{proof}

By the above proposition, we may use any small enough $\beta$. For
definiteness, let us choose once and for all $\beta=\beta_0$ much smaller
than the distance between any two singularities. This implies that, in all
the local discussions, we will have to consider at most one singularity. From
this point on, we will keep $\beta_0$ implicit, unless there is an ambiguity.

The seminorms $\norm{\cdot}'_{-k_h, k_v}$ are not norms in general on
$C^\infty_c(M-\Sigma)$. For instance, if there is a cylinder made of closed
vertical leaves, then one may find a function which is constant on each
vertical leaf, vanishes close to the singularities, and is nevertheless not
everywhere zero. Then $L_v f= 0$, so that $\norm{f}'_{-k_h, k_v}=0$ if
$k_v>0$, but still $f \neq 0$. This is not the case when there is no vertical
connection: in this case, all vertical leaves are dense, hence a function
which is constant along vertical leaves and vanishes on a neighborhood of the
singularities has to vanish everywhere. In general, this remark indicates
that the above seminorms do not behave very well by themselves. On the other
hand, the following norm is much nicer:
\begin{equation}
\label{eq:define_norm}
  \norm{f}_{-k_h, k_v} = \sup_{j\leq k_v} \norm{f}'_{-k_h, j}
  = \sup_{j\leq k_v} \sup_{I \in \boI^h} \sup_{\phi \in C^{k_h}_c(I),
    \norm{\phi}_{C^{k_h}} \leq 1} \abs*{\int_I \phi \cdot L_v^j f \dd x}.
\end{equation}
This is obviously a norm on $C^\infty_c(M-\Sigma)$. Indeed, if a function $f$
is not identically zero, then it is nonzero at some point $x$. Taking a
horizontal interval $I$ around $x$ and a test function $\phi$ on $I$
supported on a small neighborhood of $x$, one gets $\int_I \phi f \dd x \neq
0$, and therefore $\norm{f}_{-k_h, k_v}>0$.

Then, let us define the space $\boB^{-k_h, k_v}$ as the (abstract) completion
of $C^\infty_c(M-\Sigma)$ for this norm. Note that all the linear forms
$\ell_{I, \phi,j} : f \mapsto \int_I \phi \cdot L_v^j f \dd x$, initially
defined on $C^\infty_c(M-\Sigma)$, extend by continuity to $\boB^{-k_h, k_v}$
(for $I \in \boI^h$ and $\phi \in C^{k_h}_c(I)$ and $j \leq k_v$).
Heuristically, an element in $\boB^{-k_h, k_v}$ can be differentiated in the
vertical direction, and integrated in the horizontal direction. Moreover, the
norm of an element in $\boB^{-k_h, k_v}$ is
\begin{equation}
\label{eq:norme_completion}
  \norm{f}_{-k_h, k_v}
  = \sup_{j\leq k_v} \sup_{I \in \boI^h} \sup_{\phi \in C^{k_h}_c(I),
    \norm{\phi}_{C^{k_h}} \leq 1} \abs{\ell_{I,\phi, j}(f)}.
\end{equation}
This follows directly from the definition of the norm on
$C^\infty_c(M-\Sigma)$ and from the construction of $\boB^{-k_h, k_v}$ as its
completion.

\begin{rmk}
\label{rmk:non_integers} In the spaces $\boB^{-k_h, k_v}$ we have just
defined, the parameter $k_h$ of horizontal regularity can be any nonnegative
real, but the parameter $k_v$ of vertical regularity has to be an integer, as
it counts a number of derivatives. One could also use a non-integer vertical
parameter $k_v$, requiring additionally the following control: if $k_v = k +
r$ where $k$ is an integer and $r\in(0,1)$, then we require the boundedness
of
\begin{equation*}
  \epsilon^{-r} \abs*{\int_{I_0} \phi_0 L_v^k f \dd x -\int_{I_\epsilon} \phi_\epsilon L_v^k f \dd x}
\end{equation*}
when $I_0$ is a horizontal interval of length $\beta_0$, $\phi_0$ is a
compactly supported $C^{k_h}$ function on $I_0$ with norm at most $1$,
$\epsilon\in [0,\beta_0]$ is such that one can translate vertically the
interval $I_0$ into an interval $I_\epsilon$ without hitting any singularity,
and $\phi_\epsilon$ is the push-forward of $\phi_0$ on $I_\epsilon$ using the
vertical translation. In other words, we are requiring that $L_v^k f$ is
H\"older continuous of order $r$ vertically, in the distributional sense. All
the results that follow are true for such a norm, but the proofs become more
cumbersome while the results are not essentially stronger, so we will only
consider integer $k_v$ for the sake of simplicity.
\end{rmk}

Let $\phi$ be a $C^\infty$ function on $M$, and denote by $\dLeb$ the flat
Lebesgue measure on $M$. Then $\ell_\phi : f \mapsto  \int f \phi \dLeb$ is a
linear form on $C^\infty_c(M-\Sigma)$. Contrary to the previous linear forms,
$\ell_\phi$ does \emph{not} extend to a linear form on $\boB^{-k_h, k_v}$,
because of the singularities: from the point of view of the $C^\infty$
structure, horizontals and verticals close to the singularity have a lot of
curvature, so that the restriction of $\phi$ to $I \in \boI^h$ is $C^k$, but
with a large $C^k$ norm (larger when $I$ is closer to the singularity).
This prevents the extension of $\ell_\phi$ to $\boB^{-k_h, k_v}$. On the
other hand, if $\phi$ is supported by $M-B(\Sigma, \delta)$, then one has a
control of the form $\abs{\ell_\phi(f)} \leq C(\delta)
\norm{\phi}_{C^{k_h}}\norm{f}_{-k_h, k_v}$, so that $\ell_\phi$ extends
continuously to $\boB^{-k_h, k_v}$. More precisely, denote by
$\boD^\infty(M-\Sigma)$ the set of distributions on $M-\Sigma$, i.e., the
dual space of $C^\infty_c(M-\Sigma)$ with its natural topology. Then the
above argument shows that there is a map $i : \boB^{-k_h, k_v} \to
\boD^\infty(M-\Sigma)$, extending the canonical inclusion
$C^\infty_c(M-\Sigma) \to \boD^\infty(M-\Sigma)$ given by $\langle i(f), \phi
\rangle = \int f \phi \dLeb$. Locally, if $\phi$ is supported by a small
rectangle foliated by horizontal segments $I_t \in \boI^h$ (where $t$ is an
arc-length parametrization along the vertical direction), one has the
explicit description
\begin{equation}
\label{eq:decrit_if}
  \langle i(f), \phi \rangle = \int \ell_{I_t, \phi_{\restr I_t}, 0}(f) \dd t.
\end{equation}
Indeed, this formula holds when $f$ is $C^\infty$, and extends by uniform
limit to all elements of $\boB^{-k_h, k_v}$.

\begin{prop}
\label{prop:distrib} The map $i: \boB^{-k_h, k_v} \to \boD^\infty(M-\Sigma)$
is injective. Therefore, one can identify $\boB^{-k_h, k_v}$ with a space of
distributions on $M-\Sigma$.
\end{prop}
\begin{proof}
Consider $I \in \boI^h$ and $\phi \in C^{k_h}_c(I)$. For small enough $t$,
one can shift vertically $I$ by $t$, and obtain a new interval $I_t \in
\boI^h$, as well as a function $\phi_t:I_t \to \R$ (equal to the composition
of the vertical projection from $I_t$ to $I$, and of $\phi$). For any $f \in
C^\infty_c(M-\Sigma)$, the function $t \mapsto \ell_{I_t, \phi_t, 0}(f)$ is
$C^{k_v}$, with successive derivatives $t \mapsto \ell_{I_t, \phi_t, j}(f)$.
An element $f \in \boB^{-k_h, k_v}$ can be written as a limit of a Cauchy
sequence of smooth functions. Then $\ell_{I_t, \phi_t, j}(f_n)$ converges
uniformly to $\ell_{I_t, \phi_t, j}(f)$. Passing to the limit in $n$, we
deduce that $t \mapsto \ell_{I_t, \phi_t, 0}(f)$ is $C^{k_v}$, with
successive derivatives $t \mapsto \ell_{I_t, \phi_t, j}(f)$.

Consider a nonzero $f \in \boB^{-k_h, k_v}$, with norm $c>0$.
By~\eqref{eq:norme_completion}, there exist $I$, $\phi$ and $j$ such that
$\abs{\ell_{I,\phi,j}(f)} \geq c/2$. Let us shift $I$ vertically as above.
The function $t \mapsto \ell_{I_t, \phi_t, 0}(f)$ has a $j$-th derivative
which is nonzero at $0$, hence it is not locally constant. In particular, it
does not vanish at some parameter $t_0$. Consider $\delta$ such that it is
almost constant on the interval $[t_0-\delta, t_0+\delta]$ by continuity. Let
$\psi$ be a smooth function with positive integral, supported by
$[t_0-\delta, t_0+\delta]$. In local coordinates, let us finally write
$\zeta(x,y)= \phi(x) \psi(y)$. It satisfies $\langle i(f), \zeta \rangle \neq
0$ thanks to the explicit description~\eqref{eq:decrit_if} for $i(f)$.
\end{proof}

It follows that one can think of elements of $\boB^{-k_h, k_v}$ as objects
that can be integrated along horizontal segments, or after an additional
vertical integration as distributions. Even better, since the elements of
$\boB^{-k_h, k_v}$ are designed to be integrated horizontally, the natural
object to consider is rather $f\dd x$. This is a current, i.e., a
differential form with distributional coefficients, but it is nicer than
general currents as it can really be integrated along horizontal segments
(i.e., it is regular in the vertical direction). The process that associates
to such an object a global distribution is simply the exterior product with
$\dd y$. Going back and forth like that between $0$-currents and $1$-currents
will be an essential feature of the forthcoming arguments.

The next lemma makes it possible to use partitions of unity, to decompose an
element of $\boB^{-k_h, k_v}$ into a sum of elements supported in arbitrarily
small balls.
\begin{lem}
\label{lem:partition_unite} Let $\psi \in C^\infty(M)$ be constant in the
neighborhood of each singularity. Then the map $f \mapsto \psi f$, initially
defined on $C^\infty_c(M-\Sigma)$, extends continuously to a linear map on
$\boB^{-k_h, k_v}$.
\end{lem}
\begin{proof}
We have to bound $\int_I \phi \cdot L_v^j(\psi f) \dd x$ when $I$ is a
horizontal interval, $\phi$ a compactly supported $C^{k_h}$ function on $I$,
and $j\leq k_v$. We have $L_v^j(\psi f) = \sum_{k\leq j}\binom{j}{k}
L_v^{j-k}\psi \cdot L_v^k f$, hence this integral can be decomposed as a sum
of integrals of $L_v^k f$ against the functions $\phi \cdot L_v^{j-k}\psi$
which are $C^{k_h}$ and compactly supported on $I$. This concludes the proof,
by definition of $\boB^{-k_h, k_v}$.
\end{proof}

One may wonder how rich the space $\boB^{-k_h, k_v}$ is, and if the choice to
take the closure of the set of functions vanishing on a neighborhood of the
singularities really matters. Other functions are natural, for instance the
constants, or more generally the smooth functions that factorize through the
covering projection $\pi: z \mapsto z^p$ around each singularity of angle
$2\bpi p$. The largest natural class is the space of functions $f$ which are
$C^\infty$ on $M-\Sigma$ and such that, for all indices $a_h$ and $a_v$, the
function $L_v^{a_v} L_h^{a_h} f$ is bounded. The next lemma asserts that
starting from any of these classes of functions would not make any
difference, as our space $\boB^{-k_h, k_v}$ is already rich enough to contain
all of them.

\begin{lem}
\label{lem:boB_riche} Consider a function $f$ on $M$ which is $C^{k_v}$ on
every vertical segment and such that $L_v^k f$ is bounded and continuous on
$M-\Sigma$ for any $k \leq k_v$. Then the function $f$ (or rather the
corresponding distribution $i(f)$) belongs to $\boB^{-k_h, k_v}$ for any
$k_h\geq 0$. This is in particular the case of the constant function $f=1$.
\end{lem}
\begin{proof}
First, if $f$ is supported away from the singularities, one shows that $f \in
\boB^{-k_h, k_v}$ by convolving it with a smooth kernel $\rho_\epsilon$:
the sequence $f_\epsilon = f * \rho_\epsilon$ thus constructed is $C^\infty$ and forms
a Cauchy sequence in $\boB^{-k_h, k_v}$, hence it converges in this space to a limit.
As it converges to $f$ in the distributional sense, this shows $f \in \boB^{-k_h, k_v}$.

To handle the general case, by taking a partition of unity, it suffices to
treat the case of a function $f$ supported in a small neighborhood of a
singularity, such that $L_v^k f$ is continuous and bounded for any $k \leq
k_v$. Let $\pi$ denote the covering projection, defined on a neighborhood of
this singularity. Let $u$ be a real function, equal to $1$ on a neighborhood
of $0$, supported in $[-1,1]$. Let $N>0$ be large enough. For $\delta > 0$,
we define a function $\rho_\delta(x+ \ic y) = u(x/\delta^N) u(y/\delta)$,
supported on the neighborhood $[-\delta^N, \delta^N] + \ic [-\delta, \delta]$
of $0$ in $\C$.

We claim that, if $N>k_v$, then in $\C$ one has $\norm{\rho_\delta}_{-k_h,
k_v} \to 0$ when $\delta \to 0$, where by $\norm{\cdot}_{-k_h, k_v}$ we mean the formal
expression~\eqref{eq:define_norm}, which makes sense for any function but could be
infinite. To prove this, consider a horizontal
interval $I$ of length $\beta_0$, a function $\phi \in C_c^{k_h}(I)$ with
norm at most $1$, and a differentiation order $j \leq k_v$. Then
\begin{align*}
  \abs*{\int_I \phi \cdot L_v^j \rho_\delta \dd x}
  &= \delta^{-j} \abs*{\int_I \phi \cdot u(x/\delta^N) u^{(j)}(y/\delta) \dd x}
  \\&\leq \delta^{-j} \norm{\phi}_{C^0} \norm{u}_{C^0} \norm{u^{(j)}}_{C^0} \Leb([-\delta^N, \delta^N]).
\end{align*}
This quantity tends to $0$ if $N>j$, as claimed.

The same computation, taking moreover into account the fact that the vertical
derivatives of $f$ are bounded, shows that $\norm{f \cdot \rho_\delta\circ
\pi}_{-k_h, k_v} \to 0$ when $\delta \to 0$. It follows that the sequence
$f_n = f (1- \rho_{1/n} \circ \pi)$ is a Cauchy sequence in $\boB^{-k_h,
k_v}$, made of functions in $C^{k_v}_c(M-\Sigma)$ (which is indeed included
in $\boB^{-k_h, k_v}$ by the first step). It converges (in $L^1$,
and therefore in the sense of distributions) to $f$, which has therefore to
coincide with its limit in $\boB^{-k_h, k_v}$.
\end{proof}

In particular, if $\Sigma$ contains an artificial singularity $\sigma$ (i.e.,
around which the angle is equal to $2\bpi$), then one gets the same space
$\boB^{-k_h, k_v}$ by using the singularity sets $\Sigma$ or $\Sigma
-\{\sigma\}$.

\bigskip

The horizontal and vertical derivations $L_h$ and $L_v$ act on
$C^\infty_c(M-\Sigma)$. By duality, they also act on $\boD^\infty(M-\Sigma)$.
In view of Proposition~\ref{prop:distrib} asserting that $\boB^{-k_h, k_v}$ is
a space of distributions, it makes sense to ask if they
stabilize these spaces, or if they send one into the other.

\begin{prop}
\label{prop:Lh_Lv_action} The derivation $L_h$ maps continuously $\boB^{-k_h,
k_v}$ to $\boB^{-k_h-1, k_v}$, and it satisfies $\ell_{I,\phi, j}(L_h f) =
-\ell_{I,\phi', j}(f)$ for every $I \in \boI^h$, $\phi \in C^{k_h+1}_c(I)$,
$j \leq k_v$ and $f \in \boB^{-k_h, k_v}$.

The derivation $L_v$ maps continuously $\boB^{-k_h, k_v}$ to $\boB^{-k_h,
k_v-1}$ if $k_v>0$, and it satisfies $\ell_{I,\phi, j}(L_v f) = \ell_{I,
\phi, j+1}(f)$ for every $I \in \boI^h$, $\phi \in C^{k_h}_c(I)$, $j \leq
k_v-1$ and $f \in \boB^{-k_h, k_v}$.
\end{prop}
\begin{proof}
The formulas $\ell_{I,\phi, j}(L_h f) = -\ell_{I,\phi', j}(f)$ and
$\ell_{I,\phi, j}(L_v f) = \ell_{I, \phi, j+1}(f)$ are obvious when $f$ is a
smooth function. The general result follows by density.
\end{proof}

\begin{lem}
\label{lem:Lh_0} Assume that there is no horizontal saddle connection in $M$.
Let $f\in \boB^{-k_h, k_v}$ satisfy $L_h f = 0$. Then $f$ is a constant
function.
\end{lem}
\begin{proof}
As $L_h f = 0$, one has $\ell_{I,\phi', 0}(f) = 0$ for any smooth function
$\phi$ on a horizontal interval $I$. Denoting by $\tau_h$ the translation by
$h$, one gets $\ell_{I,\phi, 0}(f) = \ell_{I,\phi \circ \tau_h, 0}(f)$ if
$\phi$ and $\phi \circ \tau_h$ both have their support in $I$. It follows
that the distribution induced by $f$ on a bi-infinite horizontal leaf is
invariant by translation. Therefore, it is a multiple $c\dLeb$ of Lebesgue
measure. Since there is no horizontal saddle connection by assumption, the
horizontal flow is minimal by Keane's Criterion. In particular, the above
bi-infinite horizontal leaf is dense. At the quantities $\ell_{I,\phi, 0}(f)$
vary continuously when one moves $I$ vertically, it follows that $f$ is equal
to $c\dLeb$ on all horizontal intervals.
\end{proof}

We want to stress that Lemma~\ref{lem:Lh_0} is wrong for $L_v$. A measure
$\mu$ which is invariant for the vertical flow can locally be written as
$\nu\otimes \dd y$, where $\nu$ is a measure along horizontal leaves,
invariant under vertical holonomy. Writing $\nu$ as a limit of measures which
are equivalent to Lebesgue and with smooth densities, one checks that $\mu$
belongs to $\boB^{-k_h, k_v}$, and moreover it satisfies $L_v \mu=0$. In a
translation surface in which the vertical flow is minimal but not uniquely
ergodic, one can find such examples where $\mu$ is not Lebesgue measure.

In the case of surfaces associated to pseudo-Anosov maps, the vertical flow
is uniquely ergodic, so this argument does not apply. However, we will
see later that there are still many nonconstant distributions $f$ in
$\boB^{-k_h, k_v}$ which satisfy $L_v f = 0$.

It is enlightening to try to prove that $f\in \boB^{-k_h, k_v}$ with $L_v f =
0$ has to be constant, and see where the argument fails. The problem stems
from the fact that $f$ is a distribution on horizontal segments. Let $F$ be a
dense vertical leaf, let $I_t$ be a small horizontal interval around the
point at height $t$ on $F$, and let $\phi$ be a function on $I_0$ that we
push vertically to a function on $I_t$ (still denoted $\phi$) while this is
possible. Then we get $\int_{I_t} \phi f \dd x = \int_{I_0} \phi f \dd x$ as
$L_v f= 0$. If this were true for all real $t$, then we would deduce that $f$
is constant. However, the support of $\phi$ has positive length. Hence, when
we push it vertically, we will encounter a singularity in finite time, and
the argument is void afterwards. We could say something on a longer time
interval if we used a function $\tilde\phi$ with smaller support, but the
same problem will happen again. The key point is a competition between the
speed at which $F$ fills the surface, and how close to singularities it
passes. The existence of non-constant distributions $f$ with $L_v f = 0$ is a
manifestation of the fact that $F$ is often too close to singularities.

A related but more detailed discussion is made before the proof of
Theorem~\ref{thm:Lv_integre}, where we study the existence of primitives
under $L_v$ of some eigendistributions, not only $0$.

\subsection{Compact inclusions}

In this paragraph, we prove the following proposition, ensuring that there
is inclusion (resp.\ compact inclusion) in the family of spaces $\boB^{-k_h, k_v}$
if one requires less (resp.\ strictly less) regularity in all directions.
This corresponds to the usual intuitions.

\begin{prop}
\label{prop:compact_inclusion} Consider $k'_h$ with $-k'_h \leq -k_h$ (i.e.,
$k'_h \geq k_h$) and $k'_v$ with $k'_v \leq k_v$. Then there is a continuous
inclusion $\boB^{-k_h, k_v} \subseteq \boB^{-k'_h, k'_v}$. If the two
inequalities are strict, this inclusion is compact.
\end{prop}
\begin{proof}
The inclusion $\boB^{-k_h, k_v} \subseteq \boB^{-k'_u, k'_v}$ when $k'_h \geq
k_h$ and $k'_v \leq k_v$ is obvious, as one uses less linear forms in the
second space than in the first space to define the norm.

For the compact inclusion, we will use the following criterion. Let $\boB
\subseteq \boC$ be two Banach spaces. Assume that, for every $\epsilon>0$,
there exist finitely many continuous linear forms $\ell_1,\dotsc,\ell_P$ on
$\boB$ such that, for any $x \in \boB$,
\begin{equation}
\label{eq:critere_compacite}
  \norm{x}_{\boC} \leq \epsilon \norm{x}_{\boB} + \sum_{p \leq P} \abs{\ell_p(x)}.
\end{equation}
Then the inclusion of $\boB$ in $\boC$ is compact.

To prove the criterion, suppose its assumptions are satisfied, and consider a
sequence $x_n \in \boB$ of elements with norm at most $1$. Extracting a
subsequence, one can ensure that all the sequences $\ell_i(x_n)$ converge,
for $i\leq P$. We deduce from the above inequality that $\limsup_{m,n\to
\infty} \norm{x_m -x_n}_{\boC} \leq 2 \epsilon$. By a diagonal argument, one
can then extract a subsequence of $x_n$ which is a Cauchy sequence in $\boC$,
and therefore converges.

Let us now apply the criterion to $\boB = \boB^{-k_h, k_v}_{5\beta_0}$ and
$\boC = \boB^{-k'_h, k'_v}_{\beta_0}$ with $k'_h>k_h$ and $k'_v<k_v$. We take
larger intervals in the first space than in the second space for technical
convenience, but this is irrelevant for the result as the spaces do not
depend on $\beta$, see Proposition~\ref{prop:beta_irrelevant}.

Let us first fix a finite family of intervals $(J_n)_{n\leq N}$ in
$\boI^h_{5\beta_0}$ such that any interval in $\boI^h_{\beta_0}$ can be
translated vertically by at most $\epsilon/2$, without hitting a singularity,
and end up in one of the $J_n$, or even better in its central part denoted by
$J_n[\beta_0,4\beta_0]$. Such a family exists by compactness, and the
singularities do not create any problem there. Then, on each $J_n$, let us
fix finitely many functions $(\phi_{n,k})_{k\leq K}$ in $C^{k'_h}_c(J_n)$
with norm at most $1$ such that, for any function $\phi \in C^{k'_h}(J_n)$
with $C^{k'_h}$ norm at most $1$ and with support included in $J_n[\beta_0,
4\beta_0]$, there exists $k$ such that $\norm{\phi - \phi_{n,k}}_{C^{k_h}}
\leq \epsilon/2$. Their existence follows from the compactness of the
inclusion of $C^{k'_h}$ in $C^{k_h}$. We will use the linear forms
$\ell_{n,k,j} = \ell_{J_n, \phi_{n,k}, j}$ for $n \leq N$, $k \leq K$ and $j
\leq k_v$ to apply the criterion~\eqref{eq:critere_compacite}.

Let us fix $f \in \boB^{-k_h, k_v}_{5\beta_0}$. We want to bound its norm in
$\boB^{-k'_h, k'_v}_{\beta_0}$. By density, it is enough to do it for $f \in
C^\infty_c(M-\Sigma)$ -- this does not change anything to the following
argument, but it is comforting. Consider thus $I \in \boI^h_{\beta_0}$, and
$\phi \in C_c^{k'_h}(I)$ with norm at most $1$, and $j \leq k'_v < k_v$. Let
$(I_t)_{0 \leq t \leq \delta}$ be vertical shifts of $I$, parameterized by
the vertical length $t$, with $I_\delta$ included in an interval
$J_n[\beta_0,4\beta_0]$ and $\delta \leq \epsilon/2$. Denote by $\phi_t$ the
push-forward of $\phi$ on $I_t$. Integrating by parts, one gets
\begin{equation*}
  \int_{I_0} \phi\cdot L_v^j f \dd x = \int_{I_\delta} \phi_\delta \cdot L_v^j f \dd x
  -\int_0^\delta \pare*{\int_{I_t} \phi_t L_v^{j+1}f \dd x} \dd t.
\end{equation*}
The integrals on each $I_t$ are bounded by $\norm{f}_{-k_h, k_v}$ as $j\leq
k'_v < k_v$. Hence, the last term is at most $\delta \norm{f}_{-k_h, k_v}
\leq (\epsilon/2) \norm{f}_{-k_h, k_v}$. In the first term, choose $k$ such
that $\norm{\phi_\delta - \phi_{n, k}}_{C^{k_h}} \leq \epsilon/2$. Then this
integral is bounded by $(\epsilon/2)\norm{f}_{-k_h, k_v} +
\abs{\ell_{n,k,j}(f)}$. We have proved that
\begin{equation*}
  \norm{f}_{-k'_h, k'_v} \leq \epsilon \norm{f}_{-k_h, k_v} + \max_{n,k,j} \abs{\ell_{n,k,j}(f)}.
\end{equation*}
This shows that the compactness criterion~\eqref{eq:critere_compacite}
applies, and concludes the proof.
\end{proof}

\subsection{Duality}

Let us define the spaces $\check \boB^{\check k_h, -\check k_v}$ just like
the spaces $\boB^{-k_h, k_v}$ but exchanging horizontals and verticals.
Hence, $\check k_v$ quantifies the regularity of a test function in the
vertical direction, and $\check k_h$ the number of permitted derivatives in
the horizontal direction. The derivations $L_v$ and $L_h$ still act on
$\check\boB$, as in Proposition~\ref{prop:Lh_Lv_action}, but their roles are
swapped compared to $\boB$.

Some of the arguments later to identify the spectrum and the multiplicities
of a pseudo-Anosov map rely on a duality argument, exchanging the roles of
the horizontal and vertical directions. To carry out this argument, we need
to show that there is a duality between the spaces $\boB^{-k_h, k_v}$ and
$\check \boB^{\check k_h, -\check k_v}$ when the global regularity is
positive enough in every direction, i.e., when $-k_h+\check k_h \geq 2$ and
$k_v - \check k_v \geq 0$ (or conversely, as one can exchange the two
directions -- it is possible that the duality holds if $\check k_h - k_h \geq
0$ and $k_v - \check k_v \geq 0$, but our proof requires a little bit more).
This is not surprising: $g \in \boB^{\check k_h, -\check k_v}$ has
essentially $\check k_h$ derivatives along horizontals, and $f \in
\boB^{-k_h, k_v}$ can be integrated along horizontals against $C^{k_h}$
functions, so if $\check k_h \geq k_h$ one expects that one can integrate the
product $fg$ along horizontals, and therefore globally. This argument is
wrong since the horizontal regularity of $g$ is only in the distributional
sense, so we will also have to take advantage of the vertical smoothness of
$f$. Using a computation based on suitable integrations by parts, it is easy
to make this argument rigorous away from singularities. However, as it is
often the case, the proof is much more delicate close to singularities, as
integrations by parts can not cross the singularity, giving rise to
additional boundary terms that can a priori not be controlled, unless one
proceeds in a roundabout way as in the following proof. The technical
difficulty of this proof is probably related to our choice of Banach spaces:
it is possible that another choice of Banach space makes this proposition
essentially trivial. This proof can be skipped on first reading.

\begin{prop}
\label{prop:dualite} Assume $-k_h+\check k_h \geq 2$ and $k_v - \check k_v
\geq 0$. Then there exists $C>0$ such that, for any $f, g \in
C^\infty_c(M-\Sigma)$, one has
\begin{equation*}
  \abs*{\int f g \dLeb} \leq C \norm{f}_{\boB^{-k_h, k_v}} \cdot \norm{g}_{\check \boB^{\check k_h, - \check k_v}}.
\end{equation*}
Therefore, the map $(f,g) \mapsto \int fg \dLeb$ extends by continuity to a
bilinear map on $\boB^{-k_h, k_v} \times \check \boB^{\check k_h, - \check
k_v}$ that we denote by $\langle f, g \rangle$.
\end{prop}

The proof will rely on a decomposition of $f$ into basic pieces for which all
the above integrals can be controlled. We will denote by $\boH$ the set of
local half-planes around all singularities, bounded by horizontal or vertical
lines. Specifically, if $\sigma$ is a singularity of angle $2\bpi \kappa$
with covering projection $\pi$, these sets are the $\kappa$ components of
$\pi^{-1}\{z \st \Re z \geq 0\}$ in a neighborhood of $\sigma$, intersected
with a small disk around $\sigma$, and similarly for the upper half-planes,
lower half-planes and left half-planes, giving rise to $4\kappa$ half-planes
around $\sigma$.
\begin{lem}
\label{lem:decomposition_f} Fix $k_h$ and $k_v$. There exist $N$, $C$, and
rectangles $(R_i)_{i\leq N}$ away from the singularities with the following
property. For any $f \in C^\infty_c(M-\Sigma)$, there is a decomposition
\begin{equation}
\label{eq:decompose_f}
  f = \sum_{i=1}^N f_i + \sum_{\sigma \in \Sigma} f_\sigma + \sum_{H\in \boH} f_H
\end{equation}
where all the $f_i$ and $f_\sigma$ and $f_H$ are $C^{k_v}$ functions with
compact support in $M-\Sigma$. They belong to $\boB^{-k_h, k_v}$ and have
norm at most $C \norm{f}_{-k_h, k_v}$. Moreover, each $f_i$ is supported in
$R_i$, each $f_H$ is supported in $H$, and each $f_\sigma$ is supported in a
small disk $D_\sigma$ around $\sigma$ and is constant on the fibers of the
covering projection $\pi$ around $\sigma$.
\end{lem}
\begin{proof}
Multiplying $f$ by a partition of unity, we can assume that $f$ is supported
in a small disk around a singularity $\sigma$ with angle $2\bpi\kappa$ (the
terms away from the singularities will give rise to the terms $f_i$ in the
decomposition~\eqref{eq:decompose_f}). We have to construct a decomposition
\begin{equation}
\label{eq:wpuxcvopiuwxopicv}
  f = f_\sigma + \sum_{H\in \boH_\sigma} f_H
\end{equation}
as in the statement of the lemma, where $\boH_\sigma$ denotes the set of
half-planes around $\sigma$. We assume $\norm{f}_{-k_h, k_v}\leq 1$ for
definiteness.

Let $\pi=\pi_\sigma$ be the covering projection, sending $\sigma$ to $0$. We
may assume that $\pi^{-1}([-a,a]^2)$ only contains $\sigma$ as a singularity,
and that $f$ is supported in $\pi^{-1}([-a/2, a/2]^2)$. Denote by $\omega =
e^{2\ic\bpi/\kappa}$ the fundamental $\kappa$-th root of unity. Let $R$ be
the rotation by $2\bpi$ around $\sigma$. For $q \in \Z/\kappa \Z$, let
$f_q(z) = \kappa^{-1} \sum_{j=0}^{\kappa-1} \omega^{qj} f(R^j z)$. This is
the component of $f$ that is multiplied by $\omega^q$ when one turns by
$2\bpi$ around $\sigma$. We have $f = \sum f_q$ by construction, and each
$f_q$ is $C^\infty$, compactly supported, and satisfies
$\norm{f_q}_{\boB^{-k_h, k_v}} \leq 1$ since this is the case for $f$.

The function $f_0$ is constant along the fibers of $\pi$. It will be the
function $f_\sigma$ in the decomposition of $f$. Consider now $q\neq 0$. We
will first work in a chart $U$ sent by $\pi$ on $[-a,a]^2 -[0, \infty)$,
i.e., a chart cut along the positive real axis. When one crosses this axis
from top to bottom, the function $f_q$ is multiplied by $\omega^q$. We will
use the canonical complex coordinates on $U$.

Let us first show the following: for $\phi \in C_c^{k_h}([-a,a])$ and $j \leq
k_v$, one has
\begin{equation}
\label{eq:f_q_restr}
  \abs*{\int_{-a}^0 \phi \cdot L_v^j f_q \dd x} \leq C \norm{\phi}_{C^{k_h}}.
\end{equation}
The interest of this estimate is that $\phi$ is a priori not compactly
supported in $[-a, 0]$, so that this integral can not be controlled directly
using $\norm{f_q}_{-k_h, k_v}$.

For small $y>0$ and $\epsilon \in \{-1, 1\}$, the interval $[-a,a]+\epsilon
\ic y$ is included in $U$. Therefore,
\begin{equation}
\label{eq:iouiouwxiocvu}
  \abs*{\int_{-a}^a \phi(x) f_q(x+\epsilon \ic y) \dd x} \leq \norm{\phi}_{C^{k_h}}.
\end{equation}
Let $y$ tend to $0$. For $x \leq 0$, $f_q(x+\epsilon \ic y)$ tends to
$f_q(x)$. On the other hand, for $x>0$, the limit depends on $\epsilon$: one
gets $f_q(x^+)$ for $\epsilon=1$ and $f_q(x^-) = \omega^q f_q(x^+)$ for
$\epsilon = -1$. Hence,
\begin{equation*}
  \int_{-a}^a \phi(x) f_q(x+\ic y) \dd x - \omega^{-q}  \int_{-a}^a \phi(x) f_q(x-\ic y) \dd x
  \to (1-\omega^{-q}) \int_{-a}^0 \phi(x) f_q(x) \dd x.
\end{equation*}
Combined with the control~\eqref{eq:iouiouwxiocvu}, this
proves~\eqref{eq:f_q_restr} for $j=0$ (for $C=2/\abs{1-\omega^{-q}}$). The
argument is the same for $j>0$.

Consider a $C^\infty$ function $\rho_2$ which is equal to $1$ on $[-a/2,
a/2]^2$ and vanishes outside of $[-a,a]^2$. We define a function $f_U$ on $U$
by $f_U(x+\ic y) = 1_{x\leq 0} \rho_2(x+\ic y) \sum_{j\leq k_v} y^j L_v^j
f_q(x)$. This is a $C^\infty$ function, compactly supported in $M-\Sigma$ (we
recall that $f$, and therefore $f_q$, vanishes in a neighborhood of $\sigma$,
so that $f_q(x)=0$ for $x$ close to $0$ in the chart $U$). This function is
supported by $U$. Its interest is that its germ along $[-a,0]$ is the same as
that of $f_q$. Moreover, it follows from~\eqref{eq:f_q_restr} that the norm
of $f_U$ in $\boB^{-k_h, k_v}$ is uniformly bounded. This function is
supported in the left half-plane $H \in \boH$ contained in $U$. Let us denote
it by $f_{q, H}$. It will be part of the term $f_H$ in the
decomposition~\eqref{eq:wpuxcvopiuwxopicv}.

For each horizontal segment $\tau$ coming out of the singularity $\sigma$,
one can consider a chart $U$ as above cut along $\tau$ (with the difference
that $[-a,a]^2$ can be cut along either the positive real axis, or the
negative real axis, depending on $\tau$), and then the associated function
$f_U$. Let $\tilde f_q = f_q -\sum_U f_U$. This function is bounded by a
constant in $\boB^{-k_h, k_v}$. Its interest is that it vanishes along every
horizontal segment coming out of $\sigma$, and moreover all its vertical
derivatives up to order $k_v$ also vanish there. In particular, the
restriction of $\tilde f_q$ to any upper half-plane or lower half-plane $H\in
\boH$ is still $C^{k_v}$ and it can be extended to the rest of the manifold
by zero. Denote this extended function by $f_{q, H}$. It belongs to
$\boB^{-k_h, k_v}$ and has a bounded norm in this space, and it is supported
in $H$.

Finally, the decomposition~\eqref{eq:wpuxcvopiuwxopicv} of $f$ is obtained by
letting $f_\sigma=f_0$ and $f_H = \sum_{q\neq 0} f_{q,H}$.
\end{proof}

\begin{proof}[Proof of Proposition~\ref{prop:dualite}]
Decomposing $f$ as in Lemma~\ref{prop:dualite}, it suffices to show the
inequality $\int f g \dLeb \leq C \leq C \norm{f}_{\boB^{-k_h, k_v}} \cdot
\norm{g}_{\check \boB^{\check k_h, - \check k_v}}$ when $f$ is:
\begin{enumerate}
\item supported away from the singularities,
\item or supported on a small neighborhood of a singularity, and constant
    on the fibers of the covering projection,
\item or supported in a half-plane close to a singularity.
\end{enumerate}
For definiteness, we will also assume $\norm{f}_{\boB^{-k_h, k_v}} \leq 1$
and $\norm{g}_{\check \boB^{\check k_h, - \check k_v}} \leq 1$.

Let us first handle the case where $f$ is supported in a small rectangle
$[-a, a]^2$ away from the singularities. We can even assume that $f$ is
supported in $[-a/4, a/4]^2$. Multiplying $g$ by a cutoff function, we can
assume that it is also supported in $[-a/2, a/2]^2$.

Using a local chart, we may work in $\C$. Along the horizontal interval
$[-a,a]+\ic y$, the successive primitives of $F_0=f$ vanishing at $-a+\ic y$
are given by
\begin{equation}
\label{eq:formule_Fk}
  F_k(x+\ic y) = \int_{-a}^x f(t+\ic y) (x-t)^{k-1} /(k-1)! \dd t,
\end{equation}
as one checks easily by induction over $k$. Let us take $k = k_h+2$. With $k$
integrations by parts, one gets
\begin{equation}
\label{eq:fg_IPP}
  \int_{[-a,a]+\ic y} f g \dd x = (-1)^{k} \int_{[-a,a]+\ic y} F_{k}\cdot L_h^{k} g \dd x.
\end{equation}
Let us consider a function $\rho(x)$ equal to $1$ for $x\geq -a/2$ and
vanishing on a neighborhood of $-a$. As $f$ is supported by $[-a/2, a/2]^2$,
one has
\begin{equation*}
  F_{k}(x+\ic y) = \int_{-a}^a f(t+\ic y)\cdot \rho(t) 1_{t\leq x} (x-t)^{k-1} /(k-1)!  \dd t.
\end{equation*}
The function
\begin{equation}
\label{eq:fct_discontinue}
  t\mapsto \rho(t) 1_{t\leq x} (x-t)^{k-1} /(k-1)!
\end{equation}
is of class $C^{k-2}$ on $[-a,a]$, with a bounded $C^{k-2}$ norm: Its
singularity at $x$ is a zero of order $k-1$ to the left of $x$, and of
infinite order to the right of $x$, so that everything matches in $C^{k-2}$
topology. Therefore, by the definition of $\boB^{-k_h, k_v}$ and the choice
$k=k_h+2$, one has $\abs{F_k(x+\ic y)} \leq C$ as $\norm{f}_{\boB^{-k_h,
k_v}} \leq 1$. In the same way, the vertical derivatives of $F_k$ involve
vertical derivatives of $f$, which can be integrated against $C^{k_h}$
functions along horizontals. We get, for all $j\leq k_v$ and all $x+\ic y \in
[-a,a]^2$, the inequality $\abs*{L_v^j F_k(x+\ic y)} \leq C$. Therefore,
along any vertical segment of the form $x+\ic[-a,a]$, the function $F_k$ is
$C^{k_v}$ with bounded norm, and it is compactly supported as it vanishes for
$\abs{y} \geq a/2$ (as $f$ is supported by $[-a/2, a/2]^2$).

Let us integrate the equality~\eqref{eq:fg_IPP} with respect to $y$. We get
\begin{equation}
\label{eq:integrale_2d}
  \int f g \dLeb = (-1)^k \int_{x \in [-a,a]} \pare*{ \int_{x+\ic [-a,a]} F_k \cdot L_h^k g \dd y} \dd x.
\end{equation}
When $x$ is fixed, every integral $\int_{x+\ic [-a,a]} F_k \cdot L_h^k g \dd
y$ is the integral against a $C^{k_v}$ function with bounded norm of the
function $L_h^k g$, with $k\leq \check k_h$ and $k_v \geq \check k_v$ by
assumption. By definition, this integral is bounded by $\norm{F_k}_{C^{k_v}}
\norm{g}_{\check \boB^{\check k_h, - \check k_v}} \leq C$. Integrating in
$x$, we obtain the desired inequality $\abs*{\int fg \dLeb} \leq C$.

\medskip

We still have to consider the case where $f$ is supported in the neighborhood
of a singularity $\sigma$ with angle $2\bpi\kappa$. Multiplying $g$ by a
cutoff function, we can assume that $g$ is also supported there. Write $\pi$
for the corresponding covering projection, sending $\sigma$ to $0$. We may
assume that $\pi^{-1}([-a,a]^2)$ only contains $\sigma$ as a singularity, and
that $f$ and $g$ are supported by $\pi^{-1}([-a/2, a/2]^2)$. We would like to
carry out the same argument as before, but the function $F_k$ one obtains by
integrating along a horizontal line is smooth along vertical lines to the
left of the singularity, but it is discontinuous on vertical lines on the
right of the singularity, breaking the argument.

Assume first that $f$ is invariant under the covering projection $\pi$.
Denote by $\omega = e^{2\ic\bpi/\kappa}$ the fundamental $\kappa$-th root of
unity. Let $R$ be the rotation by $2\bpi$ around $\sigma$. For $q \in
\Z/\kappa \Z$, let $g_q(z) = \kappa^{-1} \sum_{j=0}^{\kappa-1} \omega^{qj}
g(R^j z)$. This is the component of $g$ that is multiplied by $\omega^q$ when
one turns by $2\bpi$ around $\sigma$. For $q\neq 0$, the function $f g_q$ is
multiplied by $\omega^q$ when one turns around the singularity. Therefore,
$\int f g_q \dLeb = \omega^q \int f g_q \dLeb$, which implies $\int f g_q
\dLeb = 0$ (this is just the classical fact that two functions living in
different irreducible representations are orthogonal). Let us now handle
$g_0$. The functions $f$ and $g_0$ are both $R$-invariant. They can be
written as $\tilde f\circ \pi$ and $\tilde g\circ \pi$ where $\tilde f$ and
$\tilde g$ are functions on $\C$ supported by $[-a/2, a/2]^2$. The norms of
these functions (in $\boB^{-k_h, k_v}$ and $\check \boB^{\check k_h, -\check
k_v}$ respectively) are bounded by $1$. The case of functions away from
singularities, that we have already treated, shows that $\abs*{\int \tilde f
\tilde g \dLeb}\leq C$. This gives the same estimate for $\int f g_0 \dLeb$.

\medskip

Assume now that $f$ is supported in a vertical half-plane $H$, to the left of
$\sigma$ for instance. Let us show that
\begin{equation} \label{eq:intfUg_main}
  \abs*{\int f g \dLeb} \leq C.
\end{equation}
We proceed like in the proof away from singularities, making integrations by
parts along horizontals. Let $F_j$ be the $j$-th primitive of $f$ along
horizontals, vanishing at $-a+\ic y$. It is given by the
formula~\eqref{eq:formule_Fk}. Then, we do $k=k_h+2$ integrations by parts
along each horizontal line, to get
\begin{equation*}
  \int_{[-a,0]+\ic y} f g \dd x = (-1)^{k} \int_{[-a,0]+\ic y} F_{k}\cdot L_h^{k} g \dd x + \sum_{j<k} (-1)^j F_{j+1}(\ic y) L_h^j g(\ic y).
\end{equation*}
The difference with~\eqref{eq:fg_IPP} is the boundary terms, due to the fact
that $g$ does not vanish on the line $x=0$. Integrating in $y$, we obtain
\begin{equation}
\label{eq:intfUg}
  \int f g \dLeb = (-1)^k \int_{x \in [-a, 0]} \pare*{\int_{x + \ic [-a, a]} F_k \cdot L_h^k g \dd y} \dd x
  + \sum_{j<k} (-1)^j \pare*{\int_{\ic [-a,a]} F_{j+1} \cdot L_h^j g \dd y}.
\end{equation}
The first term is controlled as in the case away from singularities, as the
function $F_k$ is bounded and $C^{k_v}$ along vertical segments since
$k=k_h+2$. On the other hand, the boundary terms are more delicate. The
difficulty is that, a priori, $F_{j+1}(\ic y)$ is not bounded just in terms
of $\norm{f}_{\boB^{-k_h, k_v}}$: The function~\eqref{eq:fct_discontinue}
(with $k$ replaced by $j$ and $x=0$) is not $C^{k_h}$ for $j<k$ because of
its singularity at $0$. Nevertheless, as the distribution $f$ is supported in
$H$, we may replace the function in~\eqref{eq:fct_discontinue} by another
function which coincides with it on $[-a, 0]$ and is $C^{k_h}$ with bounded
norm on $[-a,a]$, without changing the value of the integral. It follows that
in fact $F_j(\ic y)$ is bounded in terms of $\norm{f}_{\boB^{-k_h, k_v}}$. In
the same way, its vertical derivatives are also bounded.  As $g \in
\check\boB^{\check k_h, -\check k_v}$ has norm at most $1$, we obtain
(integrating on a segment with horizontal coordinate $-x$ with $x$ small to
avoid the singularity)
\begin{equation*}
  \pare*{\int_{-a}^a F_{j+1}(\ic y) \cdot L_h^j g(-x + \ic y) \dd y} \leq C.
\end{equation*}
Letting $x$ tend to $0$, we obtain that the second term in~\eqref{eq:intfUg}
is uniformly bounded. This proves~\eqref{eq:intfUg_main}.

\medskip

Finally, assume that $f$ is supported in a horizontal half-plane $H$, for
instance an upper half plane above $\sigma$. We proceed exactly as in the
case without singularities, integrating by parts along horizontal segments.
Let $F_j$ be the $j$-th primitive of $f$ that vanishes on $-a + \ic (0,a]$.
The only difference is at the end of the argument: the analog
of~\eqref{eq:integrale_2d} in our case is
\begin{equation*}
  \int_H f \cdot  g \dLeb = (-1)^k \int_{x \in [-a,a]} \pare*{ \int_{x+\ic (0,a]} F_k \cdot L_h^k g \dd y} \dd x.
\end{equation*}
The function $F_k$ is still smooth along vertical segments, with uniformly
bounded derivatives. However, it is not compactly supported in $x +
\ic[0,a]$, which prevents us from writing.
\begin{equation}
\label{eq:iuwxcviouopiwxcv}
  \abs*{\int_{x+\ic (0,a]} F_k \cdot L_h^k g \dd y} \leq C \norm{g}_{\check\boB^{\check k_h, -\check k_v}}.
\end{equation}
On the other hand, $F_k$ vanishes on $[-a,a]$, as well as its successive
derivatives. Indeed, $f$ is supported in $H$ and smooth vertically, so by
approximating the left and half parts of the boundary of $H$ from below one
obtains this vanishing property. Therefore, we may extend $F_k$ by $0$ for
points with negative imaginary part. This extension is still $C^{k_v}$ along
vertical lines. This justifies the inequality~\eqref{eq:iuwxcviouopiwxcv}.
Integrating in $x$, we obtain the desired inequality $\abs*{\int f \cdot g
\dLeb}\leq C$.
\end{proof}

\begin{lem}
\label{lem:dualite} We have the following duality formulas for $f \in
\boB^{-k_h, k_v}$ and $g \in \check \boB^{\check k_h, -\check k_v}$:
\begin{equation}
\label{eq:dualite}
  \langle L_h f, g \rangle = - \langle f, L_h g\rangle, \quad \langle L_v f, g \rangle = - \langle f, L_v g \rangle.
\end{equation}
\end{lem}
\begin{proof}
It is enough to check these formulas for functions in $C^\infty_c(M-\Sigma)$,
as they extend by density to the whole spaces thanks to
Proposition~\ref{prop:dualite}. The function $fg$ vanishes on a neighborhood
of the singularities. Denote by $\Omega$ the complement of a union of small
disks around the singularities such that $fg = 0$ outside of $\Omega$. We
have
\begin{equation*}
  \int_M L_h(fg) \dd x \wedge \dd y = \int_\Omega \dd(fg \dd y) = -\int_{\partial \Omega} fg\dd y = 0.
\end{equation*}
Hence, $\int L_h f \cdot g \dLeb + \int f \cdot L_h g \dLeb = 0$. This proves
the first identity in~\eqref{eq:dualite}. The second one is identical, upon
exchanging the roles of $x$ and $y$.
\end{proof}

\subsection{Cohomological interpretation}
\label{subsec:cohomological}

In the study of the Ruelle spectrum of pseudo-Anosov maps, a special role
will be played by the elements of $\boB^{-k_h, k_v} \cap \ker L_v$.
Heuristically, the relevant object associated to $f \in \boB^{-k_h, k_v}$ is
the current $f \dd x$. When $f$ satisfies additionally $L_v f = 0$, then the
formal derivative of this current is $\dd (f \dd x)=(\partial_x f \dd x +
\partial_y f \dd y) \wedge \dd x$. The term $\dd x \wedge \dd x$ vanishes.
When $L_v f = 0$, one has $\partial_y f = 0$, and one gets $\dd (f \dd x)=0$.
Therefore, the current $f \dd x$ is closed. It defines a cohomology class in
$H^1(M-\Sigma)$. We will give a more explicit description of this cohomology
class, and show that it even belongs to $H^1(M)$ (i.e., it vanishes if one
integrates it along a small path around a singularity).

Let $\gamma$ be a continuous closed loop in $M-\Sigma$ and let $f \in
\boB^{-k_h, k_v} \cap \ker L_v$. We define the integral of $f$ along
$\gamma$, denoted by $\int_\gamma f \dd x$, as follows. Deforming $\gamma$
slightly, we can first transform it into a loop made of finitely many
horizontal and vertical segments. In $\int_\gamma f \dd x$, the vertical
components of $\gamma$ do not appear. For a horizontal component $I$, we
would like it to contribute by $\int_I f \dd x$, but this does not make sense
since $f$ can only be integrated against smooth functions, which is not the
case for the characteristic function of $I$. Let us smoothen this function by
adding to the end of $I$ a smooth function going from $1$ to $0$. In the next
horizontal interval $J$, that follows $I$ in $\gamma$, on the contrary, we
subtract $\phi$ (pushed forward by the vertical translation from $I$ to $J$)
to the characteristic function $\chi_J$ of $J$ -- this process changes it to
the function $\chi_J - \phi$, which is smooth. In this way, we obtain
integrals that are well defined. As $f$ is invariant under vertical holonomy
by the assumption $L_v f = 0$, it follows that the result is independent of
the choice of $\phi$, and of the choice of the initial deformation of
$\gamma$ in $M-\Sigma$. This concludes the definition of $\int_\gamma f \dd
x$. This construction is reminiscent
of~\cite[Paragraph~1.3]{bufetov_translation_flow}, although the fact that our
distributions can not be integrated against characteristic functions enforces
an additional smoothing step in the definition above.

\begin{prop}
\label{prop:define_H1} Let $f \in \boB^{-k_h, k_v} \cap \ker L_v$. Then the
integral $\int_\gamma f \dd x$ only depends on the homology class of $\gamma$
in $H_1(M)$. Therefore, the map $\gamma \mapsto \int_\gamma f\dd x$ defines a
linear map from $H_1(M)$ to $\R$, i.e., a cohomology class in $H^1(M)$ which
we denote by $[f]$ or $[f\dd x]$.
\end{prop}
\begin{proof}
The fact that $\int_\gamma f \dd x$ only depends on the homology class of
$\gamma$ in $M-\Sigma$ follows directly from the definitions. The only
assertion that remains to be checked is that this integral is not modified
when one crosses a singularity. Equivalently, we have to show that
$\int_\gamma f \dd x = 0$ when $\gamma$ is a positive path around a
singularity $\sigma$.

Let $\pi$ be the covering projection around $\sigma$, well defined on a
neighborhood of size $\delta \in (0, \beta_0/10)$. Let us fix a function
$\phi$ on $\R$ equal to $1$ around $0$, with support included in
$[-\delta,\delta]$. For $y>0$, we may construct a path $\gamma$ around
$\sigma$ by considering $I_y^+ = \pi^{-1}([-\delta, \delta] + \ic y)$ (a
union of $\kappa$ horizontal segments, where $\kappa$ is the degree of
$\sigma$), crossed negatively, and $I_y^- = \pi^{-1}([-\delta, \delta] - \ic
y)$ (a union of $\kappa$ horizontal segments), crossed positively, as well as
the corresponding vertical segments. Then
\begin{equation}
\label{eq:int_gamma_f}
  \int_\gamma f \dd x
  = \int_{I_y^-} \phi(x) f \dd x - \int_{I_y^+} \phi(x) f \dd x
\end{equation}
for any $y>0$, by definition.

Let $\epsilon>0$. By definition of $\boB^{-k_h, k_v}$, we may choose $g \in
C_c^\infty(M-\Sigma)$ with $\norm{f-g}_{-k_h, k_v} < \epsilon$. When $y$
tends to $0$, we have $\int_{I_y^-} \phi g \dd x - \int_{I_y^+} \phi g \dd x
\to 0$ as the horizontal segments compensate each other, and the singularity
does not contribute as $g$ vanishes close to $\sigma$. We can in particular
choose $y$ for which this quantity is less than $\epsilon$. We have
\begin{equation*}
  \abs*{\int_{I_y^-} \phi g \dd x - \int_{I_y^-} \phi f \dd x} \leq \kappa \norm{\phi}_{C^{k_h}} \norm{g - f}_{-k_h, k_v}
  \leq C \epsilon,
\end{equation*}
as the integral along each of the $\kappa$ horizontal segments composing
$I_y^-$ is bounded by $\norm{\phi}_{C^{k_h}} \norm{g - f}_{-k_h, k_v}$. The
same holds on $I_y^+$. Finally, we get $\abs*{\int_{I_y^-} \phi f \dd x -
\int_{I_y^+} \phi f \dd x} \leq (2C+1)\epsilon$. This concludes the proof
thanks to~\eqref{eq:int_gamma_f}.
\end{proof}

By definition of cohomology, a closed current of degree $1$ vanishes in
cohomology if and only if it is the differential of a current of degree $0$.
In the case of currents in $\boB^{-k_h, k_v} \cap \ker L_v$, we will see that
this primitive is of the same type in the next proposition. The primitive of
the current $f \dd x$ is obtained by integrating $f$ along horizontal leaves.
We will have to see that this makes sense, and that the primitive thus
defined has all the required regularity properties. Equivalently, the
primitive $g$ has to satisfy $L_h g = f$.

\begin{prop}
\label{prop:integre_Lu} Assume that there is no horizontal saddle connection.
Consider $f \in \boB^{-k_h, k_v} \cap \ker L_v$ such that $[f] = 0 \in
H^1(M)$, with $k_h>0$. Then there exists $g \in \boB^{-k_h+1, k_v} \cap \ker
L_v$ such that $f = L_h g$.
\end{prop}
\begin{proof}
Let $x_0$ be a basepoint, and $F$ a horizontal half-line starting at $x_0$,
positively oriented, which does not end at a singularity. Since we assume
there is no horizontal saddle connection, it is dense. We identify it with
$[0,\infty)$. We will denote by $x_t$ the point of $F$ at horizontal distance
$t$ of $x_0$. Choose on $F$ a function $\rho_0$ equal to $1$ in a
neighborhood of $x_0$, and to $0$ on $[\delta/2, +\infty)$, where $\delta$ is
small enough that there is no singularity in the ball of radius $10\delta$
around $x_0$.

Let $\phi$ be a $C^{k_h-1}$ function on $F$ with compact support. Let $\Phi$
be its unique primitive that vanishes at $x_0$. It is constant after some
time $T$, equal to $\int \phi$. Choose a time $t>T$ such that $x_t$ belongs
to the vertical segment of size $\delta$ through $x_0$ (it exists as the
half-line $F$ is dense). Consider then the function $\Phi_t$ equal to $\Phi$
on $[0, t]$, to $(\int \phi) \cdot \rho_0$ on $[t, t+\delta]$ (where $\rho_0$
is pushed vertically to $[x_t, x_{t+\delta}]$), and to $0$ further on. This
is a function of class $C^{k_h}$ with compact support in $F$, so that $\int_F
\Phi_t f \dd x$ is well defined. Then we define formally an object $g$ by the
formula
\begin{equation}
\label{eq:ipwuxcviopuw}
  \int \phi\cdot g \dd x = -\int \Phi_t \cdot f \dd x.
\end{equation}
Let us first notice that this quantity does not depend on $t$. Indeed, if we
choose another time $s>t$ such that $x_s$ also belongs to the vertical
segment of size $\delta$ through $x_0$, then the difference between these two
quantities is given by $(\int \phi) \int_\gamma f \dd x$, where $\gamma$ is
the union of the piece of $F$ between $x_t$ and $x_s$, and a subsegment of
the vertical segment through $x_0$. As $[f] = 0$, this integral vanishes.
Note that, for now, $g$ is only a distribution along $F$.

The interest of this definition is the following. If we prove that $g$
defines a genuine element of $\boB^{-k_h+1, k_v}$, we will have by definition
of $L_h$ that, for any function $\phi$ with compact support on a segment $I
\subseteq F$,
\begin{equation*}
  \int_I \phi \cdot L_h g \dd x = -\int_I \phi' \cdot g \dd x
  =\int_I \Psi_t \cdot f \dd x,
\end{equation*}
where $\Psi_t$ is the primitive of $\phi'$ vanishing at $x_0$, extended to
the right by $(\int\phi') \rho_0 = 0$. Hence, $\Psi_t = \phi$. This formula
shows that $L_h g = f$, at least along subintervals of $F$. As we will see
later that $g$ is invariant under vertical holonomy, we will obtain $L_h g =
f$ everywhere, as desired.

The same argument using $[f] = 0$ shows that, if two segments $I$ and $J$ of
$F$ are obtained one from the other by a vertical translation in a small
chart without singularity, and if $\phi_I$ is a function on $I$, then $\int_I
\phi_I g \dd x = \int_J \phi_J g \dd x$, where $\phi_J$ is the push-forward
to $J$ of $\phi_I$ by vertical translation. This makes it possible to define
$\int_I \phi g \dd x$ for any horizontal segment $I$ and any $\phi \in
C^{k_h-1}_c(I)$, by using the integral on a small vertical translate of $I$
included in $F$. By the above, it does not depend on the choice of the
translate.

Let $\delta>0$ be such that any horizontal segment of length $\beta_0$ can be
translated vertically, in the positive or negative direction, by at least
$\delta$. If $T$ is large enough, then $F[0,T]$ is $\delta$-dense in $M$.
This implies that, to compute $\int_I \phi g \dd x$ for any interval $I$ of
length $\beta_0$, one can first translate it vertically to reduce the
computation to an interval included in $F[0,T+\beta_0]$, and then use a time
$t$ independent of $I$. The function $\Phi_t$ obtained in this way has a
$C^{k_h}$ norm which is bounded by $C\norm{\phi}_{C^{k_h-1}}$. This shows
that, uniformly in $I \in \boI^h$,
\begin{equation*}
  \abs*{\int_I \phi \cdot g \dd x} \leq C \norm{\phi}_{C^{k_h-1}}.
\end{equation*}
Moreover, as $g$ is locally invariant under vertical translations, we have
$\int_I \phi \cdot L_v^j g \dd x = 0$ for all $j>0$. Therefore, $g$ satisfies
all the inequalities that are satisfied by the elements of $\boB^{-k_h+1,
k_v}$.

However, this is not enough to conclude that $g$ is indeed an element of
$\boB^{-k_h+1, k_v}$. We should come back to the definition of this space as
the closure of $C^\infty_c(M-\Sigma)$, and show that $g$ is a limit of smooth
functions with compact support. This is the hardest part of the proof, as one
may not regularize $g$ blindly by convolving it with a smooth kernel along
horizontal segments: this fails for segments that hit the singularity. We
prove the statement locally, as one can then extend it using a partition of
unity. We treat the harder case of the neighborhood of a singularity
$\sigma$, the case away from singularities is easier. Let $\pi: U \to \C$ be
the covering projection of a neighborhood $U$ of $\sigma$ in $\C$, sending
$\sigma$ to $0$. We write $U_r = \pi^{-1}([-r, r] + \ic [-r, r])$. Let $a>0$
be small enough. We fix a smooth function $\rho$ that is equal to $1$ on
$U_{4a}$ and vanishes outside of $U_{5a}$.

By assumption, $f$ itself is the limit in $\boB^{-k_h, k_v}$ of a sequence of
functions $f_n \in C^\infty_c(M-\Sigma)$. Let us consider around $\sigma$ the
function $g^0_n$ which is a primitive of $f_n$ along every horizontal
segment, and vanishes on the vertical segments going through $\sigma$. Then
$\rho g^0_n \in C^\infty_c(M-\Sigma)$. However, $g_n^0$ will not converge in
general to $g$, as one has to adjust integration constants. The difficulty is
that, if we adjust the integration constant by considering what happens to
the left of $\sigma$ in complex charts (i.e., on the set of points whose
image under $\pi$ has negative real part), then this integration constant
will behave nicely along vertical segments to the left of $\sigma$, but it
will be discontinuous along vertical segments to the right of $\sigma$. The
converse problem shows up if we fix the integration constant by using what
happens to the right of $\sigma$. The idea will be to have two integration
constants, coming from the left and from the right, and to show that they are
necessarily close.

Let $\eta$ be a nonnegative $C^\infty$ function on $\R$ with support in
$[0,a]$ and with integral $1$. We will write $\eta_t$ for $\eta(\cdot -t)$,
whose support is contained in $[t, t+a]$. Given a point $y$ on a vertical
segment through $\sigma$, we write
\begin{align*}
  c_n^+(y) & = \int_{[6a, 7a] + \ic y} \eta_{6a} g \dd x
             - \int_{[6a, 7a] + \ic y} \eta_{6a} g^0_n \dd x,
  \\
  c_n^-(y) & = \int_{[-7a, -6a] + \ic y} \eta_{-7a} g \dd x
             - \int_{[-7a, -6a] + \ic y} \eta_{-7a} g^0_n \dd x
\end{align*}
(where we used the local complex coordinates given by $\pi$). These functions
are uniformly bounded. As $g$ is invariant under vertical shift and as
$g^0_n$ is $C^\infty$, they are smooth along vertical segments. More
precisely, $c_n^+$ is $C^\infty$ along vertical segments on the right of the
singularity (in the chart $\pi$), while $c_n^-$ is $C^\infty$ along vertical
segments to the left of the singularity.

We claim that, for $y$ as above, for any function $\phi \in C^{k_h-1}_c([-3a,
3a] + \ic y)$ with norm at most $1$, and for any sign $s = \pm$,
\begin{equation}
\label{eq:ineq_gn_g}
  \abs*{\int \rho \phi g \dd x - \int \rho \phi g^0_n \dd x - \pare*{\int \rho \phi} c^s_n(y)}
  \leq C \norm{f-f_n}_{-k_h, k_v},
\end{equation}
where $C$ does not depend on $n$. Let us prove this for $s=+$ for instance.
By density of $F$ and by continuity of all the objects under consideration,
it suffices to prove it if $y\in F$. The function $\rho \phi -(\int \rho
\phi) \eta_{6a}$ has a vanishing integral on $[-3a, 7a] + \ic y$. Its
primitive $\Phi$ vanishing at $-3a + \ic y$ also vanishes at $7a + \ic y$.
The definition of $g$ entails
\begin{equation*}
  \int \pare*{\rho \phi -  \pare*{\int \rho \phi} \eta_{6a}} g = - \int \Phi f.
\end{equation*}
Moreover,
\begin{equation*}
  \int \pare*{\rho \phi -  \pare*{\int \rho \phi} \eta_{6a}} g^0_n = \int \Phi' g^0_n = - \int \Phi (g^0_n)' = -\int \Phi f_n.
\end{equation*}
Taking the difference between these two equations and using the definition of
$c_n^+(y)$ yields
\begin{equation*}
  \int \rho \phi g - \int \rho \phi g^0_n - \pare*{\int \rho \phi} c_n^+(y) = \int \Phi f_n - \int \Phi f.
\end{equation*}
Thanks to the definition of the norm, this proves~\eqref{eq:ineq_gn_g} since
$\Phi$ is $C^{k_h}$ with norm and support uniformly bounded.

Let us now consider a function $\phi$ supported by $[-3a,3a]$ with integral
$1$. We have $\int \rho \phi = 1$ if $\abs{y} \leq 3a$ by definition of
$\rho$. Using the inequalities~\eqref{eq:ineq_gn_g} with the signs $+$ and
$-$ and taking their differences, we get in particular
\begin{equation}
\label{eq:cn_pm_diff}
  \abs{c^+_n(y) - c^-_n(y)} \leq C \norm{f-f_n}_{-k_h, k_v}.
\end{equation}

Let $h_n$ be a smooth function on $\R$ equal to $0$ in a neighborhood of $0$
and to $1$ for $\abs{x} \geq 1/n$. We define $g_n$ by $g_n(x + \ic y) =
g_n^0(x+\ic y) + c_n^{\sgn x}(y) h_n(x)$. This is a $C^\infty$ function on
$U_{3a}$, vanishing in a neighborhood of $\sigma$. Let $\bar\rho$ be a smooth
function equal to $1$ on $U_a$, vanishing outside of $U_{2a}$. Let us show
that $\bar\rho g_n$ converges to $\bar\rho g$ in $\boB^{-k_h+1, k_v}$, to
conclude the proof.

We first control what happens without vertical derivatives. Let $I$ be a
horizontal interval. We may assume that it is close to $\sigma$, at height
$y$ with $\abs{y}<2a$, otherwise $\bar\rho$ vanishes on $I$ and everything is
trivial. Consider also $\phi \in C^{k_h-1}_c(I)$. Then
\begin{multline*}
  \int_I \phi \cdot \bar\rho g \dd x - \int_I \phi \cdot \bar\rho g_n \dd x
  = \int_I \rho \cdot \bar\rho \phi \cdot g \dd x - \int_I \rho \cdot \bar\rho \phi \cdot g_n^0 \dd x
  - \int_I \rho \cdot \bar\rho \phi c_n^{\sgn x}(y) h_n(x) \dd x
  \\ =\pare*{\int_I \rho \cdot \bar\rho \phi} c_n^+(y) - \int_I \rho \cdot \bar\rho \phi \cdot c_n^{\sgn x}(y) h_n(x) \dd x
  + O(\norm{f-f_n}_{-k_h, k_v}),
\end{multline*}
where the first equality comes from the definition of $g_n$, and the second
one from~\eqref{eq:ineq_gn_g}. In the last integral, if one replaces
$c_n^-(y)$ by $c_n^+(y)$, one makes a mistake which is bounded by
$C\norm{f-f_n}_{-k_h, k_v}$, thanks to~\eqref{eq:cn_pm_diff}. We are left
with
\begin{equation*}
  c_n^+(y) \cdot \int_I \rho \cdot \bar\rho \phi \cdot (1-h_n(x)) \dd x + O(\norm{f-f_n}_{-k_h, k_v}).
\end{equation*}
Since $1-h_n$ is supported in an interval of length $2/n$ and since the
function $\rho \cdot \bar\rho \phi$ is uniformly bounded, as well as $c_n^+$,
this quantity is bounded by $C/n + C \norm{f-f_n}_{-k_h, k_v}$, which tends
to $0$ with $n$. We have therefore proved that $\norm{\bar\rho g_n - \bar
\rho g}_{-k_h+1, 0} \to 0$.

Let us then consider what happens with successive derivatives in the vertical
direction. In $L_v^j(\bar \rho g)$, if one differentiates $\bar \rho$, then
the number of derivatives of $g$ is less than $j$, and one concludes by
induction. We are left with proving the convergence to $0$ of
\begin{equation*}
  \int_I \phi \cdot \bar\rho L_v^j g \dd x - \int_I \phi \cdot \bar\rho L_v^j g_n \dd x.
\end{equation*}
As the vertical derivative of $g$ vanishes, the first term is $0$. For the
second term, the vertical derivatives of $f_n$, integrated against a smooth
function, are small since they are close to the corresponding term for $f$,
which vanishes as $L_v f = 0$. Integrating horizontally, we deduce that the
vertical derivatives of $g_n^0$ are small in the distributional sense. As a
consequence, the vertical derivatives of $c_n^+$ and $c_n^-$ are also small.
The same is true for the vertical derivatives of $g_n$. This concludes the
proof.
\end{proof}

The following lemma will be very important for us, to show that the
eigenvalue $\lambda^{-1}$ of a pseudo-Anosov map acting on $H^1(M)$ does not
show up in its Ruelle spectrum.

\begin{lem}
\label{lem:Lv_preimage_1} There is no $f \in \boB^{-k_h, k_v} \cap \ker L_v$
with $[f] = [\dd y]$.
\end{lem}
\begin{proof}
We argue by contradiction, assuming that $f \in \boB^{-k_h, k_v} \cap \ker
L_v$ satisfies $[f] = [\dd y]$. Increasing $k_h$ (which only makes the space
larger), we can assume $k_h \geq 1$. Since $f$ is in the kernel of $L_v$, its
vertical smoothness is infinite, so we can also assume $k_v\geq 3$. We claim
that, in this case, there exists $g \in \boB^{-k_h+1, k_v}$ with $L_h g = f$
and $L_v g = 1$.

We follow the construction in Proposition~\ref{prop:integre_Lu} to construct
the primitive $g$ of $f$. Let us use all the notations of the corresponding
proof. In particular, let $F$ be a half-infinite horizontal leaf starting at
a point $x_0$, and let $x_t$ be the point at distance $t$ of $x_0$ in $F$,
and let $\rho_0$ be a function on $F$ which is equal to $1$ on a neighborhood
of $x_0$ and to $0$ on $[\delta/2, +\infty[$, where $\delta$ is small enough.

Let $\phi$ be a $C^{k_h-1}$ function on $F$, with compact support. Denote by
$\Phi$ its primitive that vanishes at $0$. It is eventually constant and
equal to $\int \phi$ after some time $T$. Choose $t>T$ such that $x_t$
belongs to the vertical segment of size $\delta$ through $x_0$ (such a time
exists as the half-leaf $F$ is dense), at a vertical distance $y(x_t)$. Let
us consider the function $\Phi_t$ equal to $\Phi$ on $[0,t]$, to $(\int \phi)
\cdot \rho_0$ on $[t, t+\delta]$ (where $\rho_0$ is pushed vertically to
$[x_t, x_{t+\delta}]$), and to $0$ afterwards. This is a compactly support
$C^{k_h}$ function on $F$. Therefore, $\int_F \Phi_t f \dd x$ is well
defined. Let us define formally
\begin{equation}
\label{eq:puiowxuiopv}
  \int \phi\cdot g \dd x = -\int \Phi_t \cdot f \dd x - y(x_t) \cdot \int \phi.
\end{equation}
The last term is the only difference with~\eqref{eq:ipwuxcviopuw}.

This quantity does not depend on $t$. Indeed, choose $s>t$ such that $x_s$ is
also on the vertical leaf of size $\delta$ through $x_0$. Then
\begin{multline*}
  \pare*{-\int \Phi_s \cdot f \dd x - y(x_s) \cdot \int \phi} - \pare*{-\int \Phi_t \cdot f \dd x - y(x_t) \cdot \int \phi}
  \\ = -\pare*{\int \phi} \pare*{\int_\gamma f \dd x + y(x_s) - y(x_t)},
\end{multline*}
where $\gamma$ is the union of the piece of $F$ between $x_t$ and $x_s$, and
of the small vertical segment between $x_s$ and $x_t$. As $[f] = [\dd y]$, we
have $\int_\gamma f \dd x = y(x_t) - y(x_s)$. Therefore, the above difference
vanishes.

Let $I_0$ be a subsegment of $F$, let $\phi$ be a compactly supported
function on $I_0$, let $I_\epsilon$ be a vertical translate of $I_0$ by a
small parameter $\epsilon$ so that there is no singularity in between and so
that $I_\epsilon$ is also included in $F$. Then we have
\begin{equation}
\label{eq:translate_vertical}
  \int_{I_\epsilon} \phi \cdot g \dd x - \int_{I_0} \phi \cdot g \dd x = \pare*{\int \phi} \epsilon.
\end{equation}
Indeed, let us use in Definition~\eqref{eq:puiowxuiopv} a time $t$ which is
large enough to work as well for $I_0$ and $I_\epsilon$. The difference
between the primitives of $\phi$ on $I_0$ and $I_\epsilon$ is then supported
on the subsegment of $F$ between $I_0$ and $I_\epsilon$, and is equal to
$\int \phi$ except in the boundaries $I_0$ and $I_\epsilon$. We obtain
\begin{equation*}
  \int_{I_\epsilon} \phi \cdot g \dd x - \int_{I_0} \phi \cdot g \dd x = - \pare*{\int \phi} \int_\gamma f \dd x,
\end{equation*}
where $\gamma$ is made of a horizontal piece of $F$ and of the vertical
segment between the left endpoints of $I_\epsilon$ and $I_0$, with length
$\epsilon$. As $[f] = [\dd y]$, we have $\int_\gamma f \dd x = \int_\gamma
\dd y = - \epsilon$. This proves~\eqref{eq:translate_vertical}.

We can then extend by continuity $g$ to all horizontal segments, ensuring
that~\eqref{eq:translate_vertical} is always satisfied. Then, by definition,
$L_v g = 1$ in the distributional sense. It remains to check that $g$ belongs
to $\boB^{-k_h+1, k_v}$. The argument is completely identical to the
corresponding argument in the proof of Proposition~\ref{prop:integre_Lu}.

We have obtained $g \in \boB^{-k_h+1, k_v}$ with $L_v g = 1$. With the
duality from Lemma~\ref{lem:dualite}, we get
\begin{equation*}
  \Leb M = \langle 1, 1\rangle = \langle L_v g, 1 \rangle = - \langle g, L_v 1 \rangle
  = 0.
\end{equation*}
This is a contradiction, concluding the proof of the lemma.
\end{proof}

\section{The Ruelle spectrum of pseudo-Anosov maps with orientable
foliations}

\label{sec:Ruelle_spectrum}

Let $T$ be a pseudo-Anosov map preserving orientations, on a translation
surface $(M,\Sigma)$. This section is devoted to the description of its
Ruelle spectrum, culminating with the proof of
Theorem~\ref{thm:main_preserves_orientations}.

\subsection{Quasi-compactness of the transfer operator}

In this paragraph, we show that the operator $\boT$ of composition with $T$
acts on $\boB^{-k_h, k_v}$, and is quasi-compact with a small essential
spectral radius. Namely:

\begin{thm}
\label{thm:rho_ess} The operator $\boT$ acting on $\boB^{-k_h, k_v}$ has a
spectral radius bounded by $1$, and an essential spectral radius bounded by
$\lambda^{- \min(k_h, k_v)}$.
\end{thm}

The proof will use a Lasota-Yorke inequality given in the next proposition.

\begin{proof}[Proof of Theorem~\ref{thm:rho_ess} assuming Proposition~\ref{prop:LY}]
This follows readily from Hennion's Theorem~\cite{hennion}, from the compact
embedding proposition~\ref{prop:compact_inclusion} and from the Lasota-Yorke
inequality given in Proposition~\ref{prop:LY}.
\end{proof}

\begin{prop}
\label{prop:LY} Let $k_h, l_v \geq 0$. The operator $\boT : f \mapsto f \circ
T$, initially defined for $f \in C^\infty_c(M-\Sigma)$, extends to a
continuous linear operator on $\boB^{-k_h, k_v}$, whose iterates are
uniformly bounded. Moreover, it satisfies the inequality
\begin{equation}
\label{eq:LY}
  \norm{\boT^n f}_{-k_h, k_v} \leq C \lambda^{-\min(k_h, k_v) n} \norm{f}_{-k_h, k_v} + C_n \norm{f}_{-k_h-1, k_v-1},
\end{equation}
where $C$ and $C_n$ are constants that do not depend on $f$. (When $k_v=0$,
the last term should be omitted).
\end{prop}
\begin{proof}
Assume that we can prove the inequality~\eqref{eq:LY} for $f \in
C^\infty_c(M-\Sigma)$. Then, it extends to $\boB^{-k_h, k_v}$ by density, and
proves that $\boT$ acts continuously on this space thanks to the inclusion
$\boB^{-k_h, k_v} \subseteq \boB^{-k_h-1, k_v-1}$.

Let us now prove~\eqref{eq:LY} for smooth $f$. In the course of the proof, we
will also establish the boundedness of the iterates of $\boT$ on $\boB^{-k_h,
k_v}$. First, we estimate the contribution of $\norm{\boT^n f}'_{-k_h, k_v}$
to $\norm{\boT^n f}_{-k_h, k_v}$. Consider $I \in \boI^h$ and $\phi \in
C^{k_h}_c(I)$ with norm at most $1$, and compute
\begin{equation*}
  \int_I \phi \cdot L_v^{k_v}(f \circ T^n) \dd x
  = \lambda^{-k_v n} \int_I \phi \cdot (L_v^{k_v} f) \circ T^n \dd x
  = \lambda^{-k_v n} \cdot \lambda^{-n} \int_{T^n I} \phi \circ T^{-n} \cdot L_v^{k_v}f \dd x.
\end{equation*}
Let us then introduce a partition of unity $\rho_p$ on $T^n I$ into smooth
functions with supports of size $\leq \beta_0$ and bounded intersection
multiplicity. Thus, we decompose $T^n I$ as a union of at most $C \lambda^n$
intervals in $\boI^h$. On each of these intervals, the integral is bounded by
$C \norm{f}'_{-k_h, k_v}$ as the function $\phi \circ T^{-n} \cdot \rho_p$
has a $C^{k_h}$-norm which is uniformly bounded (this is the case for $\phi$
and $\rho_p$, and the map $T^{-n}$ only makes things better as it is a
uniform contraction by $\lambda^{-n}$). Summing over $p$, we get a bound $C
\lambda^{-k_v n} \norm{f}'_{-k_h, k_v}$. Hence,
\begin{equation}
\label{eq:ksprime}
  \norm{\boT^n f}'_{-k_h, k_v} \leq C \lambda^{-k_v n} \norm{f}_{-k_h, k_v}.
\end{equation}

If we use the same argument with a norm involving $j < k_v$ stable
derivatives, we get a weaker gain $\lambda^{-jn}$. Summing over $j$, this
shows that the iterates of $\boT$ are uniformly bounded on $\boB^{-k_h,
k_v}$, but this is not enough to prove~\eqref{eq:LY}. To prove it, we will
take advantage of the expansion in the horizontal direction, which we have
not used yet. We can extend $I$ in one of the two horizontal directions
without meeting a singularity, for instance to its right, to an interval $I'
\in \boI^h_{2\beta_0}$. Let  $\phi_\epsilon = \phi \star \theta_\epsilon$
where $\theta_\epsilon$ is a kernel supported on $[0,\epsilon]$, and
$\epsilon<\beta_0$ is a small parameter that will be chosen later on,
depending on $n$. (If the interval $I$ had been extended to its left, we
would have taken the support of $\theta_\epsilon$ in $[-\epsilon, 0]$). Then
$\phi_\epsilon$ is compactly supported in $I'$ if $\epsilon<\beta_0$, and it
satisfies
\begin{equation}
  \label{eq:phi_epsilon}
  \norm{\phi-\phi_\epsilon}_{C^{k_h-1}} \leq C \epsilon,\quad
  \norm{\phi_\epsilon}_{C^{k_h}} \leq C, \quad
  \norm{\phi_\epsilon}_{C^{k_h+1}} \leq C/\epsilon.
\end{equation}

Let us compute as above, introducing a partition of unity $\rho_p$ on $T^n
I'$. We get
\begin{multline*}
  \int_I \phi \cdot L_v^j(f \circ T^n) \dd x
  = \lambda^{-jn} \cdot \lambda^{-n} \sum_p \int_{I_p} (\phi-\phi_\epsilon) \circ T^{-n} \cdot \rho_p \cdot L_v^j f \dd x
  \\ + \lambda^{-jn} \cdot \lambda^{-n} \sum_p \int_{I_p} \phi_\epsilon \circ T^{-n} \cdot \rho_p \cdot L_v^j f \dd x.
\end{multline*}
In the second sum, the test function $\phi_\epsilon \circ T^{-n} \cdot
\rho_p$ has a $C^{k_h+1}$ norm which is bounded by $C/\epsilon$. As the
number $j$ of derivatives we consider is $<k_v$, we deduce that this term is
bounded by $C \lambda^{-jn} \epsilon^{-1} \norm{f}_{-k_h-1, k_v-1} \leq
C(\epsilon, n) \norm{f}_{-k_h-1, k_v-1}$. In the first sum, the first $k_h-1$
derivatives of $(\phi-\phi_\epsilon) \circ T^{-n}$ are bounded by
$C\epsilon$, as this already holds for $\phi-\phi_\epsilon$
by~\eqref{eq:phi_epsilon}. The $k_h$-th derivative of $\phi-\phi_\epsilon$ is
only bounded by a constant. As $T^{-n}$ contracts by $\lambda^{-n}$, the
$k_h$-th derivative of $(\phi-\phi_\epsilon) \circ T^{-n}$ is therefore
bounded by $C \lambda^{-k_h n}$. Hence, taking $\epsilon = \lambda^{-k_h n}$,
we get $\norm{(\phi-\phi_\epsilon) \circ T^{-n}}_{C^{k_h}} \leq C
\lambda^{-k_h n}$. Multiplying by $\rho_p$ (whose derivatives are all
bounded) and then integrating and summing, we find that the first sum is
bounded by $C \lambda^{-(j+k_h)n} \norm{f}_{-k_h, k_v} \leq C \lambda^{-k_h
n} \norm{f}_{-k_h, k_v}$.

Finally, we have proved that, for $j < k_v$,
\begin{equation*}
  \norm{\boT^n f}'_{-k_h, j} \leq C \lambda^{-k_h n} \norm{f}_{-k_h, k_v} + C_n \norm{f}_{-k_h-1, k_v-1}.
\end{equation*}
Together with the inequality~\eqref{eq:ksprime}, we get the conclusion of the
proposition.
\end{proof}

Theorem~\ref{thm:rho_ess} shows that the spectrum of $\boT$ acting on
$\boB^{-k_h, k_v}$ is discrete in $\{z \st \abs{z} > \lambda^{-\min(k_h,
k_v)}\}$, made of at most countably many eigenvalues which are all discrete
and of finite multiplicity. A priori, the spectrum could depend on the space
$\boB^{-k_h, k_v}$ we consider. However, all these spaces contain the dense
subspace $C^\infty_c(M-\Sigma)$ and they are all continuously embedded in the
distribution space $\boD^\infty(M-\Sigma)$. A theorem of
Baladi-Tsujii~\cite[Lemma~A.1]{bt_zeta} then  ensures that the spectrum (and
even the eigenspaces, considered as subspaces of the space of distributions)
do not depend on the space one considers, if one is beyond the essential
spectral radius. Hence, it makes sense to talk about the~\emph{spectrum} of
$\boT$, independently of the space $\boB^{-k_h, k_v}$. We have proved the
existence of a Ruelle spectrum for $T$ in the sense of
Definition~\ref{defn:Ruelle_spectrum}. To complete the proof of Theorem
~\ref{thm:main_preserves_orientations}, we still have to identify this
spectrum.

For $\alpha \neq 0$, let us denote by $E_\alpha^{(1)}$ the eigenspace
corresponding to the eigenvalue $\alpha$, and by $E_\alpha$ the corresponding
generalized eigenspace (containing the eigenvectors and more generally the
generalized eigenvectors, i.e., such that $(\boT -\alpha I)^k f = 0$ for some
$k > 0$). They are included in $\boB^{-k_h, k_v}$ when $\abs{\alpha}
> \lambda^{-\min(k_h, k_v)}$.

\subsection{Description of the spectrum}

To describe the spectrum, we will rely crucially on the action of the
operators $L_h$ and $L_v$.

\begin{prop}
\label{prop:action_Lv} We have $\boT \circ L_v = \lambda L_v \circ \boT$ on
$C^\infty_c(M-\Sigma)$. This equality still holds on all spaces to which
these operators extend continuously, in particular as operators from
$\boB^{-k_h, k_v}$ to $\boB^{-k_h, k_v-1}$ when $k_v>0$.

In the same way, $\boT \circ L_h = \lambda^{-1} L_h \circ \boT$ on
$C^\infty_c(M-\Sigma)$. This equality still holds on all spaces to which
these operators extend continuously, in particular as operators from
$\boB^{-k_h, k_v}$ to $\boB^{-k_h-1, k_v}$.
\end{prop}
\begin{proof}
We compute: $(\boT \circ L_v)(f) = (L_v f) \circ T$, and $(L_v \circ \boT)(f)
= L_v (f \circ T) = \lambda^{-1} (L_v f)\circ T$ as $T$ contracts by
$\lambda^{-1}$ in the vertical direction. This proves the desired equality
for $L_v$. The argument is the same for $L_h$.
\end{proof}

\begin{cor}
\label{cor:Lu_Ls_Erho} The operator $L_v$ sends $E_\alpha$ to
$E_{\lambda\alpha}$. The operator $L_h$ sends $E_\alpha$ to
$E_{\lambda^{-1}\alpha}$.
\end{cor}
\begin{proof}
A generalized eigendistribution $f$ for $\alpha$ satisfies $(\boT-\alpha I)^k
f= 0$ for large enough $k$. Moreover, we have $(\boT -\lambda \alpha I) \circ
L_v = \lambda L_v \circ (\boT - \alpha I)$ by
Proposition~\ref{prop:action_Lv}. By induction, $(\boT -\lambda \alpha I)^k
\circ L_v = \lambda^k L_v \circ (\boT-\alpha I)^k$. Therefore, $(\boT
-\lambda \alpha I)^k (L_v f) = \lambda^k L_v ( (\boT-\alpha I)^k f) = 0$.
This shows that $L_v$ maps $E_\alpha$ to $E_{\lambda \alpha}$. The argument
is the same for $L_h$.
\end{proof}

\begin{cor}
\label{cor:iter_Lv_0} For $f \in E_\alpha$, we have $L_v^k f = 0$ when $k$ is
large enough, more specifically when $\lambda^k \abs{\alpha}>1$.
\end{cor}
\begin{proof}
We have $L_v^k f \in E_{\lambda^k \alpha}$. This space is trivial if
$\abs{\lambda^k \alpha}>1$ as the iterates of $\boT$ are bounded on
$\boB^{-k_h, k_v}$ by Proposition~\ref{prop:LY}.
\end{proof}

If we start from a nonzero generalized eigendistribution, we can consider the
smallest $k$ such that $L_v^k f = 0$. Then $L_v^{k-1} f$ is a generalized
eigendistribution for $\boT$, and it satisfies $L_v f = 0$. Such elements are
the main building blocks to describe the spectrum of $\boT$. We will take
advantage of the cohomological description of such objects we have given in
Paragraph~\ref{subsec:cohomological} to go further in the description of the
spectrum.

Let us now try to see if any cohomology class can be realized by elements in
$\boB^{-k_h, k_v} \cap \ker L_v$ -- and if the class is a (generalized)
eigenfunction for the action of $T$ on cohomology we will try to realize it
by a (generalized) eigendistribution for $\boT$, for the same eigenvalue.
This is not always possible: if one considers the action of a linear Anosov
matrix on the torus, then the cohomology has dimension $2$, but the spectrum
of $\boT$ is reduced to $\{1\}$: it is not possible to realize in this way
the cohomology class corresponding to the stable foliation. We will see that
this is the only obstruction: all the other eigenvectors in cohomology (which
correspond to eigenvalues in $(\lambda^{-1}, \lambda]$) can be realized.

\begin{thm}
\label{thm:realise_H1} Let $h \in H^1(M)$ be a cohomology class which is a
generalized eigenfunction for the linear action of $T$ on cohomology: we have
$(T^* - \mu)^J h = 0$ for some $J\geq 1$ and some $\mu$ with $\abs{\mu}\in
[\lambda^{-1}, \lambda]$ (where $\mu = \lambda$ if and only if $h$ is a
multiple of the class of the horizontal foliation $\dd x$, and $\mu =
\lambda^{-1}$ if and only if $h$ is a multiple of the class of the vertical
foliation $\dd y$). We assume $\mu \neq \lambda^{-1}$, i.e., we exclude
multiples of $\dd y$.

Then, for $\min(k_h, k_v) \geq 3$, there exists $f \in \boB^{-k_h, k_v} \cap
\ker L_v$ in the generalized eigenspace $E_{\lambda^{-1} \mu}$ whose
cohomology class $[f]$ is equal to $h$. In particular, if $h \neq 0$, the
eigenspace is nontrivial.
\end{thm}
\begin{proof}
Let $\omega$ be a closed $1$-form with compact support in $M-\Sigma$ such
that $[\omega] = h$, i.e., $\int_\gamma \omega = \langle h, \gamma \rangle$
for any closed curve $\gamma$. It is possible to choose such an $\omega$
which vanishes on a neighborhood of $\Sigma$ as part of the long exact
sequence in cohomology reads $H^1_c(M-\Sigma) \to H^1_c(M) \to
H^1_c(\Sigma)$. As the last term is $0$, the previous arrow is onto.

Let us write $\omega = \omega_x \dd x + \omega_y \dd y$ where $\omega_x$ and
$\omega_y$ belong to $C^\infty_c(M-\Sigma)$. Then we have
\begin{equation*}
  (T^n)^* \omega = \lambda^n (\boT^n \omega_x) \dd x + \lambda^{-n} (\boT^n \omega_y) \dd y,
\end{equation*}
as $T$ expands horizontally by $\lambda$ and contracts vertically by
$\lambda$.

Consider a closed path $\gamma$ made of horizontal and vertical segments,
away from the singularities. Denote by $\gamma_t$ the same path but shifted
horizontally by $t$. If $t$ is small enough, it does not meet any singularity
either. Let $\bar\gamma = \int_t \eta(t) \gamma_t$ where $\eta$ is a smooth
function whose support is small enough to ensure that this is well defined.
This integral should be understood in the weak sense, i.e., for any form
$\omega$ the integral of $\omega$ on $\bar\gamma$ is by definition $\int_t
\eta(t) (\int_{\gamma_t} \omega)$. Then $\bar \gamma$ is made of horizontal
segments weighted by a $C^\infty$ compactly supported function -- we denote
this part by $\bar\gamma_h$ -- and of vertical parts that we denote by $\bar
\gamma_v$. Then
\begin{equation*}
  \int_{\bar \gamma_h} (\boT^n \omega_x) \dd x = \lambda^{-n} \int_{\bar\gamma} (T^n)^* \omega -
  \lambda^{-2n} \int_{\bar \gamma_v} (\boT^n \omega_y) \dd y.
\end{equation*}
The last integral is uniformly bounded as $\omega_y$ is a bounded function.
Hence, its contribution is $O(\lambda^{-2n})$. In the first term, as $(T^n)^*
\omega$ is closed, it is equivalent to integrate just on $\gamma$. This only
depends on the homology class $h$ of $\omega$, which is a generalized
eigenvector for $T^*$. By Jordan's decomposition, we may write
\begin{equation*}
  (T^n)^* h = \mu^n \sum_{j < J} n^j h_j,
\end{equation*}
with $h_0 = h$. We get
\begin{equation}
\label{eq:asymp_int_gamma_h}
  \int_{\bar \gamma_h} (\boT^n \omega_x) \dd x = \pare*{\int \eta} \cdot (\lambda^{-1}\mu)^n \sum_{j< J} n^j \langle h_j, \gamma \rangle + O(\lambda^{-2n}).
\end{equation}
In $\boB^{-k_h, k_v}$, we can write
\begin{equation*}
  \boT^n \omega_x =
  \sum_{\abs{r}\geq \lambda^{-2}} \sum_{j \leq C} r^n n^j f_{r,j}
  + R_n,
\end{equation*}
where $r$ runs along the eigenvalues of modulus $\geq \lambda^{-2}$ of
$\boT$, the $f_{r,j}$ belong to $E_r$ and $R_n$ is a remainder term which
decays faster than $\lambda^{-2n}$. Identifying the terms in the
asymptotic~\eqref{eq:asymp_int_gamma_h} thanks to the assumption $\abs{\mu} >
\lambda^{-1}$ and using $h_0= h$, we obtain for $f=f_{\lambda^{-1}\mu,0}$ the
equality
\begin{equation}
\label{eq:bargammah}
  \int_{\bar\gamma_h} f \dd x = \pare*{\int \eta} \langle h, \gamma \rangle.
\end{equation}

Let us show that $f$ satisfies $L_v f = 0$. Consider a horizontal interval
$I_0 = [0, q]$, a small vertical translate $I_\epsilon = I_0+\ic \epsilon$ of
this interval (in a chart away from singularities), and a compactly supported
test function $\phi_0$ on $I_0$. We want to show that $\int_{I_0} \phi_0 f
\dd x = \int_{I_\epsilon} \phi_\epsilon f \dd x$ where $\phi_\epsilon$ is the
vertical push-forward of $\phi_0$ on $I_\epsilon$. To do this, denote by
$\gamma_t$ the path from $0$ to $t$ then to $\ic \epsilon+t$ then to
$\ic\epsilon$ then to $0$. $0$. Let also $\eta(t) = -\phi_0'(t)$. In
$\bar\gamma=\int \eta(t) \gamma_t \dd t$, a point $x \in [0,q]$ is counted
with a weight $\int_{t \in [x,q]} \eta(t) \dd t = -\phi_0(q) + \phi_0(x) =
\phi_0(x)$. One can argue similarly along $I_\epsilon$. Therefore, by
definition, $\int_{I_0} \phi_0 f \dd x - \int_{I_\epsilon} \phi_\epsilon f
\dd x = \int_{\bar \gamma_h} f \dd x$. This integral vanishes
by~\eqref{eq:bargammah} as $\int \eta=0$. This shows that $f$ is invariant
under vertical translation, i.e., $L_v f = 0$.

The cohomology class $[f]$ is then well defined by
Proposition~\ref{prop:define_H1}, as well as $\int_\gamma f \dd x$ for any
closed path. By definition of this integral, it coincides with
$\int_{\bar\gamma_h} f \dd x$ when $\bar\gamma$ is a smoothing of $\gamma$ as
above and $\eta$ has integral $1$. We deduce from~\eqref{eq:bargammah} that
$\int_\gamma f \dd x= \langle h, \gamma \rangle$ for any closed path
$\gamma$. By definition, this shows that $[f] = h$.
\end{proof}

We can use this statement to show that the spectrum of $T$ contains the set
mentioned in Theorem~~\ref{thm:main_preserves_orientations}:

\begin{cor}
\label{cor:spectre_contient} The Ruelle spectrum of $T$ contains all the
$\lambda^{-n} \mu$ for $n\geq 1$ and $\mu \in \Xi$, where $\Xi$ is the
spectrum of $T^*$ on the subspace of $H^1(M)$ made of $1$-forms which are
orthogonal to $\dd x$ and $\dd y$, as in the statement of
Theorem~\ref{thm:main_preserves_orientations}.
\end{cor}
\begin{proof}
Theorem~\ref{thm:realise_H1} ensures that $\lambda^{-1} \mu$ belongs to the
Ruelle spectrum of $T$. The map $L_h$ is injective on the generalized
eigenspace $E_{\lambda^{-1}\mu}$ by Lemma~\ref{lem:Lh_0}, as the kernel of
$L_h$ is included in $E_1$. It sends it to $E_{\lambda^{-2}\mu}$ by
Corollary~\ref{cor:Lu_Ls_Erho}, hence this space is nontrivial. By induction,
one proves in the same way that all the spaces $E_{\lambda^{-n}\mu}$ are
nontrivial.
\end{proof}

\begin{prop}
\label{prop:L_h_surj} For any $\alpha \neq 0$, the operator $L_h$ is onto
from $E_\alpha \cap \ker L_v$ to $E_{\lambda^{-1}\alpha} \cap \ker L_v \cap
\ker [\cdot]$. It is bijective for $\alpha \neq 1$.
\end{prop}
\begin{proof}
First, $L_h$ sends $E_\alpha$ to $E_{\lambda^{-1}\alpha}$ by
Corollary~\ref{cor:Lu_Ls_Erho}. As it commutes with $L_v$, it even sends
$E_\alpha \cap \ker L_v$ to $E_{\lambda^{-1}\alpha} \cap \ker L_v$. Let us
show that its image is contained in $\ker [\cdot]$. Let $f \in \ker L_v$, we
have to see that $[L_h f] = 0$. Consider a path $\gamma$ made of horizontal
and vertical segments. We compute $\int_\gamma L_h f \dd x$ by coming back to
its definition. Informally, we have $\int_\gamma L_h f \dd x = \sum_I \int_I
L_h f \dd x$ where the sum is over horizontal parts of $\gamma$. With an
integration by parts, $\int_\gamma L_h f \dd x = \sum_I (f(y_I) - f(x_I))$
where $y_I$ and $x_I$ are the endpoints of $I$. As $\gamma$ is a closed path
and $f$ is invariant vertically each $f(y_I)$ cancels out with $-f(x_J)$
where $J$ is the horizontal interval following $I$ in $\gamma$. We are left
with $\int_\gamma L_h f \dd x = 0$.

This computation is not rigorous as $f$ can not be integrated against
characteristic functions, and $f(y_I)$ makes no sense ($f$ is only a
distribution). This is why $\int_\gamma L_h f \dd x$ is defined in
Paragraph~\ref{subsec:cohomological} by using a regularization of the
characteristic function of $I$. The above argument works with the
regularization. As $f$ is vertically invariant, the contribution of the end
of the interval $I$ to $\int_\gamma L_h f \dd x$ compensates exactly with the
contribution of the beginning of the next interval, and we are left with
$\int_\gamma L_h f = 0$ as desired.

It remains to show that $L_h : E_\alpha \cap \ker L_v \to
E_{\lambda^{-1}\alpha} \cap \ker L_v \cap \ker [\cdot]$ is surjective (its
bijectivity for $\alpha \neq 1$ follows directly as $L_h$ is injective away
from constants by Lemma~\ref{lem:Lh_0}). Fix $f \in E_{\lambda^{-1}\alpha}
\cap \ker L_v \cap \ker [\cdot]$. By Proposition~\ref{prop:integre_Lu}, if
$k_h$ and $k_v$ are large enough, there exists $g \in \boB^{-k_h+1, k_v}$
such that $L_v g = 0$ and $L_h g = f$. The question is whether one can take
$g \in E_\alpha$.

Consider $j$ such that $(\boT-\lambda^{-1}\alpha)^j f = 0$. We have
$(\boT-\lambda^{-1}\alpha)^j \circ L_h = \lambda^{-j} L_h  \circ
(\boT-\alpha)^j$ by Proposition~\ref{prop:action_Lv}. Therefore, $L_h(
(\boT-\alpha)^j g) = 0$, i.e., there exists a constant $c$ such that
$(\boT-\alpha)^j g = c$ by Lemma~\ref{lem:Lh_0}. If $\alpha \neq 1$, we have
then $(\boT-\alpha)^j (g - c/(1-\alpha)^j) = 0$. Therefore, $\tilde g = g -
c/(1-\alpha)^j$ satisfies $\tilde g \in E_\alpha \cap \ker L_v$ and $L_h
\tilde g = f$, as announced. If $\alpha = 1$, then $(\boT-\alpha)^{j+1} g =
(\boT - 1) c = 0$, so $g$ itself already belongs to $E_\alpha$.
\end{proof}

There are two possible spectral values, corresponding to the eigenvalues
$\lambda$ and $\lambda^{-1}$ of $T^*: H^1(M) \to H^1(M)$, i.e., to $\dd x$
and $\dd y$. They have a special status in
Theorem~\ref{thm:main_preserves_orientations}: the first one is simple and
does not interact with the rest of the spectrum, while the second one does
not belong to the Ruelle spectrum. Let us now give the specific results about
these values that we will need to classify the Ruelle spectrum.

\begin{lem}
\label{lem:E1} The generalized eigenspace $E_1$ is one-dimensional, made of
constants.
\end{lem}
\begin{proof}
The generalized eigenspace $E_1$ contains the constants as the function $1$
belongs to $\boB^{-k_h, k_v}$ by Lemma~\ref{lem:boB_riche}. Moreover, any
element $f$ of $E_1$ satisfies $L_v f = 0$ (as $L_v f$ belongs to
$E_{\lambda}$ by  Corollary~\ref{cor:Lu_Ls_Erho}, and this space is trivial
by Theorem~\ref{thm:rho_ess}). Therefore, there is a linear map $f \mapsto
[f]$ from $E_1$ to $H^1(M)$, taking its values in the generalized eigenspace
for the eigenvalue $\lambda$ of $T^*$. This space has dimension $1$. To
conclude, it suffices to show that this map is injective, i.e., if $f\in E_1$
satisfies $[f] = 0$ then $f$ vanishes. When $[f] = 0$,
Proposition~\ref{prop:L_h_surj} shows that $f$ can be written as $L_h g$ with
$g \in E_{\lambda}$. As this space is trivial, we get $g=0$ and then $f=0$.
\end{proof}

We have almost all the tools to show that the Ruelle spectrum of $T$ is given
exactly by the set described in
Theorem~\ref{thm:main_preserves_orientations}. More precisely, we can already
show the following partial result.

\begin{prop}
\label{prop:partiel} The Ruelle spectrum of $T$ is given exactly by the set
described in Theorem~\ref{thm:main_preserves_orientations}, i.e., it is made
of $1$ and of the numbers $\lambda^{-n} \mu$ with $n\geq 1$ and $\mu \in
\Xi$.
\end{prop}
\begin{proof}
On the one hand, $1$ belongs to the spectrum by Lemma~\ref{lem:E1}. On the
other hand, for $\mu \in \Xi$ and $n\geq 1$, then $E_{\lambda^{-n} \mu}$ is
nontrivial by Corollary~\ref{cor:spectre_contient}. This shows one inclusion
in the proposition.

For the converse, consider $\alpha\neq 0$ such that $E_\alpha$ is nontrivial,
and take a nonzero $f \in E_\alpha$. Let $k\geq 0$ be the integer such that
$L_v^k f \neq 0$ and $L_v^{k+1}f=0$. It exists by
Corollary~\ref{cor:iter_Lv_0}. The function $f_k = L_v^k f$ belongs to
$E_{\lambda^k \alpha}$ by Corollary~\ref{cor:iter_Lv_0}, and to $\ker L_v$ by
construction. If $[f_k] = 0$, Proposition~\ref{prop:L_h_surj} shows that
there exists $f_{k+1} \in E_{\lambda^{k+1}\alpha} \cap \ker L_v$ with $L_h
f_{k+1} = f_k$. If $[f_{k+1}] = 0$, we can iterate the same process. It has
to stop at some point as $E_{\lambda^{k+n} \alpha}$ is trivial for $n$ large.
Therefore, we get an integer $n$ and a distribution $f_{k+n} \in
E_{\lambda^{k+n} \alpha} \cap \ker L_v$ with $L_h^n f_{k+n} = f_k$ and
$[f_{k+n}] \neq 0$. The cohomology class $[f_{k+n}]$ belongs to the
generalized eigenspace for $T^*: H^1(M) \to H^1(M)$ for the eigenvalue
$\alpha' = \lambda^{k+n+1} \alpha$. We have $\alpha' \neq \lambda^{-1}$,
since otherwise the corresponding cohomology class would be a nonzero
multiple of $[\dd y]$, contradicting Lemma~\ref{lem:Lv_preimage_1}. Hence,
$\alpha' \in \Xi$ or $\alpha' = \lambda$. If $\alpha' \in \Xi$, we have
written $\alpha$ as $\lambda^{-p} \alpha'$ with $p\geq 1$, in accordance with
the claim of the proposition. If $\alpha' = \lambda$, then $f_{k+n} \in E_1$.
By Lemma~\ref{lem:E1}, $f_{k+n}$ is constant. As $L_h^n f_{k+n} = f_k \neq
0$, we deduce $n = 0$. Then $L_v^k f = f_k$ is a nonzero constant $c$. Using
the duality formula from Lemma~\ref{lem:dualite}, we get
\begin{equation*}
  c \Leb M = \langle f_k, 1 \rangle = \langle L_v^k f, 1 \rangle
  = - \langle f, L_v^k 1 \rangle.
\end{equation*}
If $k$ were nonzero, then $L_v^k 1$ would vanish and we would get a
contradiction. Therefore, $k=0$. Finally, $\alpha=1$, again in accordance
with the claim.
\end{proof}

The conclusion of the proof of Theorem~\ref{thm:main_preserves_orientations}
relies on the following statement.
\begin{thm}
\label{thm:Lv_integre} Let $\alpha \notin  \{0,1\}$. Then $L_v :
E_{\lambda^{-1}\alpha} \to E_\alpha$ is onto.
\end{thm}
Before proving the theorem, let us show how we can conclude the proof of
Theorem~\ref{thm:main_preserves_orientations}.

\begin{proof}[Proof of Theorem~\ref{thm:main_preserves_orientations} using
Theorem~\ref{thm:Lv_integre}] To simplify the notations, we will assume that
for $\mu \in \Xi$ then $\lambda^{-1} \mu \notin \Xi$ (otherwise, there is a
superposition phenomenon as explained after the statement of
Theorem~\ref{thm:main_preserves_orientations}, which makes things more
complicated to write but does not change anything to the proof).

In Proposition~\ref{prop:partiel}, we have described exactly the spectrum of
$T$, and moreover we have shown how the generalized eigenspaces were
constructed. On the one hand, there is the space $E_1$, which is
one-dimensional by Lemma~\ref{lem:E1}. On the other hand, for $\mu \in \Xi$,
the space $E_{\lambda^{-1}\mu}$ is in bijection with the generalized
eigenspace for the action of $T^*$ on $H^1(M)$ and the eigenvalue $\mu$, with
dimension $d_\mu$.

Finally, $E_{\lambda^{-(n+1)} \mu}$ is made of elements sent by $L_v$ to
$E_{\lambda^{-n} \mu}$, and of elements in $E_{\lambda^{-(n+1)} \mu} \cap
\ker L_v$. Proposition~\ref{prop:L_h_surj} shows that $L_h$ is a bijection
between $E_{\lambda^{-n} \mu} \cap \ker L_v$ and $E_{\lambda^{-(n+1)} \mu}
\cap \ker L_v$ (as, on the second space, the condition $[f] = 0$ is always
satisfied thanks to our non-superposition assumption). Therefore, by
induction, all these spaces have dimension $d_\mu$. As $L_v :
E_{\lambda^{-(n+1)} \mu} \to E_{\lambda^{-n} \mu}$ is onto by
Theorem~\ref{thm:Lv_integre}, we get
\begin{equation*}
  \dim E_{\lambda^{-(n+1)} \mu} =
  \dim E_{\lambda^{-n} \mu} + \dim E_{\lambda^{-(n+1)} \mu} \cap \ker L_v
  = \dim E_{\lambda^{-n} \mu} + d_\mu.
\end{equation*}
By induction, we obtain $\dim E_{\lambda^{-n} \mu} = n d_\mu$. In fact, we
have even proved the flag decomposition expressed
in~\eqref{eq:flag_decomposition}.
\end{proof}

We recall that $L_v$ sends $E_{\lambda^{-1}\alpha}$ to $E_\alpha$ by
Corollary~\ref{cor:Lu_Ls_Erho}. To prove Theorem~\ref{thm:Lv_integre}, the
most natural approach would be to start from an element of $E_\alpha$ with
$\alpha\notin\{0,1\}$ and to construct a preimage under $L_v$, by integrating
along vertical lines as we did in the proof of
Proposition~\ref{prop:integre_Lu}. But we have no cohomological condition to
use, and moreover we only have a distributional object for which the meaning
of vertical integration is not clear. If one thinks about it, the result of
the theorem is even counterintuitive.

Let us try to prove the opposite of Theorem~\ref{thm:Lv_integre}, to see the
subtlety. Assume for instance that $f \in E_\alpha$ is nonzero and satisfies
$L_v f = 0$, and that we can find a vertical primitive $g$ of $f$, i.e., one
has $L_v g = f$. Let us try to prove that $f=0$. We should not succeed (this
would be a contradiction with Theorem~\ref{thm:Lv_integre}), but we will see
that there is a strong nonrigorous argument in favor of the equality $f=0$.
Consider an embedded rectangle with horizontal sides $I_0$ and $I_R$ and very
long vertical sides of length $R$. Fix a smooth compactly supported function
$\phi$ on $I_0$, and push it vertically to $I_R$. We should have $\int_{I_R}
\phi g \dd x - \int_{I_0} \phi g \dd x = R \int_{I_0} \phi f \dd x$. As the
left hand side is bounded, we obtain
\begin{equation*}
  \int_{I_0} \phi f \dd x = O( \norm{\phi}_{C^{k_h}} / R).
\end{equation*}
Letting $R$ tend to infinity, we can almost deduce that $f$ vanishes, except
that this argument is not correct as one can not take $R$ arbitrarily large
because of the singularities. If one tries to cut $I_0$ into smaller pieces
for which one can increase $R$, then we will use a partition of unity with a
large $C^{k_h}$ norm, so that we will improve the bound at the level of
$1/R$, but lose at the level of $\norm{\phi}_{C^{k_h}}$. Therefore, we can
not prove in this way that $f$ vanishes, so there is hope that
Theorem~\ref{thm:Lv_integre} is true. But this shows that this theorem is
non-trivial, and follows from a subtle balance.

The proof we will give of Theorem~\ref{thm:Lv_integre} will not follow the
constructive approach we sketched above. Instead, it will follow from an
indirect duality argument: we will show that the adjoint of $L_v$ is
injective. To do this, let us define the operator $\check \boT$ which extends
to $\check \boB^{\check k_h, -\check k_v}$ the operator $f \mapsto f \circ
T^{-1}$ initially defined on $C^\infty_c(M-\Sigma)$. As $T^{-1}$ is a
pseudo-Anosov map, all the results of the previous paragraphs apply to
$\check\boT$. In particular, one can talk about its Ruelle spectrum. We will
write $\check E_\alpha$ for the generalized eigenspace of $\check\boT$
associated to the eigenvalue $\alpha$, on any space $\check\boB^{\check k_h,
- \check k_v}$ with $\abs{\alpha} > \lambda^{-\min(\check k_h, \check k_v)}$.

From this point on, we will only consider non-negative integers $k_h$, $k_v$, $\check k_h$
and $\check k_v$ that satisfy the conditions of the duality
Proposition~\ref{prop:dualite}, i.e., $-k_h+\check k_h \geq 2$ and $k_v -
\check k_v \geq 0$ (or conversely). If we are dealing with an eigenvalue
$\alpha$, we will moreover choose them with $\abs{\alpha} >
\lambda^{-\min(k_h, k_v)}$ and $\abs{\alpha} > \lambda^{-\min(\check k_h,
\check k_v)}$ to ensure that the corresponding generalized eigenspaces for
$\boT$ and $\check\boT$ are included respectively in $\boB^{-k_h, k_v}$ and
$\boB^{\check k_h, -\check k_v}$. This implies in particular that the duality
is well defined on $E_\alpha \times \check E_{\alpha'}$ for all $\alpha,
\alpha' \neq 0$.

In addition to the duality formulas for $L_h$ and $L_v$ given in
Lemma~\ref{lem:dualite}, we will also use the following one: For $f \in
\boB^{-k_h, k_v}$ and $g \in \check \boB^{\check k_h, -\check k_v}$,
\begin{equation}
\label{eq:dualite2}
  \langle \boT f, g \rangle = \langle f, \check \boT g \rangle.
\end{equation}
It follows readily from the definitions and the fact that $T$ preserves
Lebesgue measure.

\begin{lem}
\label{lem:dualite_Ealpha} We have $\langle f, g\rangle = 0$ for $f\in
E_\alpha$ and $g\in \check E_{\alpha'}$ with $\alpha \neq \alpha'$. Moreover,
$(f,g)\mapsto \langle f, g\rangle$ is a perfect duality on $E_\alpha \times
\check E_\alpha$, i.e., it identifies $E_\alpha$ with the dual of $\check
E_\alpha$, and conversely.
\end{lem}
\begin{proof}
Take $f\in E_\alpha$. Then $\boT^n f = \sum_{j\leq J} \alpha^n n^j f_j$ for
some $f_j \in E_\alpha$, with $f_0 = f$. In the same way, for $g \in \check
E_{\alpha'}$, we have $\check \boT^n g = \sum_{j\leq J} (\alpha')^n n^j g_j$
for some $g_j \in \check E_{\alpha'}$ with $g_0 = g$. Using the
duality~\eqref{eq:dualite2}, we obtain for all $n$
\begin{equation*}
  \sum \alpha^n n^j \langle f_j, g \rangle = \langle \boT^n f, g\rangle = \langle f, \check \boT^n g \rangle
  = \sum (\alpha')^n n^j \langle f, g_j \rangle.
\end{equation*}
When $\alpha \neq \alpha'$, one gets by identifying the asymptotics that
$\langle f_j, g \rangle = 0$ for all $j$. In particular, for $j=0$, this
gives $\langle f, g \rangle = 0$ and shows that $E_\alpha$ and $\check
E_{\alpha'}$ are orthogonal.

To prove that there is a perfect duality between $E_\alpha$ and $\check
E_\alpha$, we have to show that the duality is nondegenerate: for any $f \in
E_\alpha$, we have to find $g \in \check E_\alpha$ with $\langle f, g \rangle
\neq 0$ (and conversely, but the argument is the same). As $f$ is a
distribution, there exists a function $h \in C^\infty_c(M-\Sigma)$ with
$\langle f, h\rangle \neq 0$. We think of $h$ as an element of
$\check\boB^{\check k_h, -\check k_v}$, and we write its spectral
decomposition for $\check \boT$: we have $\check\boT^n h = \sum_{i,j}
\alpha_i^n n^j h_{i,j} + O(\epsilon^n)$ where $\epsilon<\abs{\alpha}$ and
$h_{i,j} \in \check E_{\alpha_i}$. As above, using~\eqref{eq:dualite2}, we
find
\begin{equation*}
  \sum \alpha^n n^j \langle f_j, h \rangle = \langle \boT^n f, h\rangle = \langle f, \check \boT^n h \rangle
  = \sum_{i,j}  \alpha_i^n n^j \langle f, h_{i,j}\rangle + O(\epsilon^n).
\end{equation*}
In the sum on the left, there is the term $\alpha^n \langle f_0, h\rangle$
with $\langle f_0,h \rangle = \langle f, h\rangle \neq 0$. Therefore, there
also has to be a term in $\alpha^n$ on the right hand side. This entails that
one of the $\alpha_i$ equals $\alpha$, and the corresponding function $g =
h_{i,0}$ belongs to $\check E_\alpha$ and satisfies $\langle f, g\rangle \neq
0$, as desired.
\end{proof}

\begin{proof}[Proof of Theorem~\ref{thm:Lv_integre}]
Let $\alpha \notin \{0, 1\}$. We want to show that $L_v : E_{\lambda^{-1}
\alpha} \to E_\alpha$ is onto. Equivalently, we want to show that its
adjoint, from $E_\alpha^*$ to $E_{\lambda^{-1} \alpha}^*$, is injective.
These spaces are identified respectively with $\check E_\alpha$ and $\check
E_{\lambda^{-1} \alpha}$ by the duality of Lemma~\ref{lem:dualite_Ealpha},
and the adjoint of $L_v$ is $-L_v$ by~\eqref{eq:dualite}. Hence, it is enough
to show that $L_v : \check E_\alpha \to \check E_{\lambda^{-1} \alpha}$ is
injective. This follows from Lemma~\ref{lem:Lh_0} (we recall that $L_v$ plays
in $\check \boB$ the same role as $L_h$ in $\boB$).
\end{proof}

\section{Vertically invariant distributions}
\label{sec:vertically_invariant}

Let $(M,\Sigma)$ be a translation surface, and $T$ a linear pseudo-Anosov map
on $(M,\Sigma)$, preserving orientations.
Theorem~\ref{thm:main_preserves_orientations} and its proof give a whole set
of distributions which are annihilated by $L_v$. Indeed, this is the case of
the constant distribution, of the distributions in $E_{\lambda^{-1}\mu_i}
\cap \ker L_v$, and of their images under $L_h^n$. These are the only
distributions in $\boB^{-k_h, k_v}$ which are vertically invariant:

\begin{lem}
\label{lem:Lv_invariant_boB} Any distribution in $\boB^{-k_h, k_v} \cap \ker
L_v$ belongs to the linear span of the constant distributions and of the
spaces $L_h^n(E_{\lambda^{-1}\mu_i} \cap \ker L_v)$ for $i=1,\dotsc, 2g-2$
and $n\geq 0$.
\end{lem}
\begin{proof}
This follows from the same inductive strategy used to classify Ruelle
resonances. We show that any $\omega \in \boB^{-k_h, k_v} \cap \ker L_v$
belongs to the space $F$ spanned by the constant distributions and the spaces
$L_h^n(E_{\lambda^{-1}\mu_i} \cap \ker L_v)$ for $i=1,\dotsc, 2g-2$ and
$n\geq 0$, by induction on the order of $\omega$.

The constant distributions and the distributions in $E_{\lambda^{-1}\mu_i}
\cap \ker L_v$ have cohomology classes which span all the classes without any
$[\dd y]$ components, i.e., the orthogonal to $[\dd x]$. Therefore, there
exists $\tilde \omega$ in $F$ such that $[\omega - \tilde \omega]$ is a
multiple of $[\dd y]$. By Lemma~\ref{lem:Lv_preimage_1}, we have in fact
$[\omega - \tilde \omega] = 0$. Therefore, by
Proposition~\ref{prop:integre_Lu}, there exists $\eta \in \boB^{-k_h+1, k_v}
\cap \ker L_v$ (and therefore in $\boB^{-k_h, k_v} \cap \ker L_v$) such that
$\omega-\tilde \omega = L_h \eta$. The order of $\eta$ being strictly smaller
than the order of $\omega$, the induction assumption ensures that $\eta \in
F$. As $F$ is stable under $L_h$, we get $\omega = \tilde \omega + L_h \eta
\in F$.

We should also check the initial step of the induction, when $\omega$ is of
order $0$. With the same construction as above, $\eta$ is a continuous
function. As it is vertically invariant, we deduce that it is constant by
minimality of the vertical flow. In particular, it belongs to $F$, and so
does $\omega$.
\end{proof}

However, there are some distributions that are not seen with this point of
view, as they are not in the closure of $C^\infty_c(M-\Sigma)$. To describe
them, we will follow the same route as above, but replacing our Banach space
$\boB^{-k_h, k_v}$ by an extended space $\boB_{ext}^{-k_h, k_v}$.

We define an element $\omega$ of $\boB_{ext}^{-k_h, k_v}$ to be a family of
distributions $\omega_I$ of order at most $k_h$ on all horizontal segments
$I$ in $\boI^h$, with the following conditions:
\begin{enumerate}
\item Compatibility: if two segments $I, I'\in \boI^h$ intersect, then the
    corresponding distributions coincide on functions supported in $I\cap
    I'$.
\item Smoothness in the vertical direction: for any interval $I \in \boI$,
    and any test function $\phi \in C^{k_h}_c(I)$ with norm at most $1$,
    denote by $I_t$ the vertical translation by $t$ of $I$ for small enough
    $t$, and by $\phi_t$ the vertical push-forward of $\phi$ on $I_t$. Then
    we require that $t\mapsto \int_{I_t} \phi_t \omega_{I_t}$ is $C^{k_v}$,
    with all derivatives bounded by a constant $C$ independent of $I$ or
    $\phi$. The best such $C$ is by definition the norm of $\omega$ in
    $\boB_{ext}^{-k_h, k_v}$.
\item Extension to the singularity: if $(I_t)_{t \in (0,\epsilon]}$ is a
    family of vertical translates of a horizontal segment, parameterized by
    height, such that the limit $I_0$ contains a singularity, then we
    require that $\omega_{I_t}$ and all its $k_v$ vertical derivatives
    extend continuously up to $I_0$.
\end{enumerate}
The first two conditions are very natural, and reproduce directly what we
have imposed in the construction of $\boB^{-k_h, k_v}$ in
Paragraph~\ref{subsec:Bkhkv}. The third condition is to exclude pathological
behaviour such as in the following example. Consider a vertical segment
$\Gamma= (0,\epsilon]$ ending on a singularity at $0$, a function $\rho$ on
$\Gamma$ with support in $[0,\epsilon/2]$ that oscillates like $\sin(1/t)$ at
$0$, and define $\omega_I$ to be equal to $\rho(x_I) \delta_{x_I}$ if $I$
intersects $\Gamma_\sigma$ at a point $x_I$, and $0$ otherwise. Then this
would be an element of our extended space without the third condition. Recall
that $\boB_{ext}^{-k_h, k_v} \neq \boB^{-k_h, k_v}$ (see the example on
Page~\pageref{page_B_ext_not_B}).

With this definition, many of the results of the previous sections extend
readily. We indicate in the next proposition all the results for which the
statements and the proofs do not need any modification.

\begin{prop}
\label{prop:ext} The spaces $\boB_{ext}^{-k_h, k_v}$ have the following
properties:
\begin{enumerate}
\item The space $\boB^{-k_h, k_v}$ is a closed subspace of
    $\boB_{ext}^{-k_h, k_v}$.
\item The space $\boB_{ext}^{-k_h, k_v}$ is canonically a space of
    distributions, as in Proposition~\ref{prop:distrib}.
\item Multiplication by $C^\infty$ functions which are constant on a
    neighborhood of the singularities, or more generally by
    $C^{k_h+k_v}$-functions on $M-\Sigma$ with $L_h^a L_v^b \psi$ uniformly
    bounded for $a \leq k_h$ and $b \leq k_v$, maps $\boB_{ext}^{-k_h,
    k_v}$ into itself continuously, as in Lemma~\ref{lem:partition_unite}.
\item The derivation $L_h$ maps continuously $\boB_{ext}^{-k_h, k_v}$ to
    $\boB_{ext}^{-k_h-1, k_v}$. The derivation $L_v$ maps continuously
    $\boB_{ext}^{-k_h, k_v}$ to $\boB_{ext}^{-k_h, k_v-1}$ if $k_v \geq 1$,
    as in Proposition~\ref{prop:Lh_Lv_action}.
\item As there is no horizontal saddle connection, an element in
    $\boB_{ext}^{-k_h, k_v}$ satisfying $L_h f=0$ is constant, as in
    Lemma~\ref{lem:Lh_0}.
\item The space $\boB_{ext}^{-k_h, k_v}$ is continuously included in
    $\boB^{-k'_h, k'_v}$ if $k'_h \geq k_h$ and $k'_v\leq k_v$. This
    inclusion is compact if both inequalities are strict, as in
    Proposition~\ref{prop:compact_inclusion}.
\item \label{prop:Text} The composition operator $\boT$ acts continuously
    on $\boB_{ext}^{-k_h, k_v}$, and it satisfies a Lasota-Yorke
    inequality~\eqref{eq:LY}. Therefore, its spectral radius is bounded by
    $1$, and its essential spectral radius is at most $\lambda^{-\min(k_h,
    k_v)}$, as in Theorem~\ref{thm:rho_ess}.
\item We have $\boT \circ L_v = \lambda L_v \circ \boT$ and $\boT \circ L_h
    = \lambda^{-1} L_h\circ \boT$, as in Proposition~\ref{prop:action_Lv}.
\end{enumerate}
\end{prop}

The space $\boB^{-k_h, k_v}_{ext}$ is relevant to study vertically invariant
distributions, as all such distributions belong to these spaces:
\begin{lem}
\label{lem:all_Lvinv_in_boBext} Assume that $\omega$ is an $L_v$-annihilated
distribution. Then for large enough $k_h$ and for any $k_v$ one has $\omega
\in \boB^{-k_h, k_v}_{ext}$.
\end{lem}
\begin{proof}
Let $\omega$ be an $L_v$-annihilated distribution. For an interval $I \in
\boI^h$, define a distribution $\eta_I$ on $I$ by the equality $\int_I
\phi(x) \eta_I(x) = \int \phi(x) \rho(y) \omega(x,y)$, where $\rho$ is a
smooth function supported in $[-\delta,\delta]$ (where $\delta$ is small
enough so that $I \times [-\delta, \delta]$ does not contain any singularity)
with $\int \rho = 1$. We claim that this quantity does not depend on $\rho$.
Indeed, if $\tilde \rho$ is another such function, then $(x,y)\mapsto \phi(x)
(\rho(y)-\tilde \rho(y))$ has zero average along every vertical segment
through $I \times [-\delta, \delta]$, hence it can be written as $L_v f$ for
some function $f$ supported in $I \times [-\delta,\delta]$. Then
\begin{equation*}
  0 = \langle L_v \omega, f \rangle = - \langle \omega, L_v f \rangle
  = \langle \omega, \phi(x)\tilde \rho(y) \rangle - \langle \omega, \phi(x) \rho(y) \rangle.
\end{equation*}
This shows that $\eta_I$ is well defined. It is a finite order distribution
on any interval $I$. Moreover, as $\omega$ is vertically invariant, one has
$\eta_{I_t} = \eta_I$ if $I_t$ is a vertical family of horizontal segments
through $I$.

By compactness of the manifold, there is a finite family of horizontal
segments such that any horizontal segment can be obtained as a subinterval of
a vertical translate of one interval in the finite family. If follows that
the order of all the distributions $\eta_I$ is uniformly bounded,
independently of $I \in \boI^h$. By vertical invariance, it follows that the
family $\eta_I$ defines an element $\eta \in \boB^{-k_h, k_v}_{ext}$ if $k_h$
is large enough.

Let us finally prove that $\omega = \eta$ as distributions. Consider a smooth
function $\phi$ supported by a rectangle $I \times [-\delta, \delta]$ away
from singularities. Then
\begin{align*}
  \langle \eta, \phi \rangle
  &= \int_{t={-\delta}}^\delta \int_{I_t} \phi(x,t) \eta_{I_t}
  = \int_{t={-\delta}}^\delta \int_{I_t} \phi(x,t) \eta_I
  =\int_I \pare*{\int_{t={-\delta}}^\delta \phi(x,t) \dd t} \eta_I
  \\&=\int \pare*{\int_{t={-\delta}}^\delta \phi(x,t) \dd t} \rho(y) \omega(x,y),
\end{align*}
where the last equality is the definition of $\eta_I$. Since the integrals of
$\pare*{\int_{t={-\delta}}^\delta \phi(x,t) \dd t} \rho(y)$ and $\phi$ are
the same along all vertical segments, this is equal to $\int \phi \omega$
thanks to the vertical invariance of $\omega$ as we have explained above.

We have proved that $\langle \eta, \phi \rangle = \langle \omega, \phi
\rangle$ for any smooth function $\phi$ with compact support in a rectangle
away from the singularities. As any $\phi \in C^\infty_c(M-\Sigma)$ can be
decomposed as a finite sum of such functions, we obtain $\eta = \omega$ as
desired.
\end{proof}

Since the space $C^\infty_c(M-\Sigma)$ is \emph{not} dense in
$\boB^{-k_h,k_v}_{ext}$, we can not use the theorem of Baladi-Tsujii to claim
that the eigenspaces beyond the essential spectral radius do not depend on
$k_h$ or $k_v$. Nevertheless, we will show that this is the case, by
describing explicitly the new eigenvalues compared to $\boB^{-k_h, k_v}$.

For $\sigma \in \Sigma$ and $i_h, i_v \geq 0$, we define a distribution
$\xi^{(0)}_{\sigma, i_h, i_v}$ as follows. Choose a vertical segment
$\Gamma_\sigma$ ending on $\sigma$ and whose image under the covering
projection is in the negative half-plane, choose a function $\rho$ on this
segment which is equal to $1$ on a neighborhood of the singularity and to $0$
on a neighborhood of the other endpoint of the segment, and define a
distribution $\xi^{(0)}_{\sigma, i_h, i_v} \in \boB_{ext}^{-k_h, k_v}$ by
$\langle \xi^{(0)}_{\sigma, i_h, i_v} f\rangle = \int_{\Gamma_\sigma}
\rho(y)y^{i_v} L_v^{i_h} f(y)\dd y$. In other words, the corresponding
distribution on a horizontal segment $I$ is equal to $\rho(y_I) y_I^{i_v}
\delta^{(i_h)}_{x_I}$ if $I$ intersects $\Gamma_\sigma$ at a point $z_I=(x_I,
y_I)$, and $0$ otherwise. This is clearly an element of $\boB_{ext}^{-k_h,
k_v}$ if $i_h \leq k_h$.

\begin{prop}
\label{prop:finite_codim} An element $\omega$ of $\boB_{ext}^{-k_h, k_v}$ can
be written uniquely as
\begin{equation}
\label{eq:dec_omega}
\omega = \tilde \omega + \sum_{\sigma \in \Sigma} \sum_{i_h \leq k_h, i_v\leq
k_v} c_{\sigma, i_h, i_v} \xi^{(0)}_{\sigma, i_h, i_v},
\end{equation}
with $\tilde\omega \in\boB^{-k_h-1, k_v} \cap \boB_{ext}^{-k_h, k_v}$.
Moreover, this decomposition depends continuously on $\omega$.
\end{prop}
The reason we have $\tilde\omega \in\boB^{-k_h-1, k_v}$ and not $\tilde\omega
\in\boB^{-k_h, k_v}$ in the statement is that a distribution of order $k_h$
is not well approximated in $(C^{k_h})^*$ by a regularization by convolution:
one needs to use smoother test functions, in $C^{k_h+1}$, to get uniform norm
controls.
\begin{proof}
Let us first prove the uniqueness in the decomposition~\eqref{eq:dec_omega}.
Consider a singularity $\sigma$, of angle $2\bpi\kappa$. There are $\kappa$
half-planes above $\sigma$, and $\kappa$ half-planes below $\sigma$. Along
any of these half-planes $U$, consider horizontal intervals $I_t$ which are
all vertical translates of an interval $I_0=I_0(U)$ through the singularity
$\sigma$, identified with $[-\delta,\delta] \subset \C$ by the covering
projection sending $\sigma$ to $0$. By Condition (3) in the definition of
$\boB_{ext}^{-k_h, k_v}$, the corresponding distributions $\omega_{I_t}$
converge to $\omega_{I_0(U)}$. Consider now the distribution on
$[-\delta,\delta]$ defined by
\begin{equation*}
  \omega_\sigma \coloneqq \sum \omega_{I_0(U^+)} - \sum \omega_{I_0(U^-)}
\end{equation*}
where the first sum is over all half-planes above $\sigma$, and the second
sum is over all half-planes below $\sigma$. By vertical continuity to the
left and to the right of the singularity, there are many cancellations in the
definition of $\omega_\sigma$, so that this distribution on $[-\delta,
\delta]$ is in fact supported at $0$. Therefore, it is a linear combination
of derivatives of Dirac masses~\cite[Theorem~2.3.4]{hormander}, of the form
$\sum_{i\leq k_h} c_i \delta^{(i)}_0$. Let us do the same construction with
the term on the right of~\eqref{eq:dec_omega}. For functions $f \in
C^\infty_c(M-\Sigma)$, the distribution $f_\sigma$ is obviously $0$. By
density, this extends to $\boB^{-k_h-1, k_v}$, hence $\tilde \omega_\sigma =
0$. In the same way, the singularities different from $\sigma$ do not
contribute, and the functions $\xi^{(0)}_{\sigma, i_h, i_v}$ contribute only
when $i_v = 0$, with a distribution $\delta^{(i_h)}$. Identifying the
coefficients, we get that $c_{\sigma, i_h, 0} = c_i$ is uniquely defined by
$\omega$. In the same way, we can identify $c_{\sigma, i_h, i_v}$ from
$\omega$ by the same process after $i_v$ vertical differentiations. This
shows that the decomposition~\eqref{eq:dec_omega} is unique. Moreover, the
continuity of the decomposition follows from the continuity of all the
coefficients $c_{\sigma, i_h, i_v}$, which is obvious from the construction.

For the existence, let us decompose $\omega$ as
\begin{equation*}
  \omega = \sum_{i=1}^N \omega_i + \sum_{\sigma\in \Sigma} \omega_\sigma + \sum_{H \in \boH} \omega_H
\end{equation*}
as in Lemma~\ref{lem:decomposition_f}, where $\omega_i$ is supported in a
rectangle $R_i$ away from the singularities, and $\omega_\sigma$ is supported
in a small disk around the singularity $\sigma$ and is constant along fibers
of the covering projection $\pi_\sigma$, and $\omega_H$ is supported in a
local half-plane $H$ based at a singularity. Indeed, the proof of
Lemma~\ref{lem:decomposition_f} goes through in $\boB^{-k_h, k_v}_{ext}$. We
will show that each term in this decomposition can be written as
in~\eqref{eq:dec_omega}.

\emph{We start with $\omega_i$.} Let $\rho_\epsilon(x)$ be a real $C^\infty$
approximation of the identity. For $z = (x,y)$ in a chart, define
\begin{equation*}
  f_\epsilon(z) = \omega_i * \rho_\epsilon (z) = \int \omega_i(x-h, y) \rho_\epsilon(h) \dd h.
\end{equation*}
This is an integral of $\omega_i$ along a small horizontal interval against a
$C^\infty_c$ function, hence it is well defined. Moreover, $f_\epsilon$ is
$C^\infty$ along the horizontal direction, $C^{k_v}$ along the vertical
direction, and compactly supported away from the singularities. By
Lemma~\ref{lem:boB_riche}, $f_\epsilon \in \boB^{-k_h-1, k_v}$. Moreover,
$f_\epsilon$ converges in $\boB^{-k_h-1, k_v}$ to $\omega_i$ thanks to the
fact that $\omega_i$ is of order $k_h$ and to the fact that we are using
$C^{k_h+1}$ test functions: standard properties of convolutions ensure that
their difference is bounded by $O(\epsilon)$ in norm. It follows that
$\omega_i \in \boB^{-k_h-1, k_v}$. This gives the
decomposition~\eqref{eq:dec_omega} for $\omega_i$, just taking
$\tilde\omega=\omega_i$ and the other terms equal to $0$.

\emph{Let us now consider $\omega_\sigma$.} Its push-forward $\eta =\pi_*
\omega_\sigma$ under the covering projection $\pi$ is almost in
$\boB_{ext}^{-k_h, k_v}(\C)$, except for the fact that the horizontal
distributions do not have to match when one reaches $0$ from above and from
below. The difference is exactly given by a sum of the form $\sum_{i_h, i_v}
c_{i_h, i_v} \xi^{(0)}_{0, i_h, i_v}$ as constructed above. In other words,
we have
\begin{equation*}
  \eta = \tilde \eta + \sum_{i_h, i_v} c_{i_h, i_v} \xi^{(0)}_{0, i_h, i_v},
\end{equation*}
with $\tilde\eta \in \boB_{ext}^{-k_h, k_v}(\C)$. The case away from
singularities shows that $\tilde \eta \in \boB^{-k_h-1, k_v}(\C)$. Lifting
everything with $\pi$, we get
\begin{equation*}
  \omega_\sigma = \eta \circ \pi = \tilde \eta \circ \pi + \sum_{i_h, i_v} c_{i_h, i_v} \xi^{(0)}_{0, i_h, i_v} \circ \pi.
\end{equation*}
The first term $\tilde \eta \circ \pi$ belongs to $\boB^{-k_h-1, k_v}$. For
the other terms, $\xi^{(0)}_{0, i_h, i_v} \circ \pi$ is not equal to
$\xi^{(0)}_{\sigma, i_h, i_v}$ as the latter is supported on one single
vertical segment ending on $\sigma$ while the former is supported on all
$\kappa$ such segments. We claim that the difference belongs to
$\boB^{-k_h-1, k_v}$, which will conclude the proof.

To prove this, consider a vertical half-plane $H$ with $\sigma$ in its
boundary, and denote by $\Gamma_+$ and $\Gamma_-$ the two components of its
boundary, above and below $\sigma$. Define a distribution $\alpha_H =
\int_{\Gamma_-} y^{i_v} \delta^{(i_h)} \rho(y) \dd y + \int_{\Gamma_+}
y^{i_v} \delta^{(i_h)} \rho(y) \dd y$ where $\rho$ is smooth and equal to $1$
on a neighborhood of $0$. This distribution belongs to $\boB^{-k_h-1, k_v}$,
as it is the limit of a smooth function supported in the interior of $H$,
constructed by approximating inside $H$ the derivative of the Dirac mass with
a smooth function. Consider now two consecutive half-planes $H$ and $H'$
sharing the same $\Gamma_+$. Taking the difference between $\alpha_H$ and
$\alpha_{H'}$, we deduce that
\begin{equation*}
  \int_{\Gamma_-} y^{i_v} \delta^{(i_h)} \rho(y) \dd y
  -   \int_{\Gamma'_-} y^{i_v} \delta^{(i_h)} \rho(y) \dd y
  \in \boB^{-k_h-1, k_v}.
\end{equation*}
Iterating the argument using a sequence of half-planes, we deduce that the
same holds for any vertical segments $\Gamma_-$ and $\Gamma'_-$ ending at
$\sigma$. This concludes the proof of the decomposition for $\omega_\sigma$.

\emph{Let us now consider $\omega_H$ where $H$ is a local vertical half-plane
with a singularity $\sigma$ in its boundary.} This case is easy: as in the
case away from singularities, one can smoothen $\omega_i$ by convolving it
with a kernel $\rho_\epsilon$, with the additional condition that
$\rho_\epsilon$ is supported in $[\epsilon,2\epsilon]$ if $H$ is to the right
of $\sigma$, and in $[-2\epsilon, -\epsilon]$ if $H$ is to the left of
$\sigma$: this ensures that $\omega_i * \rho_\epsilon$ is supported in $H$
and everything matches vertically. In fact, the resulting distribution will
not be smooth vertically if there is a discrepancy between what happens on
the boundaries $\Gamma_+$ and $\Gamma_-$ of $H$ above and below $\sigma$.
This discrepancy is handled as in the case of $\omega_\sigma$, by first
subtracting a distribution supported on $\Gamma_-$ to make sure there is no
discrepancy, and then arguing that this distribution supported on $\Gamma_-$
can be written in the form~\eqref{eq:dec_omega}.

\emph{Finally, let us consider $\omega_H$ where $H$ is a local horizontal
half-plane with a singularity $\sigma$ in its boundary.} Subtracting if
necessary a distribution $\eta$ supported in the vertical segment inside $H$
ending on $\sigma$, we can assume that the distribution induced by $\omega_H$
on the boundary of $H$ vanishes, as well as all its vertical derivatives up
to order $k_v$. The distribution $\eta$ is handled as in the two previous
cases. Let us then smoothen $\omega_H$ by convolving with a kernel
$\rho_\epsilon$ in the horizontal direction. Inside $H$, we get a smooth
function. On the boundary of $H$, this function vanishes, as well as its
vertical derivatives up to order $k_v$. Hence, if one extends this function
by $0$ outside of $H$, we get a $C^{k_v}$ function, which belongs to
$\boB^{-k_h-1, k_v}$ by Lemma~\ref{lem:boB_riche}. It approximates $\omega_H$
in the $\boB_{ext}^{-k_h, k_v}$ norm, showing that $\omega_H \in
\boB^{-k_h-1, k_v}$. This concludes the proof.
\end{proof}

\begin{cor}
\label{cor:spectrum_ext} The spectrum of $\boT$ on $\boB_{ext}^{-k_h, k_v}$
in $\{z \st \abs{z}>\lambda^{-\min(k_h, k_v)}\}$ is given by the spectrum of
$\boT$ on $\boB^{-k_h, k_v}$ in this region as described in
Theorem~\ref{thm:main_preserves_orientations}, and additionally $j \Card
\Sigma$ eigenvalues of modulus $\lambda^{-j}$ for any $j\geq 1$ with
$j<\min(k_h, k_v)$.
\end{cor}
One can be more specific about the additional eigenvalues. If $T$ stabilizes
pointwise each singularity, then $\lambda^{-j}$ itself is an eigenvalue of
multiplicity $j\Card \Sigma$. Otherwise, there are cycles of singularities,
and each cycle of length $p$ gives rise to eigenvalues $e^{2\ic k\bpi/p}
\lambda^{-j}$ with multiplicity $j$ for $k=0,\dotsc, p-1$.

We can also formulate the results in terms of the action of $T^\ast$ on
relative cohomology group $H^1(M,\Sigma,\C)$ (the eigenvalues of $T^\ast$ are
then $\lambda,\lambda^{-1},\mu_i$ for $i=1,\dots,2g-2$ and roots of unity
$e^{2\ic k\bpi/p}$ for some $p$ corresponding to cycles of singularities of
length $p$).

\begin{proof}
Define $E = \boB^{-k_h-1, k_v} \cap \boB^{-k_h, k_v}_{ext}$ and $F =
\boB_{ext}^{-k_h, k_v}/(\boB^{-k_h-1, k_v} \cap \boB^{-k_h, k_v}_{ext})$. The
space $E$ is closed and the space $F$ is finite-dimensional, isomorphic to
the span of $\xi^{(0)}_{\sigma, i_h, i_v}$ for $i_h \leq k_h$ and $i_v\leq
k_v$, by Proposition~\ref{prop:finite_codim}.

The space $E$ is stable under $\boT$, and the essential spectral radius of
$\boT$ on this space is $\leq \lambda^{-\min(k_h, k_v)}$ as this is the case
on the whole space $\boB_{ext}^{-k_h, k_v}$ by
Proposition~\ref{prop:ext}(\ref{prop:Text}). Since $C^\infty_c(M-\Sigma)$ is
dense in $E$, it follows from the theorem of Baladi-Tsujii that the spectrum
of $\boT$ on $E$ beyond $\lambda^{-\min(k_h, k_v)}$ is the same as on
$\boB^{-k_h, k_v}$. Moreover, since $\boT$ stabilizes $E$, its spectrum on
the whole space is the union of its spectrum on $E$ and on $F$. To conclude,
we should thus describe the spectrum of $\boT$ on $F$.

The image under $\boT$ of $\xi^{(0)}_{\sigma, i_h, i_v}$ is equal to the sum
of $\lambda^{-1-i_h-i_v} \xi^{(0)}_{T^{-1}\sigma, i_h, i_v}$ and of a
distribution in $\boB^{-k_h-1, k_v} \cap \boB^{-k_h, k_v}_{ext}$. Indeed,
this follows readily from the definition if the vertical segment
$\Gamma_{T^{-1}\sigma}$ is sent by $T$ to $\Gamma_\sigma$. In general, it is
sent to another vertical segment ending on $\sigma$, but
Proposition~\ref{prop:finite_codim} shows that changing the choice of the
vertical segment results in a difference in $\boB^{-k_h-1, k_v} \cap
\boB_{ext}^{-k_h, k_v}$. This shows that the matrix of $\boT$ on the
finite-dimensional space $F$ is a union of permutation matrices multiplied by
$\lambda^{-j}$ for $j=1+i_h+i_v$. The spectrum of such a permutation matrix,
along a cycle of length $p$, is made of the eigenvalues $e^{2\ic k\bpi/p}$
for $k=0,\dotsc, p-1$. Hence, the spectrum of $\boT$ on $F$ is made of
eigenvalues of modulus $\lambda^{-j}$, and the number of such eigenvalues is
\begin{equation*}
   \Card \{(i_h, i_v) \st i_h \leq k_h, i_v \leq k_v, \ j = i_h+iv+1\} \cdot \Card\Sigma.
\end{equation*}
For $j < \min(k_h, k_v)$, this is equal to $j \Card \Sigma$.
\end{proof}
The description of the spectrum of $\boT$ on $F$ in this proof is reminiscent
of the description of the spectrum of $\boT$ on $\boB^{-k_h, k_v}$, but in a
simpler situation. Assume to simplify the discussion that $T$ acts as the
identity on $\Sigma$. Then there are some basic eigenfunctions for the
eigenvalue $\lambda^{-1}$, which are the classes of the functions
$\xi^{(0)}_\sigma = \xi^{(0)}_{\sigma, 0, 0} = \int_{\Gamma_\sigma} \delta\cdot
\rho(y)\dd y$ modulo $\boB^{-k_h-1, k_v} \cap \boB_{ext}^{-k_h, k_v}$. The
other eigenfunctions are given by $\xi^{(0)}_{\sigma, i_h, i_v} =
\int_{\Gamma_\sigma} y^{i_v} \delta^{(i_h)}\cdot \rho(y)\dd y$. They are obtained
by differentiating the original function $i_h$ times in the horizontal
direction, and integrating it $i_v$ times in the vertical direction. To
obtain the eigenvalue $\lambda^{-j}$, the total number of such operations
$i_h+i_v$ should be equal to $j-1$, giving $j$ choices.

It follows from the above corollary that one can define the generalized
eigenspace $E_{\alpha, ext}$ associated to the eigenvalue $\alpha$ of $\boT$
acting on $\boB^{-k_h, k_v}$ for large enough $k_h$ and $k_v$. This space of
distributions does not depend on $k_h$ and $k_v$ if they are large enough.
Moreover, $L_v$ maps $E_{\alpha, ext}$ to $E_{\lambda\alpha, ext}$ and $L_h$
maps $E_{\alpha, ext}$ to $E_{\lambda^{-1}\alpha, ext}$ as in
Corollary~\ref{cor:Lu_Ls_Erho}.

To proceed, we will need some ingredients of duality. In general, there is no
canonical way to define a pairing between $\boB^{-k_h, k_v}_{ext}$ and
$\check \boB^{\check k_h, -\check k_v}_{ext}$. Indeed, consider a
distribution $\phi$ on $[-1,1]$ for which $\int_{-1}^1 1_{y \leq 0} \phi(y)$
does not make sense, and define a distribution $\omega \in \check
\boB^{-\check k_h, \check k_v}_{ext}$ which is equal to $\phi$ on each
vertical leaf around a singularity $\sigma$, multiplied by a cutoff function
to extend it by $0$ elsewhere. Then one can not make sense of $\langle
\xi^{(0)}_{\sigma, 0, 0}, \phi \rangle$. However, there is no difficulty to
define $\langle \omega, 1 \rangle$ by integrating a partition of unity along
horizontal segments, and then summing over the partition of unity. When
$\omega$ belongs to $\boB^{-k_h, k_v}$, this coincides with the duality
between $\boB^{-k_h, k_v}$ and $\check\boB^{\check k_h, -\check k_v}$ defined
in Proposition~\ref{prop:dualite} if one considers the distribution $1$ as an
element of $\check\boB^{\check k_h, -\check k_v}$. The main property of this
linear form we will use is the following.

\begin{lem}
\label{lem:duality_ext} Let $\omega \in \boB^{-k_h, k_v}_{ext}$. Consider its
decomposition given by Proposition~\ref{prop:finite_codim}. Then
\begin{equation*}
  \langle L_v \omega, 1 \rangle = \sum_{\sigma} c_{\sigma, 0, 0}.
\end{equation*}
\end{lem}
\begin{proof}
We should show that $\langle L_v \tilde \omega, 1 \rangle = 0$, and that
$\langle L_v\xi^{(0)}_{\sigma, i_h, i_v}, 1\rangle = 1$ if $i_h=i_v=0$ and
$0$ otherwise. First, $\langle L_v \tilde \omega, 1 \rangle = - \langle
\tilde\omega, L_v 1 \rangle = 0$ by Lemma~\ref{lem:dualite}. The fact that
$L_v$ is antiselfadjoint does not apply to $\xi^{(0)}_{\sigma, i_h, i_v}$ as
additional boundary terms show up when one does integrations by parts
(contrary to the case of elements of $\boB^{-k_h, k_v}$, which are in the
closure of compactly supported functions and for which there is therefore no
boundary term). These boundary terms are responsible for the formula in the
lemma, as we will see in the following computation.

We show that $\langle L_v\xi^{(0)}_{\sigma, 0, 0}, 1\rangle = 1$, the other
case is similar. Write $\xi^{(0)}_{\sigma, 0, 0} = \int_{y=-\delta}^{0}
\rho(y) \delta_{(x,y)} \dd y$ as in its definition, where we are integrating
on a vertical segment ending at a singularity and $\rho$ vanishes on a
neighborhood of $-\delta$ and is equal to $1$ on a neighborhood of $0$. Then
$L_v \xi^{(0)}_{\sigma, 0, 0} = \int_{y=-\delta}^{0} \rho'(y) \delta_{(x,y)}
\dd y$. Therefore,
\begin{equation*}
  \langle L_v \xi^{(0)}_{\sigma, 0, 0}, 1 \rangle = \int_{y=-\delta}^{0} \rho'(y) \dd y
  = \rho (0) - \rho(-\delta) = 1.
  \qedhere
\end{equation*}
\end{proof}

We can now prove Proposition~\ref{prop:spurious_distributions}, asserting
that $\xi^{(0)}_\sigma = \xi^{(0)}_{\sigma, 0, 0}$ can be modified by adding
an element of $\boB^{-k_h, k_v}$ to obtain a distribution which is mapped by
$L_v$ to the constant distribution $1/\Leb M$. As in the statement of the
proposition, we will denote this modified distribution by $\xi_\sigma$ or
$\xi_{\sigma,0,0}$.

\begin{proof}[Proof of Proposition~\ref{prop:spurious_distributions}]
We work in $\boB_{ext}^{-2, k_v}$. On this space, the essential spectral
radius of $\boT$ is $\leq \lambda^{-2} < \lambda^{-1}$. Replacing $T$ by a
power of $T$ if necessary, we can assume without loss of generality that
$\sigma$ is fixed by $T$. Then $\boT \xi^{(0)}_\sigma = \lambda^{-1}
\xi^{(0)}_\sigma + \eta$ where $\eta \in E = \boB^{-3, k_v} \cap
\boB_{ext}^{-2,k_v}$ as explained in the proof of
Corollary~\ref{cor:spectrum_ext}. Since the essential spectral radius of
$\boT$ on $E$ is $\leq \lambda^{-2}$ (see again the proof of
Corollary~\ref{cor:spectrum_ext}), we can decompose $\eta = \eta_1 + \eta_2$
where $\eta_1$ is in the generalized eigenspace associated to $\lambda^{-1}$,
and $\eta_2$ belongs to its spectral complement, on which $\boT-\lambda^{-1}$
is invertible. Therefore, we can write $\eta_2 = -(\boT-\lambda^{-1}) \omega$
for some $\omega \in E$. Finally, we have
\begin{equation*}
  (\boT - \lambda^{-1}) (\xi^{(0)}_\sigma + \omega) = \eta - \eta_2 = \eta_1.
\end{equation*}
Since $\eta_1$ is a generalized eigenvector for the eigenvalue
$\lambda^{-1}$, we have $(\boT-\lambda^{-1})^N \eta_1 = 0$ for large enough
$N$. Hence, $(\boT - \lambda^{-1})^{N+1} (\xi^{(0)}_\sigma + \omega) = 0$.
This shows that $\xi^{(0)}_\sigma + \omega$ belongs to the generalized
eigenspace $E_{\lambda^{-1}, ext}$ associated to the eigenvalue
$\lambda^{-1}$ of $\boT$ acting on $\boB^{-2, 2}_{ext}$. Moreover, as
$\min(k_h, k_v)\geq 3$, we have $\omega \in \boB^{-k_h, k_v}$.

To conclude the proof, it remains to show that $L_v(\xi^{(0)}_\sigma +
\omega) = 1/\Leb M$. Since $\xi^{(0)}_\sigma + \omega \in E_{\lambda^{-1},
ext}$, we have $L_v(\xi^{(0)}_\sigma + \omega) \in E_{1,ext}$. The
description of the spectrum in Corollary~\ref{cor:spectrum_ext} shows that
this space is just $E_1$. By Lemma~\ref{lem:E1}, it is made of constants. We
get the existence of a constant $c$ such that $L_v(\xi^{(0)}_\sigma + \omega)
= c$.

To identify $c$, we compute
\begin{equation*}
  c \Leb M = \langle c, 1 \rangle = \langle L_v(\xi^{(0)}_\sigma + \omega), 1\rangle
  = 1,
\end{equation*}
thanks to Lemma~\ref{lem:duality_ext}. This proves that $c = 1/\Leb M$.
\end{proof}

\begin{lem}
\label{lem:Lvinv_description} Let $k_h, k_v \geq 3$. Then all
$L_v$-annihilated distributions in $\boB^{-k_h, k_v}_{ext}$ are of the form
described in Theorem~\ref{thm:Lv_invariant_preserves_orientations}, i.e.,
they are linear combinations of distributions $\xi_\sigma - \xi_{\sigma'}$
for $\sigma,\sigma'\in \Sigma$, of $L_h^n \xi_\sigma$ with $n\geq 1$ and
$\sigma\in \Sigma$, of $1$, and of $L_h^n E^H_{\lambda^{-1}\mu_i}$ with
$n\geq 0$ and $i=1,\dotsc, 2g-2$.
\end{lem}
\begin{proof}
Define a distribution $\xi_{\sigma, i_h, i_v} = L_h^{i_h} \xi_{\sigma, 0, 0}$
if $i_v = 0$ and $\xi_{\sigma, i_h, i_v} = \xi^{(0)}_{\sigma, i_h, i_v}$
otherwise. Then we have
\begin{equation*}
  \boB^{-k_h, k_v}_{ext} = (\boB^{-k_h-1, k_v} \cap \boB^{-k_h, k_v}_{ext}) \oplus \bigoplus_{i_h \leq k_h, i_v\leq k_v} \R \xi_{\sigma, i_h, i_v},
\end{equation*}
by Proposition~\ref{prop:finite_codim} and the fact that $\xi_{\sigma, i_h,
i_v} - \xi^{(0)}_{\sigma, i_h, i_v} \in \boB^{-k_h-1, k_v} \cap \boB^{-k_h,
k_v}_{ext}$. Write this decomposition as $\boB^{-k_h, k_v}_{ext}  = E \oplus
F$. On $\boB^{-k_h, k_v}_{ext} / E$, the operator $L_v$ maps $\xi_{\sigma,
i_h, i_v}$ to $\xi_{\sigma, i_h, i_v-1}$ if $i_v > 0$, and to $0$ if $i_v=0$.
Therefore, a distribution $\omega$ with $L_v \omega = 0$ must have zero
components on $\xi_{\sigma, i_h, i_v}$ for $i_v>0$: it can be written as
$\tilde \omega + \sum_{i_h \leq k_h} c_{\sigma, i_h} \xi_{\sigma, i_h, 0}$.
Moreover, $L_v \tilde \omega = 0$.

By Lemma~\ref{lem:duality_ext}, we have
\begin{equation*}
  0 = \langle L_v \omega, 1 \rangle = \sum_{\sigma} c_{\sigma, 0}.
\end{equation*}
This shows that $\omega -\tilde \omega$ belongs to the vector space generated
by the $\xi_\sigma-\xi_{\sigma'}$ over $\sigma, \sigma'$, and by all the
$L_h^n \xi_\sigma$ for $n>0$. Moreover, Lemma~\ref{lem:Lv_invariant_boB}
shows that $\tilde\omega$ belongs to the span of the constant distribution
and of $L_h^n E^H_{\lambda^{-1}\mu_i}$ with $n\geq 0$ and $i=1,\dotsc, 2g-2$.
This concludes the proof.
\end{proof}

Since all $L_v$-invariant distributions belong to some space $\boB^{-k_h,
k_v}_{ext}$ by Lemma~\ref{lem:all_Lvinv_in_boBext},
Theorem~\ref{thm:Lv_invariant_preserves_orientations} giving the
classification of all vertically invariant distributions follows directly
from Lemma~\ref{lem:Lvinv_description}.

\begin{rmk}
\label{rmk:cohomological} Although it is not needed for the above proof, it
is enlightening to describe a cohomological interpretation for all the
elements of $\boB^{-k_h, k_v}_{ext} \ker L_v$, i.e., for all vertically
invariant distributions.

If $\gamma$ is a continuous closed loop in $M-\Sigma$ and $\omega \in
\boB^{-k_h, k_v}_{ext} \cap \ker L_v$, one can define the integral
$\int_\gamma \omega$ just like for elements in $\boB^{-k_h, k_v} \cap \ker
L_v$ (see the discussion before Proposition~\ref{prop:define_H1}). This
integral only depends on $\gamma$ up to deformation in $M-\Sigma$. Therefore,
it defines an element of $H^1(M-\Sigma)$, that we denote by $[\omega]_{ext}$.
Contrary to the case of $\boB^{-k_h, k_v} \cap \ker L_v$, however, the
integral $\int_{\gamma_\sigma} \omega$ along a small loop $\gamma_\sigma$
around a singularity $\sigma$ does not have to vanish, so that
$[\omega]_{ext}$ is \emph{not} an element of $H^1(M)$ in general. Indeed, if
one considers two different singularities $\sigma$ and $\sigma'$, then
$\xi_\sigma-\xi_\sigma'$ is annihilated by $L_v$, but the corresponding
cohomology class integrates to $1$ along a small positive loop around
$\sigma$, and to $-1$ along a small positive loop around $\sigma'$. This is a
direct consequence of the definition of $\xi_\sigma^{(0)}$, with a Dirac mass
along a vertical segment ending at $\sigma$, that will be intersected once by
a small loop around $\sigma$. In general, for $\omega \in \boB^{-k_h,
k_v}_{ext} \cap \ker L_v$, one has
\begin{equation}
\label{eq:pwuixcvpoiupwoxv}
  \int_{\gamma_\sigma} [\omega]_{ext} = c_{\sigma, 0, 0}(\omega),
\end{equation}
where $c_{\sigma, 0, 0}$ is defined in the decomposition of
Proposition~\ref{prop:finite_codim}. Indeed, $\xi^{(0)}_{\sigma, 0, 0}$
contributes by $1$ to the integral along a small loop around $\sigma$, while
the contribution of all the other terms tends to $0$ when the loop tends to
$\sigma$. In fact, the map $\omega \mapsto c_{\sigma, 0, 0}(\omega)$ corresponds
to the~\emph{boundary operator} of~\cite{marmi_yoccoz_Holder} (it does not appear
in the case of Ruelle resonances as all our functions are continuous
in this setting).

If a distribution $\omega \in \boB^{-k_h, k_v}_{ext} \cap \ker L_v$ satisfies
$[f]_{ext} = 0$, then one proves as in Proposition~\ref{prop:integre_Lu} that
it can be written as $\omega=L_h \eta$ for some $\eta \in \boB_{ext}^{-k_h+1,
k_v} \cap \ker L_v$. Indeed, the proof of this proposition goes through, and
it is in fact easier as one does not need to show that the resulting object
one constructs by horizontal integration belongs to the closure of
$C_c^\infty(M-\Sigma)$, which is the hard part in
Proposition~\ref{prop:integre_Lu}.

With~\eqref{eq:pwuixcvpoiupwoxv} and Lemma~\ref{lem:duality_ext}, one has
\begin{equation*}
  \sum_{\sigma} \int_{\gamma_\sigma} [\omega]_{ext}
  = \sum_{\sigma} c_{\sigma, 0, 0}(\omega) =
  \langle L_v \omega, 1 \rangle
  = 0.
\end{equation*}
This corresponds to the fact that, in the homology of $M-\Sigma$, one has
$\sum [\gamma_\sigma] = 0$.

The cohomology classes one can get in this way are all cohomology classes
without any $[\dd y]$ component, i.e., orthogonal to $[\dd x]$, as one can
realize all such classes in $H^1(M)$ using $\boB^{-k_h, k_v}$, and one can
account for the additional $\Card \Sigma - 1$ dimensions in $H^1(M-\Sigma)$
by using the $\xi_\sigma - \xi_{\sigma'}$. It turns out that one can also
recover the class $[\dd y]$. Indeed, start from $\xi_{\sigma, 0, 0}$ and
consider a path $\gamma$ made of horizontal and vertical segments. As $\dd
\xi_{\sigma,0, 0}$ is exact, one may compute formally
\begin{equation*}
  0 = \int_\gamma \dd \xi_{\sigma,0, 0}
  = \int_\gamma L_h \xi_{\sigma, 0,0} \dd x + \int_{\gamma} L_v \xi_{\sigma, 0, 0} \dd y
  = \int_\gamma L_h \xi_{\sigma, 0,0} \dd x + \frac{1}{\Leb M} \int_{\gamma} \dd y,
\end{equation*}
where the last equality follows from
Proposition~\ref{prop:spurious_distributions}. It follows that the element
$-\Leb M \cdot L_h \xi_{\sigma, 0,0}$, which belongs to $\boB^{-k_h,
k_v}_{ext} \cap \ker L_v$, has a cohomology class whose integral along any
path coincides with the integral of $\dd y$ along this path, i.e., $[-\Leb M
\cdot L_h \xi_{\sigma, 0,0}]_{ext} = [\dd y]$. The above formal computation
can be made rigorous by smoothing the path $\gamma$ horizontally, as we did
to define the cohomology classes. This shows that, for $k_h, k_v \geq 3$, the
map from $\boB^{-k_h, k_v}_{ext} \cap \ker L_v$ to $H^1(M-\Sigma)$ is onto.
This is the analogue of~\cite[Theorem~7.1(ii)]{forni_deviation} in our
setting.
\end{rmk}

\section{Solving the cohomological equation}
\label{sec:cohomological}

Consider a $C^\infty$ function $f$ which is compactly supported away from the
singularity set $\Sigma$ on a translation surface $M$. Solving the
cohomological equation for the vertical flow on $M$ amounts to finding a
function $F$, which is smooth along vertical lines, and satisfies the
equality $L_v F = f$. In general, the function $F$ will not be compactly
supported on $M-\Sigma$, but it will hopefully be continuous on $M$. More
generally, one may ask how smooth the solution $F$ can be chosen.

A direct obstruction to solve the cohomological equation with a smooth
solution is given by distributions in the kernel of $L_v$: if $L_v \omega =
0$, then
\begin{equation*}
  \langle \omega, f \rangle = \langle \omega, L_v F \rangle
  =- \langle L_v \omega, F \rangle = 0,
\end{equation*}
where the last equalities make sense if $F$ belongs to the space on which
$\omega$ acts. Indeed, in general, a distribution $\omega \in
\boD^\infty(M-\Sigma)$ is in the dual of $C^\infty_c(M-\Sigma)$, so that
$\langle \omega, F\rangle$ does not make sense if $F$ is not $C^\infty$ or
not compactly supported away from $\Sigma$. However, many distributions act
on larger classes of functions, so an important question in the discussion
below will be to see if $\langle \omega, F\rangle$ is meaningful.

The Gottschalk-Hedlund theorem states that, for a minimal continuous flow on
a compact manifold, a continuous function is a continuous coboundary if and
only if its Birkhoff integrals $\int_0^T f(g_t x) \dd t$ are bounded
independently of $x$ and $T$. We will use a variation around this result due
to Giulietti-Liverani~\cite{giulietti_liverani}. Its interest is that it
gives an explicit formula for the coboundary, which we will use to study its
smoothness.

In this section, we fix once and for all a $C^\infty$ function $\chi:\R \to
[0,1]$ which is equal to $1$ on a neighborhood of $(-\infty, 0]$ and to $0$
on a neighborhood of $[1,\infty)$.

\begin{lem}
\label{lem:GL_coboundary} Consider a semiflow $g_t$ on a space $X$, and a
function $f:X \to \R$ for which there exist $C>0$ and $\epsilon>0$ and $r\in
\N$ with the following property: for any $x \in X$, for any $\tau\geq 1$, for
any function $\phi$ which is compactly supported on $(0,1)$,
\begin{equation}
\label{eq:ziuerpoiuprt}
  \abs*{\int_{t=0}^\tau \phi(t/\tau) f(g_t x) \dd t} \leq C\norm{\phi}_{C^r}/\tau^\epsilon.
\end{equation}
Then $f$ is a coboundary: there exists a function $F$ such that $\int_0^\tau
f(g_t x) \dd t = F(x) - F(g_\tau x)$ for all $x\in X$ and all $\tau \geq 0$.

More specifically, $F$ can be constructed as follows. Fix $\lambda > 1$.
Define a function $F_n(x) = \int_{t=0}^{\lambda^n} \chi(t/\lambda^n) f(g_t x)
\dd t$. Then $F_n$ converges uniformly to a function $F$ as above. Moreover,
$\abs{F_n(x)-F(x)} \leq C \lambda^{-\epsilon n}$ where $C$ does not depend on
$x$ or $n$.
\end{lem}
In fact, one can even prove that $\int_{t=0}^{\tau} \chi(t/\tau) f(g_t x) \dd
t$ converges to $F(x)$ at a uniform rate $O(1/\tau^\epsilon)$ when $\tau\to
\infty$, without having to restrict to the subsequence $\lambda^n$, with a
small modification of the following proof. We will not need this more precise
version of the lemma.
\begin{proof}
This is essentially a reformulation of~\cite[Lemmas~1.4
and~3.1]{giulietti_liverani}.

Define $\phi(t) = \chi(t)-\chi(\lambda t)$. This is a $C^\infty$ function
with compact support on $(0,1)$. Moreover,
\begin{equation*}
  F_{n+1}(x) - F_n(x) = \int_{t=0}^{\lambda^{n+1}} (\chi(t/\lambda^{n+1}) - \chi(t/\lambda^n)) f(g_t x) \dd t
  = \int_{t=0}^{\lambda^{n+1}} \phi(t/\lambda^{n+1}) f(g_t x) \dd t,
\end{equation*}
Under the assumptions of the lemma, this is bounded by
$C(\phi)/\lambda^{(n+1)\epsilon}$. This shows that $F_n(x)$ is a Cauchy
sequence, converging uniformly to a limit $F(x)$ with $\abs{F_n(x)-F(x)} \leq
C \lambda^{-\epsilon n}$.

To conclude, we should show that $F$ solves the cohomological equation. Let
us fix $x$ and $\tau$. We have
\begin{align*}
  F_n(g_\tau& x) + \int_{0}^\tau f(g_t x) \dd t - F_n(x)
  \\&= \int_{\tau}^{\lambda^n + \tau} \chi((t-\tau)/\lambda^n) f(g_t x) + \int_{0}^\tau \chi((t-\tau)/\lambda^n) f(g_t x) \dd t
  - \int_0^{\lambda^n} \chi(t/\lambda^n) f(g_t x) \dd t
  \\& = \int_0^{\lambda^n + \tau} \phi_{n,\tau}(t/(\lambda^n + \tau)) f(g_t x) \dd t,
\end{align*}
where
\begin{equation*}
  \phi_{n,\tau}(s) = \chi( ((\lambda^n+\tau) s - \tau)/\lambda^n) - \chi( (\lambda^n+\tau)s/\lambda^n).
\end{equation*}
The function $\phi_{n,\tau}$ has compact support in $(0,1)$ and uniformly
bounded $C^r$ norm when $n$ tends to infinity. By~\eqref{eq:ziuerpoiuprt}
applied to $\phi_{n,\tau}$, we deduce that $F_n(g_\tau x) + \int_{0}^\tau
f(g_t x) \dd t - F_n(x)$ tends to $0$. Passing to the limit, we get $F(g_\tau
x) + \int_0^\tau f(g_t x) - F(x) = 0$.
\end{proof}

We will denote by $C^k_h$ the space of functions $M\to \R$ which are $C^k$
along the horizontal direction and such that $L_h^i f$ is continuous and
bounded on $M-\Sigma$ for $i \leq k$. Elements of $C^k_h$ belong to $\check
\boB^{k, 0}$ by Lemma~\ref{lem:boB_riche}. To formulate the assumptions of
our theorems, we will use the following fact:
\begin{equation}
\label{eq:makes_sense}
  \langle \omega, f\rangle \text{ makes sense for $f \in C^{k+2}_h$ and
  $\omega \in E_\alpha$ with $\abs{\alpha} \geq \lambda^{-k-1}$}.
\end{equation}
Indeed, elements of $E_\alpha$ for $\abs{\alpha} \geq \lambda^{-k-1}$ belong
to $\boB^{-k-2, k+2}$ as the essential spectral radius of $\boT$ on this
space is $\leq \lambda^{-k-2} < \lambda^{-k-1}$. Therefore, since $f\in
\check \boB^{k+2, 0}$, the coupling $\langle \omega, f \rangle$ is well
defined by Proposition~\ref{prop:dualite} (exchanging the roles of the
horizontal and the vertical direction to make sure that the inequalities on
the exponents are satisfied). One could even weaken slightly more the
conditions, by requiring only $f \in C^{k+1+\epsilon}_h$ for $\epsilon>0$, by
exploring the route alluded to in Remark~\ref{rmk:non_integers} if one were
striving for minimal assumptions.

We will apply the previous lemma in the setting of the vertical flow on a
translation surface endowed with a pseudo-Anosov map preserving orientations,
with expansion factor$\lambda$. We obtain the following criterion to have a
continuous coboundary.

\begin{thm}
\label{thm:continuous_coboundary} Let $T$ be a linear pseudo-Anosov map
preserving orientations on a translation surface $(M,\Sigma)$. Denote by
$g_t$ the vertical flow on this surface. Consider a function $f$ on $M$ in
$C^2_h$. Assume that, for any $\omega \in \bigcup_{\abs{\alpha} \geq
\lambda^{-1}} E_\alpha$, one has $\langle \omega, f \rangle = 0$. Then $f$ is
a continuous coboundary: there exists a continuous function $F$ on $M$ such
that, for any $x$ and any $\tau$ such that $g_t x$ is well defined for $t\in
[0,\tau]$, holds
\begin{equation}
\label{eq:coboundary}
  \int_0^\tau f(g_t x) \dd t = F(x) - F(g_\tau x).
\end{equation}
\end{thm}
The assumptions of the theorem make sense by~\eqref{eq:makes_sense}. The
distributions appearing in the statement of the theorem have been completely
classified in Theorem~\ref{thm:main_preserves_orientations} and its proof. In
particular, they are all vertically invariant.

To prove this theorem, let us first check that the assumptions of the
Giulietti-Liverani criterion of Lemma~\ref{lem:GL_coboundary} are satisfied.

\begin{lem}
\label{lem:GL_criterion_satisfied} Under the assumptions of
Theorem~\ref{thm:continuous_coboundary}, there exists $\epsilon>0$ such that
the inequality $\abs*{\int_{t=0}^\tau \phi(t/\tau) f(g_t x) \dd t} \leq
C\norm{\phi}_{C^2}/\tau^\epsilon$ in~\eqref{eq:ziuerpoiuprt} holds, with
$r=2$.
\end{lem}
\begin{proof}
It suffices to prove the estimate for $\tau$ of the form $\lambda^n$, as the
case of a general $\tau$ follows by using $n$ such that $\tau \in
[\lambda^{n-1}, \lambda^n]$. Fix $x$ and $\phi$. We have
\begin{equation}
\label{eq:wpicvpowixcv}
  \int_0^{\lambda^n} \phi(t/\lambda^n) f(g_t x) \dd t
  = \lambda^n \int_0^1 \phi(s) f(T^{-n} (g_s (T^n x))) \dd s
  = \lambda^n \int_0^1 \phi(s) \check \boT^n f (g_s y) \dd s,
\end{equation}
for $y = T^n x$. The integral is the integral of $\check\boT^n f \in \check
\boB^{2, -2}$ along a vertical manifold against a $C^2$ smooth function.
Therefore, this is bounded by $\lambda^n \norm{\phi}_{C^2} \norm{\check\boT^n
f}_{\check \boB^{2, -2}}$.

On this space, the essential spectral radius of $\check\boT$ is $\leq
\lambda^{-2} < \lambda^{-1}$, by Theorem~\ref{thm:rho_ess}. Let us decompose
$f$ as $\sum_{\alpha} f_\alpha + \tilde f$, where $\alpha$ runs among the
(finitely many) eigenvalues of $\check \boT$ of modulus $>\lambda^{-2}$, and
$f_\alpha$ is the component of $f$ on the corresponding generalized
eigenspace $\check E_\alpha$. By assumption, $\langle \omega, f \rangle = 0$
for any $\omega \in E_\alpha$ with $\abs{\alpha} \geq \lambda^{-1}$. Thanks
to the perfect duality statement given in Lemma~\ref{lem:dualite_Ealpha}, this
gives $f_\alpha = 0$ for all such $\alpha$. Let $\gamma<\lambda^{-1}$ be such
that all eigenvalues of modulus $<\lambda^{-1}$ have in fact modulus
$<\gamma$. We deduce that $\norm{\check\boT^n f}_{\check \boB^{2, -2}}$ grows
at most like $C \gamma^n$. Together with~\eqref{eq:wpicvpowixcv}, this gives
\begin{equation*}
  \abs*{\int_0^{\lambda^n} \phi(t/\lambda^n) f(g_t x) \dd t}
  \leq C \norm{\phi}_{C^2} (\lambda \gamma)^n.
\end{equation*}
As $\lambda \gamma < 1$, one may write $\lambda \gamma = \lambda^{-\epsilon}$
for some $\epsilon>0$. Then this bound is of the form $C \norm{\phi}_{C^2}
/(\lambda^n)^\epsilon$, as requested.
\end{proof}

There is a difficulty to apply Lemma~\ref{lem:GL_coboundary} due to the
singularities, which imply that the flow is not defined everywhere for all
times. One can circumvent the difficulty by going to a bigger space in which
trajectories ending on a singularity are split into two trajectories going on
both sides of the singularity. This results in a compact space with a Cantor
transverse structure and a minimal flow, to which
Lemma~\ref{lem:GL_coboundary} applies. This classical strategy works well for
continuous coboundary results, but there are difficulties in higher
smoothness. Instead, we will use a strategy which avoids the use of such an
extension, and works also for higher smoothness. The idea is to iterate the
flow in forward time or backward time depending on the point one considers.

\begin{proof}[Proof of Theorem~\ref{thm:continuous_coboundary}]
Let $M_n^+ \subseteq M$ be the set of points for which the vertical flow is
defined for all times in $[0,\lambda^n]$, and let $M^+ = \bigcap_n M_n^+$,
i.e., the set of points that do not reach a singularity in finite positive
time. In the same way, but using backward time, we define $M_n^-$ and $M^-$.
Then $M -\Sigma = M^+ \cup M^-$ as there is no vertical saddle connection.

Let us define functions $F_n^+(x)$ on $M_n^+$ and $F_n^-$ on $M_n^-$ by
\begin{equation*}
  F_n^+(x) = \int_0^{\lambda^n} \chi(t/\lambda^n) f(g_t x) \dd t, \quad F_n^-(x) = - \int_0^{\lambda^n} \chi(t/\lambda^n) f(g_{-t} x) \dd t.
\end{equation*}
For $x \in M_n^+ \cap M_n^-$, the difference $F_n^+(x) - F_n^-(x)$ can be
written as
\begin{equation*}
  F_n^+(x) - F_n^-(x) = \int_{-\lambda^n}^{\lambda^n} \tilde \chi(t/\lambda^n) f(g_t x) \dd t,
\end{equation*}
where $\tilde \chi(t) = \chi(\abs{t})$. By
Lemma~\ref{lem:GL_criterion_satisfied}, this tends to $0$ like
$C(\tilde\chi)/(2\lambda^n)^\epsilon$.

Lemma~\ref{lem:GL_coboundary} applied to the semiflow $g_t$ on $M^+$, and to
the semiflow $g_{-t}$ on $M^-$, shows that $F_n^+(x)$ converges uniformly to
a function $F^+(x)$ on $M^+$, and that $F_n^-(x)$ converges uniformly to a
function $F^-(x)$ on $M^-$. From the fact that the difference between $F_n^+$
and $F_n^-$ is small where defined, we deduce that $F^+ = F^-$ on $M^+ \cap
M^-$. Let us define a function $F$ on $M-\Sigma$, equal to $F^+$ on $M^+$ and
to $F^-$ on $M^-$. By the above, we have
\begin{equation}
\label{eq:Fn+F}
  \abs{F_n^+(x) - F(x)} \leq C/\lambda^{\epsilon n} \text{ for $x \in M_n^+$}, \quad
  \abs{F_n^-(x) - F(x)} \leq C/\lambda^{\epsilon n} \text{ for $x \in M_n^-$}.
\end{equation}
Moreover, the function $F$ satisfies the coboundary
equation~\eqref{eq:coboundary}, as $F^+$ and $F^-$ satisfy it respectively on
$M^+$ and $M^-$ by Lemma~\ref{lem:GL_coboundary}.

Let us show that $F$ is continuous on $M-\Sigma$. Take $x \in M-\Sigma$, for
instance in $M^+$. Let $\delta>0$. Let $n$ be large. The function $F_n^+$ is
well defined and continuous on a neighborhood of $x$. In particular, it
oscillates by at most $\delta$ on a neighborhood of $x$. As $F$ differs from
$F_n^+$ by $C/\lambda^n$, we deduce that $F$ oscillates by at most
$\delta+C/\lambda^n$ on a neighborhood of $x$. This proves the continuity of
$F$ at $x$.

Finally, let us show that $F$ extends continuously to $\Sigma$. It suffices
to show that it is uniformly continuous on $M-\Sigma$. For this, it suffices
to show that it is uniformly continuous on small horizontal segments close to
a singularity, as uniform continuity along vertical segments follows from the
coboundary equation. Let $(I_t)_{t\in (0,\delta]}$ be a family of vertical
translates of horizontal segments such that $I_0$ contains a singularity. For
$x, y\in I_0$, we have $F(x) - F(y) = F(g_t x)-F(g_t y) + \int_0^t (f(g_s x)
- f(g_s y)) \dd s$. Thanks to the boundedness of $L_h f$, the last integral
is small if $x$ and $y$ are close and $t$ is small, while the first
difference is small if $x$ and $y$ are close enough thanks to the continuity
of $F$ on $I_t$. Hence, $F(x)-F(y)$ itself is small. This concludes the
proof.
\end{proof}

To get further smoothness results, one needs to assume more cancellations for
$f$. The next theorem gives such conditions ensuring that $F$ is $C^1$.
\begin{thm}
\label{thm:C1_coboundary} Under the assumptions of
Theorem~\ref{thm:continuous_coboundary}, assume additionally that $f \in
C^3_h$. Assume moreover that, for any $\omega \in \bigcup_{\abs{\alpha} \geq
\lambda^{-2}} E_\alpha \cap \ker L_v$, one has $\langle \omega, f \rangle =
0$. Then the function $F$ solving the cohomological
equation~\eqref{eq:coboundary} is $C^1$ along the horizontal direction, and
$L_h F$ extends continuously to $M$.
\end{thm}
The assumptions of the theorem make sense by~\eqref{eq:makes_sense}. The
distributions appearing in the statement of the theorem have been completely
classified in Theorem~\ref{thm:main_preserves_orientations} and its proof.

Let us start with a preliminary reduction.
\begin{lem}
\label{lem:stronger_assumptions} To prove Theorem~\ref{thm:C1_coboundary}, it
is sufficient to prove it assuming the stronger condition that $\langle
\omega, f \rangle = 0$ for all $\omega \in \bigcup_{\abs{\alpha} \geq
\lambda^{-2}} E_\alpha$.
\end{lem}
The difference with the assumptions in Theorem~\ref{thm:C1_coboundary} is
that our new assumption is not restricted only to the vertically invariant
distributions.
\begin{proof}
Consider a function $f \in C^3_h$ such that $\langle \omega, f \rangle = 0$
for all $\omega \in \bigcup_{\abs{\alpha} \geq \lambda^{-2}} E_\alpha \cap
\ker L_v$. We can not deduce from the assumptions of the lemma that $f$ is a
smooth coboundary, as there might exist distributions $\omega \in E_\alpha -
\ker L_v$ with $\langle \omega, f \rangle \neq 0$. We will bring these
quantities back to $0$ by subtracting from $f$ a suitable coboundary. The
additional distributions we have to handle belong to $E_{\lambda^{-2}\mu_i}$
for some $\mu_i$ with $\abs{\mu_i} \in [1, \lambda)$. Denote by $F_i$ a
subspace of $E_{\lambda^{-2}\mu_i}$ sent isomorphically by $L_v$ to
$E_{\lambda^{-1}\mu_i}$. Then $E_{\lambda^{-2}\mu_i} = F_i \oplus
(E_{\lambda^{-2}\mu_i} \cap \ker L_v)$, see~\eqref{eq:flag_decomposition}.

Consider on $\bigoplus_{\abs{\mu_i} \in [1, \lambda)} E_{\lambda^{-1}\mu_i}$
the linear form $\omega \mapsto \langle L_v^{-1}\omega, f\rangle$, where by
$L_v^{-1}\omega$ we mean the unique $\tilde \omega \in \bigoplus F_i$ with
$L_v \tilde \omega = \omega$. As $\boB^{-k_h, k_v}$ is a space of
distributions, any linear form on a finite-dimensional subspace can be
realized by a smooth function. Hence, there exists $g_0 \in
C^\infty_c(M-\Sigma)$ such that, for any $\omega \in \bigoplus_{\abs{\mu_i}
\in [1, \lambda)} E_{\lambda^{-1}\mu_i}$, then $\langle L_v^{-1}\omega,
f\rangle = \langle \omega, g_0\rangle$. Hence, for $\tilde \omega \in
\bigoplus F_i$, applying the previous equality to $\omega = L_v\tilde\omega$,
we have
\begin{equation*}
  \langle \tilde \omega, f \rangle =
  \langle L_v \tilde \omega, g_0 \rangle = -\langle \tilde \omega, L_v g_0 \rangle.
\end{equation*}
This shows that the function $\tilde f = f + L_v g_0$ vanishes against any
distribution in $\bigoplus F_i$. It also vanishes against any distribution on
$\bigcup_{\abs{\alpha}\geq \lambda^{-2}} E_\alpha \cap \ker L_v$, as this is
the case of $f$ by assumption, and of $L_v g_0$. Hence, it vanishes against
all distributions in $\bigcup_{\abs{\alpha} \geq \lambda^{-2}} E_\alpha$.
Under the assumptions of the lemma, it follows that $f+L_v g_0$ can be
written as $L_v F$ for some function $F \in C^1_h$. Then $f= L_v(F-g_0)$,
concluding the proof.
\end{proof}

From this point on, we will assume that $f$ satisfies the strengthened
assumptions of Lemma~\ref{lem:stronger_assumptions}. To prove the theorem, we
start with a stronger version of Lemma~\ref{lem:GL_criterion_satisfied}.
\begin{lem}
\label{lem:GL_criterion_satisfied2} Under the assumptions of
Lemma~\ref{lem:stronger_assumptions}, there exists $\epsilon>0$ such that the
inequality $\abs*{\int_{t=0}^\tau \phi(t/\tau) f(g_t x) \dd t} \leq
C\norm{\phi}_{C^3}/\tau^{1+\epsilon}$ in~\eqref{eq:ziuerpoiuprt} holds, with
$r=3$.
\end{lem}
\begin{proof}
The proof is the same as for Lemma~\ref{lem:GL_criterion_satisfied2}, with
the difference that the additional vanishing conditions in
Lemma~\ref{lem:stronger_assumptions} give more vanishing terms in the
spectral decomposition of $f$, and thus a faster decay of $\check \boT^n f$.
\end{proof}

Let us now prove that the function $F$ given by
Theorem~\ref{thm:continuous_coboundary} is Lipschitz along horizontal
segments. This is the main step of the proof.

\begin{lem}
\label{lem:Fh_Lipschitz} Under the assumptions of
Lemma~\ref{lem:stronger_assumptions}, there exists $C$ such that, for any
points $x,y$ on the same horizontal segment, one has $\abs{F(x)-F(y)} \leq C
d(x,y)$.
\end{lem}
\begin{proof}
It suffices to prove the result for nearby points. Let $\delta>0$ be such
that any horizontal segment of size $\leq \delta$ can be completed above or
below to form a rectangle of vertical size $1$, not containing any
singularity. We will show the statement when $d = d(x,y)$ belongs to $(0,
\delta/\lambda)$.

Let $n \geq 1$ be the integer such that $\lambda^n d \in
(\delta/\lambda,\delta]$. Let $I$ be the horizontal interval between $x$ and
$y$. Assume for instance that $T^n I$ (which is of length $\leq \delta$) can
be completed above by a rectangle of height $1$ (otherwise, it can be
completed below, and the argument is the same but using $F_n^-$ instead of
$F_n^+$). In particular, there is no singularity in the rectangle of height
$\lambda^n$ above $I$. Note first that
\begin{equation*}
  \abs*{F^+_0(x)-F^+_0(y)} = \abs*{\int_{t=0}^1 \chi(t) f(g_t x) - f(g_t y) \dd t}.
\end{equation*}
As $L_h f$ is bounded by assumption and $g_t x$ and $g_t y$ are at distance
$d$ along a horizontal segment, we get
\begin{equation}
\label{eq:F0+woipuxcv}
  \abs*{F^+_0(x) - F^+_0(y)} \leq C d.
\end{equation}
Next, for $0<k\leq n$, we have $F^+_k(x)-F^+_{k-1}(x) =
\int_{t=0}^{\lambda^k} \phi(t/\lambda^k) f(g_t x) \dd t$ where $\phi(t) =
\chi(t)-\chi(\lambda t)$. Taking the difference, we get
\begin{align*}
  (F^+_k(x)-F^+_{k-1}(x)) - (F^+_k(y)-F^+_{k-1}(y))
  &= \int_{t=0}^{\lambda^k} \phi(t/\lambda^k) (f(g_t x) - f(g_t y)) \dd t
  \\& = \lambda^k \int_{s=0}^1 \phi(s) (\check \boT^k f(g_s x_k) - \check \boT^k f(g_s y_k)) \dd s,
\end{align*}
for $x_k = T^k x$ and $y_k = T^k y$, as in~\eqref{eq:wpicvpowixcv}. Since the
points $g_s x_k$ and $g_s y_k$ are on the same horizontal segment of length
$\lambda^k d$, we can integrate by parts and get
\begin{equation*}
  (F^+_k(x)-F^+_{k-1}(x)) - (F^+_k(y)-F^+_{k-1}(y))
  = \lambda^k \int_{u = y_k}^{x_k} \pare*{\int_{s=0}^1 \phi(s) L_h \check \boT^k f(g_s u) \dd s} \dd u.
\end{equation*}
Each integral over $s$ is an integral over a vertical segment, against a
smooth function $\phi$. By the definition of $\check\boB$, it is bounded by
$C\norm{\phi}_{C^3}\norm{\check\boT^k f}_{\check \boB^{3, -3}} $. Moreover,
the vanishing conditions on $f$ in the assumptions of
Theorem~\ref{thm:C1_coboundary} ensure that $\norm{\check \boT^k f}_{\check
\boB^{3, -3}}$ decays like $C\lambda^{-(2+\epsilon)k}$ for some $\epsilon>0$.
We get
\begin{align*}
  \abs*{(F^+_k(x)-F^+_{k-1}(x)) - (F^+_k(y)-F^+_{k-1}(y))}
  &\leq C \lambda^k \abs{x_k - y_k} \lambda^{-(2+\epsilon)k}
  = C \lambda^k \cdot \lambda^k d \cdot \lambda^{-(2+\epsilon)k}
  \\ & = C d \lambda^{-\epsilon k}.
\end{align*}
As the geometric series $\lambda^{-\epsilon k}$ is summable, we get starting
from~\eqref{eq:F0+woipuxcv} and summing over $k$ from $1$ to $n$ the
inequality
\begin{equation}
\label{eq:uiowoiuwxcvuio}
  \abs*{F^+_n(x) - F_n^+(y)} \leq C d.
\end{equation}
Moreover, by~\eqref{eq:Fn+F} (but with $\epsilon$ replaced by $1+\epsilon$
thanks to Lemma~\ref{lem:GL_criterion_satisfied2}), we have
\begin{equation*}
  \abs*{F_n^+(x) - F(x)} \leq C/\lambda^{(1+\epsilon) n}
  \leq C \lambda^{-n} \leq C (\lambda d/\delta),
\end{equation*}
thanks to the inequality $\lambda^n d \geq \delta/\lambda$. This is bounded
by $Cd$. In the same way, $\abs*{F_n^+(y) - F(y)} \leq C d$. Together
with~\eqref{eq:uiowoiuwxcvuio}, this gives $\abs*{F(x)-F(y)} \leq C d$.
\end{proof}
\begin{rmk}
\label{rmk:Holder} Under the weaker assumptions of
Theorem~\ref{thm:continuous_coboundary}, then the same proof goes through to
prove that $\abs{F(x)-F(y)} \leq C d(x,y)^\epsilon$, where $\epsilon$ comes
from Lemma~\ref{lem:GL_criterion_satisfied}. Hence, the
solution $F$ to the cohomological equation is automatically H\"older
continuous, without any further assumption. This corresponds in a different
setting to the main result of~\cite{marmi_yoccoz_Holder}.
\end{rmk}

\begin{proof}[Proof of Theorem~\ref{thm:C1_coboundary}]
Consider a function $f$ satisfying the assumptions of
Lemma~\ref{lem:stronger_assumptions}. We have to show that it is a $C^1$
coboundary. Let $F$ be the solution to the coboundary equation given by
Theorem~\ref{thm:continuous_coboundary}. By Lemma~\ref{lem:Fh_Lipschitz},
along any horizontal segment, it is differentiable almost everywhere, and
equal to the primitive of its derivative. We get a bounded measurable
function $F_h$ such that, for every horizontal interval $I$, for every $x,
y\in I$, one has
\begin{equation}
\label{eq:wuipxcvylnkmuy}
  F(y)-F(x) = \int_x^y F_h(u) \dd u.
\end{equation}
The difficulty is that we do not know if $F_h$ is continuous and well defined
everywhere.

The function $L_h f$ belongs to $C^2_h$. Moreover, it satisfies $\langle
\omega, L_h f\rangle = 0$ for $\omega \in \bigcup_{\abs{\alpha} \geq
\lambda^{-1}} E_\alpha$, as this is equal to $-\langle L_h \omega, f\rangle$,
which vanishes under the assumptions of Lemma~\ref{lem:stronger_assumptions}
as $L_h \omega \in \bigcup_{\abs{\alpha} \geq \lambda^{-2}} E_\alpha$. It
follows that $L_h f$ satisfies all the assumptions of
Theorem~\ref{thm:continuous_coboundary}. Hence, there exists a continuous
function $G$ on $M$ such that $\int_0^\tau L_h f(g_t x) = G(x)-G(g\tau x)$
for all $x$ and $\tau$.

Consider two points $x$ and $y$ on a small horizontal interval, and $\tau>0$
so that there is no singularity between the orbits $(g_s x)_{s\leq \tau}$ and
$(g_s y)_{s\leq \tau}$. Then one can compute
\begin{align*}
  \int_{u=x}^y (G-F_h)(u) - &(G-F_h)(g_\tau u) \dd u
  \\&= \int_{u=x}^y \int_0^\tau L_h f(g_t u) \dd t \dd u - (F(y)-F(x)) + (F(g_\tau y) - F(g_\tau x))
  \\& = \int_0^\tau f(g_t y) - f(g_t x) \dd t - (F(y)-F(x)) + (F(g_\tau y) - F(g_\tau x))
  = 0.
\end{align*}
Since this also holds along any subsegment $[x',y']$ of $[x,y]$, it follows
that $(G-F_h)(u) - (G-F_h)(g_\tau u)$ vanishes almost everywhere on the
segment $[x,y]$. One deduces that, for almost every $\tau \geq 0$ and almost
every $u\in M$, one has $(G-F_h)(g_\tau u) = (G-F_h)(u)$. By ergodicity of
the vertical flow, it follows that $G-F_h$ is almost everywhere constant, and
we can even assume that this constant vanishes by subtracting it from $G$ if
necessary.

By Fubini, for almost every horizontal interval $I$ one has $F_h = G$ almost
everywhere on $I$. On such an interval, we deduce
from~\eqref{eq:wuipxcvylnkmuy} the equality $F(y)-F(x) = \int_x^y G(u) \dd
u$. By continuity of $F$ and $G$, this equality extends to all horizontal
intervals. It follows from this formula that $F$ is differentiable in the
horizontal direction, with derivative $G$. As $G$ is continuous on $M$, this
concludes the proof of the theorem.
\end{proof}

The following theorem is the precise version of
Theorem~\ref{thm:coboundary_Ck_main} on $C^k$ solutions to the cohomological
equation.
\begin{thm}
\label{thm:coboundary_Ck} Under the assumptions of
Theorem~\ref{thm:continuous_coboundary}, assume additionally that $f \in
C^{k+2}_h$. Assume moreover that, for any $\omega \in \bigcup_{\abs{\alpha}
\geq \lambda^{-k-1}} E_\alpha \cap \ker L_v$, one has $\langle \omega, f
\rangle = 0$. Then the function $F$ solving the cohomological
equation~\eqref{eq:coboundary} is $C^k$ along the horizontal direction, and
$L^j_h F$ extends continuously to $M$ for all $j\leq k$.
\end{thm}
The assumptions of the theorem make sense by~\eqref{eq:makes_sense}. As
explained after that equation, the assumptions of the theorem could even be
weakened to $f \in C^{k+1+\epsilon}_h$. The loss of $1+\epsilon$ derivatives
corresponds in this setting to the result of Forni on the regularity loss in
the cohomological equation on almost every translation
surface~\cite{forni_regularity}. The conclusion can also be strengthened as
the $k$-th derivative is also H\"older continuous for some small exponent, see
Remark~\ref{rmk:Holder}.
\begin{proof}
We argue by induction on $k$, the cases $k=0$ and $k=1$ being true thanks to
Theorems~\ref{thm:continuous_coboundary} and~\ref{thm:C1_coboundary}. Assume
$k\geq 2$. By Theorem~\ref{thm:C1_coboundary}, there exists a function $F$
solving the cohomological equation for $f$, such that $L_h F$ is well defined
and continuous. Differentiating horizontally, one gets that $L_h F$ is a
continuous function, solving the cohomological equation for $L_h f$.

Moreover, the function $L_h f$ satisfies all the assumptions of the theorem
for the smoothness degree $k-1$. By the inductive assumption, there exists a
function $G$ solving the cohomological equation for $L_h f$, such that $L_h^i
G$ is well defined for $i \leq k-1$. The functions $G$ and $L_h F$ solve the
same cohomological equation. Hence, $G-L_h F$ is constant along orbits of the
vertical flow. As this flow is minimal, it follows that $G-L_h F$ is
constant. Therefore, $L_h F$ has $k-1$ continuous horizontal derivatives.
This concludes the proof.
\end{proof}

\section{When orientations are not preserved}
\label{sec:orientations}

\subsection{Orientable foliations whose orientations are not preserved}
\label{subsec:orientable_foliations}

Consider a translation surface $(M,\Sigma)$, and a linear pseudo-Anosov map
$T$ on $M$ which does not necessarily preserve the orientations of the
horizontal and vertical foliations. There are two global signs $\epsilon_h$
and $\epsilon_v$ indicating respectively if $T$ preserves the orientations of
the horizontal and the vertical foliations. Then the spectrum of $T^*$ on
$H^1(M)$ is given by $\epsilon_h \lambda$, by $\epsilon_v \lambda^{-1}$, and
by $\Xi = \{\mu_1,\dotsc, \mu_{2g-2}\}$ with $\abs{\mu_i} \in (\lambda^{-1},
\lambda)$ (where this last property follows from the same result for the map
$T^2$, which preserves orientations). One can describe the Ruelle spectrum
exactly as we did in the orientations preserving case, with the only
difference that the commutation relations between the composition operator
$\boT$ and the horizontal and vertical derivatives are not the same:
Proposition~\ref{prop:action_Lv} should be replaced by the equalities
\begin{equation*}
  \boT \circ L_v = \epsilon_v \lambda L_v \circ \boT,
  \quad
  \boT \circ L_h = \epsilon_h \lambda^{-1} L_h \circ \boT
\end{equation*}
on appropriate spaces. On the other hand, the definition of the Banach spaces
$\boB^{-k_h, k_v}$ need not be changed (their very definition in
Section~\ref{sec:def_boB} is independent of the existence of a pseudo-Anosov
map on the surface).

The largest eigenvalues of $\boT$, in addition to $1$, are given by
$\epsilon_h \lambda^{-1} \mu_i$. Then, to build new eigenfunctions from such
an eigenfunction, one can either differentiate in the horizontal direction,
or integrate in the vertical direction. When $\epsilon_h \neq \epsilon_v$,
this gives rise to two different eigenvalues, while when they coincide one
obtains the same eigenvalue again. In general, choosing to apply $k-1$
horizontal derivatives and $\ell$ vertical integrations (with $k \geq 1$ and
$\ell \geq 0$) gives an eigenfunction for the eigenvalue $\epsilon_h^k
\epsilon_v^\ell \lambda^{-k-\ell}\mu_i$. Hence, one obtains the following
description of the spectrum:

\begin{thm}
\label{thm:orientable_not_preserved_foliations} Let $T$ be a linear
pseudo-Anosov map on a translation surface of genus $g$, with orientable
horizontal and vertical foliations. Denoting by $\lambda
> 1$ its expansion factor, then the spectrum of $T^*$ on $H^1(M)$ has the
form $\{\epsilon_h \lambda, \epsilon_v
\lambda^{-1},\mu_1,\dotsc,\mu_{2g-2}\}$ with $\abs{\mu_i} \in
(\lambda^{-1},\lambda)$ for all $i=1,\dotsc, 2g-2$. Then $T$ has a Ruelle
spectrum on $\boC = C^\infty_c(M-\Sigma)$, given (with multiplicities) by
\begin{equation*}
  \{1\} \cup \bigcup_{i=1}^{2g-2} \bigcup_{k \geq 1} \bigcup_{\ell \geq 0} \{\epsilon_h^k \epsilon_v^\ell \lambda^{-k-\ell}\mu_i\}.
\end{equation*}
\end{thm}
For $\epsilon_h=\epsilon_v=1$, one recovers
Theorem~\ref{thm:main_preserves_orientations}.

One can also obtain a full description of the vertically invariant
distributions, and solve the cohomological equation for the vertical flow.
However, the simplest way to do this is certainly to apply the results of the
previous sections to the map $T^2$, which preserves orientations, so we will
not discuss these results any further.

It is more interesting to check that the trace formula of
Theorem~\ref{thm:trace_formula}  still holds in this more general context.
\begin{thm}
\label{thm:trace_orientable} Let $T$ be a linear pseudo-Anosov map on a
compact surface with orientable horizontal and vertical foliations. Then, for
all $n$,
\begin{equation}
\label{eq:sdfuiolkjxcvuiopaa}
  \flattr(\boT^n) = \sum_\alpha d_\alpha \alpha^n,
\end{equation}
where the sum is over all Ruelle resonances $\alpha$ of $T$, and $d_\alpha$
denotes the multiplicity of $\alpha$.
\end{thm}
\begin{proof}
We follow the proof of Theorem~\ref{thm:trace_formula}, with appropriate
modifications. The Lefschetz fixed-point formula gives
\begin{align*}
  \sum_{T^n x = x} \ind_{T^n} x
  &= \tr((T^n)_{\restr H^0(M)}^*) - \tr((T^n)_{\restr H^1(M)}^*) + \tr((T^n)_{\restr H^2(M)}^*)
  \\& = 1 - \pare*{\epsilon_h^n\lambda^n + \epsilon_v^n \lambda^{-n} + \sum_{i=1}^{2g-2} \mu_i^n} + \epsilon_h^n \epsilon_v^n,
\end{align*}
where $\{\mu_1,\dotsc, \mu_{2g-2}\}$ denote the eigenvalues of $T^*$ on the
subspace of $H^1(M)$ orthogonal to $[\dd x]$ and $[\dd y]$, as in the
statement of Theorem~\ref{thm:main_preserves_orientations}. The last term
$\epsilon_h^n \epsilon_v^n$ is equal to $1$ if $T^n$ preserves orientation,
$-1$ if it reverses orientation.

We can also compute the right hand side of~\eqref{eq:sdfuiolkjxcvuiopaa},
using the description of Ruelle resonances: By
Theorem~\ref{thm:orientable_not_preserved_foliations}, $\sum d_\alpha
\alpha^n$ is given by
\begin{align*}
  & 1 + \sum_{i=1}^{2g-2} \sum_{k=1}^\infty \sum_{\ell=0}^\infty
  (\epsilon_h^k \lambda^{-k})^n (\epsilon_v^\ell \lambda^{-\ell})^n \mu_i^n
  = 1 + \sum_{i=1}^{2g-2} \frac{ \epsilon_h^n \lambda^{-n}}{1-\epsilon_h^n \lambda^{-n}} \cdot \frac{1}{1-\epsilon_v^n \lambda^{-n}} \cdot \mu_i^n
  \\&
  = 1 - \sum_{i=1}^{2g-2} \frac{\mu_i^n}{(1-\epsilon_h^n \lambda^n) \cdot (1-\epsilon_v^n \lambda^{-n})}
  = \frac{(1-\epsilon_h^n \lambda^n) \cdot (1-\epsilon_v^n \lambda^{-n}) - \sum_{i=1}^{2g-2} \mu_i^n}{(1-\epsilon_h^n \lambda^n) \cdot (1-\epsilon_v^n
  \lambda^{-n})}
  \\&
  = \frac{1 - \pare*{\epsilon_h^n \lambda^n + \epsilon_v^n \lambda^{-n} + \sum_{i=1}^{2g-2} \mu_i^n} + \epsilon_h^n \epsilon_v^n} {(1-\epsilon_h^n \lambda^n) \cdot
  (1-\epsilon_v^n \lambda^{-n})}.
\end{align*}
Combining the two formulas with the definition of the flat trace, we get the
conclusion of the theorem.
\end{proof}

\subsection{Non-orientable foliations}

Consider a pseudo-Anosov map $T$ on a half-translation surface $M$, but such
that the horizontal and vertical foliations are not orientable. Note that,
with our Definition~\ref{def:half_translation}, a half-translation surface is
always orientable as $x \mapsto -x$ preserves orientation in $\R^2$. Hence,
if the horizontal foliation is not orientable, then neither is the vertical
foliation, and conversely. In this case, one can not argue directly in $M$ as
the differentiation operators $L_h$ and $L_v$ do not make sense anymore:
there is a sign ambiguity regarding the direction of differentiation. (On the
other hand, the squares $L_h^2$ and $L_v^2$ of these operators are well
defined.)

Let $\bar M$ be the two fold orientation (ramified) covering of $M$: away
from singularities, an element of $\bar M$ is a pair $(x, v)$ where $x \in
M-\Sigma$ and $v$ is an orientation of the horizontal foliation at $x$
(equivalently, it is a horizontal unit-norm vector). Let $\bar \pi : \bar M
\to M$ be the covering projection, and write $\bar \Sigma = \bar
\pi^{-1}(\Sigma)$. Then $(\bar M, \bar \Sigma)$ is a translation surface.
Let $i: \bar M \to \bar M$ be the involution $i(x,v) = (x, -v)$. It is a
homeomorphism of $\bar M$.

$T$ lifts to two pseudo-Anosov maps $\bar T$ and $i\circ \bar T$ of $\bar M$ and
the homeomorphism $i$ commutes with $\bar T$. Let us consider $\epsilon_h,\ \epsilon_v$
where $\epsilon_h, \epsilon_v \in \{\pm 1\}$ indicate whether $\bar T$ fixes or reverses
the orientation in the horizontal (resp.\ vertical) direction, as in Paragraph~\ref{subsec:orientable_foliations}.
Obviously the corresponding pair associated to the other lift $i\circ \bar T$ is $(-\epsilon_h,-\epsilon_v)$.

The action of $i^\ast$ gives rise to a splitting of $H^1(\bar M)$ as the
direct sum of the two subspaces $H^1_\pm(\bar M) = \{ h \in H^1(\bar M) \st
i^* h = \pm h\}$. The invariant part $H^1_+(\bar M)$ corresponds to classes
that are lifts of classes in $H^1(M)$. On the other hand, $[\dd x]$ and $[\dd
y]$ belong to the anti-invariant part. If $f$ is a function on $M$, then
$f\circ \pi \cdot \dd x$ if also anti-invariant.

The spectrum of $\bar T^*$ on $H^1_+(\bar M)$ is equal to the spectrum of $T$
on $H^1(M)$, given by $2g$ eigenvalues that we denote by $\mu_1^+,\dotsc,
\mu_{2g}^+$. Let us denote the spectrum of $\bar T^*$ on $H^1_-(\bar M)$ by
$\epsilon_h \lambda$, $\epsilon_v \lambda^{-1}$ and $\mu_1^-,\dotsc,
\mu_{2g_- -2}^-$. The
Ruelle spectrum of $\bar T$ is expressed in terms of all these data as in
Theorem~\ref{thm:orientable_not_preserved_foliations}, but the Ruelle
spectrum of $T$ is a strict subset of the Ruelle spectrum of $\bar T$ as one
should only consider those distributions in the spectrum that do not vanish
on functions coming from the basis.

\begin{thm}
\label{thm:ruelle_spectrum_nonorientable} In this setting, $T$ has a Ruelle
spectrum on $\boC = C^\infty_c(M-\Sigma)$, given (with multiplicities) by
\begin{equation*}
  \{1\} \cup
  \bigcup_{i=1}^{2g} \bigcup_{\substack{k \geq 1, \ell \geq 0\\k + \ell \text{ even}}}
     \{\epsilon_h^k \epsilon_v^\ell \lambda^{-k-\ell}\mu_i^+\}
  \cup
  \bigcup_{i=1}^{2g_- -2} \bigcup_{\substack{k \geq 1, \ell \geq 0\\k + \ell \text{ odd}}}
     \{\epsilon_h^k \epsilon_v^\ell \lambda^{-k-\ell}\mu_i^-\}.
\end{equation*}
\end{thm}
It is remarkable that, in this theorem only mentioning the correlations of
functions in $M$, all the eigenvalues of $\bar T^*$ appear: both the
invariant and anti-invariant parts of the cohomology can be read off the
correlations of functions in $M$.

This statement does not depend on the choice of the lift of $T$. Indeed, if
one chooses the other lift $i\circ \bar T$ of $T$, then the $\mu_i^+$ do not
change, but $\epsilon_h$, $\epsilon_v$ and $\mu_i^-$ are replaced by their
opposites, so that the above spectrum is not modified.
\begin{proof}
Among the distributions constructed in the proof of
Theorem~\ref{thm:orientable_not_preserved_foliations}, one should understand
which are orthogonal to functions from the basis, and which come from the
basis. First, for the cohomology classes, one writes them as $h = f \dd x$
for some $f$ in the Banach space $\boB^{-k_h, k_v}$. As $\dd x$ is
anti-invariant, it follows that $f$ is invariant if and only if $h$ is
anti-invariant. Hence, the eigenvalues $\mu_i^-$ give rise to distributions
coming from the base, for the eigenvalue $\epsilon_h \lambda^{-1} \mu_i^-$.
On the other hand, the eigendistributions for $\epsilon_h
\lambda^{-1}\mu_i^+$ are anti-invariant, and do not appear in the Ruelle
spectrum of $T$. Then, in $\bar M$, differentiating with respect to $L_h$ or
integrating with respect to $L_v$ exchanges the invariant and anti-invariant
subspaces. The full description of the spectrum follows.
\end{proof}

In this context, the trace formula of Theorem~\ref{thm:trace_formula} still
holds.
\begin{thm}
Let $T$ be a linear pseudo-Anosov map. Then, for all $n$,
\begin{equation}
\label{eq:wsldfjknrt,nwfkjh}
  \flattr(\boT^n) = \sum_\alpha d_\alpha \alpha^n,
\end{equation}
where the sum is over all Ruelle resonances $\alpha$ of $T$, and $d_\alpha$
denotes the multiplicity of $\alpha$.
\end{thm}
\begin{proof}
We have already proved this result when the foliations are orientable, in
Theorem~\ref{thm:trace_orientable}. Hence, we can assume that the foliations
are not orientable. In this case, the Ruelle spectrum is given in
Theorem~\ref{thm:ruelle_spectrum_nonorientable}.

Let $x$ be a fixed point of $T^n$. Denote by $x_1$ and $x_2$ its two lifts.
They are either fixed or exchanged by $\bar T^n$. We say that $x$ is
positively fixed if its lifts are fixed by $\bar T^n$, and negatively fixed
if they are exchanged by $\bar T^n$, i.e., fixed by $i \circ \bar T^n$. Let
$\Fix^+(T^n)$ and $\Fix^-(T^n)$ denote respectively the set of positively and
negatively fixed points of $T^n$. Around a positively fixed point, the local
picture of $T^n$ is the same as the local picture of $\bar T^n$ around the
lifts. In particular, $\det(I-DT^n)$ is equal to $(1-\epsilon_h^n \lambda^n)
(1-\epsilon_v^n \lambda^n)$. If $x$ is negatively fixed, on the other hand,
the local picture of $T^n$ is the same as that of $i \circ \bar T^n$, hence
locally $\det(I-DT^n) = (1+\epsilon_h^n \lambda^n) (1+\epsilon_v^n
\lambda^n)$. With the definition of the flat trace, we get
\begin{equation}
\label{eq:wcvlmijmkljwxcvwcx}
  \flattr(\boT^n) =
  \sum_{x \in \Fix^+(T^n)} \frac{ \ind_{T^n} x}{(1-\epsilon_h^n \lambda^n) (1-\epsilon_v^n \lambda^{-n})}
  + \sum_{x \in \Fix^-(T^n)} \frac{ \ind_{T^n} x}{(1+\epsilon_h^n \lambda^n) (1+\epsilon_v^n \lambda^{-n})}
\end{equation}
To proceed, we note that to one point in $\Fix^+(T^n)$ correspond two fixed
points of $\bar T^n$, with the same Lefschetz index. Therefore,
\begin{equation*}
  2\sum_{x \in \Fix^+(T^n)} \ind_{T^n}(x) = \sum_{\bar T^n y = y} \ind_{\bar T^n}(y).
\end{equation*}
We can apply Lefschetz index formula for $\bar T^n$ to the last sum, yielding
\begin{align*}
  2\sum_{x \in \Fix^+(T^n)} \ind_{T^n}(x) & = \tr((\bar T^n)_{\restr H^0(\bar M)}^*) - \tr((\bar T^n)_{\restr H^1(\bar M)}^*) + \tr((\bar T^n)_{\restr H^2(\bar M)}^*)
  \\& = 1 - \pare*{\epsilon_h^n \lambda^n + \epsilon_v^n \lambda^{-n} + \sum_{i=1}^{2g} (\mu_i^+)^n + \sum_{i=1}^{2g_--2} (\mu_i^-)^n } + \epsilon_h^n \epsilon_v^n
  \\& = (1-\epsilon_h^n \lambda^n) (1-\epsilon_v^n \lambda^{-n}) - \sum_{i=1}^{2g} (\mu_i^+)^n - \sum_{i=1}^{2g_--2} (\mu_i^-)^n.
\end{align*}
A point in $\Fix^-(T^n)$ corresponds to two fixed points of $i \circ \bar
T^n$. Applying the Lefschetz formula to $i\circ \bar T^n$, we get in the same
way
\begin{equation*}
  2\sum_{x \in \Fix^-(T^n)} \ind_{T^n}(x) = (1+\epsilon_h^n \lambda^n) (1+\epsilon_v^n \lambda^{-n}) - \sum_{i=1}^{2g} (\mu_i^+)^n + \sum_{i=1}^{2g_--2} (\mu_i^-)^n,
\end{equation*}
as the eigenvalues of $i\circ \bar T^n$ in cohomology are $-\epsilon_h^n
\lambda^n$, $-\epsilon_v^n \lambda^n$, $(\mu_i^+)^n$ and $-(\mu_i^-)^n$.
Combining these two formulas with~\eqref{eq:wcvlmijmkljwxcvwcx}, we obtain
\begin{equation}
\label{eq:wcxljmlwkxcvwxcv}
\begin{split}
  \flattr(\boT^n) = 1 & -\frac{1}{2} \sum (\mu_i^+)^n
    \pare*{\frac{1}{(1-\epsilon_h^n \lambda^n) (1-\epsilon_v^n \lambda^{-n})} + \frac{1}{(1+\epsilon_h^n \lambda^n) (1+\epsilon_v^n \lambda^{-n})}}
  \\ & - \frac{1}{2} \sum (\mu_i^-)^n
      \pare*{\frac{1}{(1-\epsilon_h^n \lambda^n) (1-\epsilon_v^n \lambda^{-n})} - \frac{1}{(1+\epsilon_h^n \lambda^n) (1+\epsilon_v^n \lambda^{-n})}}.
\end{split}
\end{equation}
Let us expand
\begin{align*}
  \frac{1}{(1-\epsilon_h^n \lambda^n) (1-\epsilon_v^n \lambda^{-n})}
  & = -\epsilon_h^n \lambda^{-n} \frac{1}{1-\epsilon_h^n \lambda^{-n}}\cdot \frac{1}{1-\epsilon_v^n \lambda^{-n}}
  \\& = -\epsilon_h^n \lambda^{-n} \pare*{\sum_{k\geq 0} (\epsilon_h^n \lambda^{-n})^k} \pare*{\sum_{k\geq 0} (\epsilon_v^n \lambda^{-n})^\ell}
  \\& = - \sum_{k\geq 1, \ell \geq 0} (\epsilon_h^k \epsilon_v^\ell \lambda^{-k-\ell})^n
\end{align*}
and analogously
\begin{equation*}
  \frac{1}{(1+\epsilon_h^n \lambda^n) (1+\epsilon_v^n \lambda^{-n})}
  = - \sum_{k\geq 1, \ell \geq 0} (-1)^{k+\ell} (\epsilon_h^k \epsilon_v^\ell \lambda^{-k-\ell})^n
\end{equation*}
Therefore, when one computes the terms in~\eqref{eq:wcxljmlwkxcvwxcv}, there
comes out a factor $(1+(-1)^{k+\ell})/2$ on the first line, which is $1$ when
$k+\ell$ is even and $0$ otherwise, and a factor $(1-(-1)^{k+\ell})/2$ on the
second line, which is $1$ when $k+\ell$ is odd and $0$ otherwise. We finally
get
\begin{equation*}
  \flattr(\boT^n) = 1 + \sum_{i=1}^{2g} \sum_{\substack{k \geq 1, \ell \geq 0\\k + \ell \text{ even}}} (\epsilon_h^k \epsilon_v^\ell \lambda^{-k-\ell}\mu_i^+)^n
  + \sum_{i=1}^{2g_--2} \sum_{\substack{k \geq 1, \ell \geq 0\\k + \ell \text{ odd}}} (\epsilon_h^k \epsilon_v^\ell \lambda^{-k-\ell}\mu_i^+)^n.
\end{equation*}
In view of the expression for the Ruelle spectrum given in
Theorem~\ref{thm:ruelle_spectrum_nonorientable}, this is the desired result.
\end{proof}

\bibliography{biblio}
\bibliographystyle{amsalpha}
\end{document}